\documentclass[10pt,letterpaper,reqno]{amsart}
\usepackage{color}
\usepackage[latin1]{inputenc}
\usepackage{amsmath}
\usepackage{amsfonts}
\usepackage{amsthm}
\usepackage{epsfig}
\usepackage{graphicx}
\usepackage{amssymb}
\usepackage{bm}
\usepackage{esint}
\usepackage{caption}
\usepackage{stmaryrd}

\usepackage[a4paper,bindingoffset=0.2in,%
left=1.2in,right=1.2in,top=1.2in,bottom=1.2in,%
footskip=.35in]{geometry}

\newcommand{\Lip}{{\rm Lip}}

\newcommand{\Id}{{\rm Id}\,}
\newcommand{\dist}{{\rm dist}}
\newcommand{\loc}{{\rm loc}}
\newcommand{\diam}{{\rm diam}\,}

\newcommand{\spt}{{\rm spt}}

\DeclareMathOperator{\restrict}{\llcorner}
\DeclareMathOperator{\Tan}{Tan}  
\newcommand{\Diff}{{\rm Diff}}

\newcommand{\Der}{D}

\DeclareMathOperator{\ap}{ap}
\newcommand{\der}{\mathbf{D}}

\DeclareMathOperator{\Nor}{Nor}
\DeclareMathOperator{\nor}{nor}
\DeclareMathOperator{\Cut}{Cut}

\DeclareMathOperator{\Unp}{Unp}

\theoremstyle{plain}
\newtheorem{theorem}{Theorem}[section]
\newtheorem{lemma}[theorem]{Lemma}
\newtheorem{corollary}[theorem]{Corollary}

\newtheorem*{theorem*}{Theorem}
\newtheorem*{corollary*}{Corollary}

\newtheorem{thm}{Theorem}

\theoremstyle{definition}
\newtheorem{definition}[theorem]{Definition}
\newtheorem{remark}[theorem]{Remark}

\numberwithin{equation}{section}
\numberwithin{figure}{section}

\DeclareRobustCommand{\rchi}{{\mathpalette\irchi\relax}}
\newcommand{\irchi}[2]{\raisebox{\depth}{$#1\chi$}}

\usepackage{xcolor}

\title{Alexandrov sphere theorems for $ W^{2,n} $-hypersurfaces}

\author{Mario Santilli}
\address{Dipartimento di Ingegneria e Scienze dell'Informazione e Matematica, Universit\'a degli Studi dell'Aquila, 67100 L'Aquila, Italy}
\email{mario.santilli@univaq.it}
\author{Paolo Valentini}
\address{Dipartimento di Ingegneria e Scienze dell'Informazione e Matematica, Universit\'a degli Studi dell'Aquila, 67100 L'Aquila, Italy}
\email{paolo.valentini@graduate.univaq.it}

\begin{document}

	\begin{abstract}
We prove that the proximal unit normal bundle of the subgraph of a $ W^{2,n} $-function carries a natural structure of Legendrian \emph{cycle}. This result is used to obtain an Alexandrov-type sphere theorem for hypersurfaces in $ \mathbf{R}^{n+1} $ which are locally  graphs of \emph{arbitrary} $ W^{2,n} $-functions. We also extend the classical umbilicality theorem to $ W^{2,1} $-graphs, under the Lusin (N) condition for the graph map.
	\end{abstract}

		\maketitle
	\tableofcontents
		\paragraph*{\small MSC-classes 2020:}{\small 53C24, 53C65, 49Q20.}
	
	\section{Introduction}
	
\subsection{Background and motivation}
It is important and natural to understand if classical results of smooth differential geometry still hold if one weakens the regularity hypothesis. Since most of the classical differential-geometric techniques rely on some smoothness assumption, such a  question often calls for substantial generalizations of the existing methods. 
	
Our starting point in this paper is the following general version of the sphere theorem by Alexandrov \cite{Aleksandrov}. 
\begin{theorem*}[Alexandrov]
A bounded and connected $ C^2 $-domain $ \Omega \subseteq \mathbf{R}^{n+1} $ must be a round ball, provided there exist a $ C^1 $ function $ \varphi : \mathbf{R}^n \rightarrow \mathbf{R} $ and $ \lambda \in \mathbf{R} $ such that
$$ \varphi(\rchi_{\Omega,1}(p), \ldots, \rchi_{\Omega,n}(p)) = \lambda $$ 
and  
\begin{equation}\label{intro uniform ellipticity}
	\frac{\partial \varphi}{\partial t_i}(\rchi_{\Omega, 1}(p), \ldots , \rchi_{\Omega, n}(p)) >0 \quad \textrm{for $ i = 1, \ldots , n $,}
\end{equation}  
for every  $ p \in \partial \Omega $. Here  $ \rchi_{\Omega,1} \leq \ldots \leq \rchi_{\Omega,n} $ are the principal curvatures of $ \partial \Omega $.
\end{theorem*}
 \noindent The simplest case of this result is the famous rigidity result for hypersurfaces with constant mean curvature. More generally,  choosing $ \varphi = \sigma_k $, where  $ \sigma_k $ is the $ k $-th elementary symmetric function (see Definition \ref{symmetric function}), one can deduce (see  \cite[Appendix]{KorevaarRos})  that if $ \Omega \subseteq \mathbf{R}^{n+1} $ is a bounded $ C^2 $-domain such that $ H_{\Omega, k} =  \sigma_k(\rchi_{\Omega,1}, \ldots, \rchi_{\Omega,n}) $ is constant for some $ k $, then $ \Omega $ is a round ball. This result was proved by Alexandrov using the moving plane method. A completely different approach to treat the special case $ \varphi = \sigma_k $ and based on integral identities was found by Ros in \cite{RosRevista} and Montiel-Ros \cite{MontielRos}.
		
		It is natural to ask about generalization of the sphere theorem beyond the classical smooth regime. This problem was addressed by Alexandrov in \cite{Aleksandrov62}, where he generalized the sphere theorem to bounded domains whose boundary can be locally represented as graph of $ C^1 $-functions with second-order distributional derivatives in $ L^n $, and under the \emph{uniform ellipticity condition}
		\begin{equation}\label{intro uniform ellipticity 1}
	0 < \mu_1 \leq	\frac{\partial \varphi}{\partial t_i}(\rchi_{\Omega, 1}(p), \ldots , \rchi_{\Omega, n}(p)) \leq \mu_2 < \infty \quad \textrm{for $ i = 1, \ldots , n $,}
		\end{equation}
for $ \mathcal{H}^n $ a.e.\  $ p \in \partial \Omega $. Here  $ \rchi_{\Omega,1} \leq \ldots \leq \rchi_{\Omega,n} $ are the weak principal curvatures of $ \partial \Omega $. Obviously, \eqref{intro uniform ellipticity 1} reduces to \eqref{intro uniform ellipticity} for $ C^2 $-domains. The proof of this result is based on the generalization of the moving plane method by means of maximum principles for $ W^{2,n} $-solutions of uniformly elliptic PDE's. Both the hypothesis of $ C^1 $-regularity and the \emph{uniform ellipticity condition} \eqref{intro uniform ellipticity 1} are important for the applicability of this method. On the other hand, it is natural to ask if these hypotheses are convenient conditions, rather than necessary restrictions. Additionally, arbitrary $ W^{2,n} $-functions exhibit very different behaviours than $ C^1 $-functions.  For instance,  T.\ Toro in \cite{Toro94} constructs a $ W^{2,n}$-function with a (countable) dense subset of singular points, and J.\ Fu in \cite{FuAlexandrov} points out the existence of $ W^{2,n} $ functions on $ \mathbf{R}^n $ whose gradient has a dense graph in $ \mathbf{R}^n \times \mathbf{R}^n $. 

With these motivations in mind, in this paper we generalize Alexandrov sphere theorem to \emph{arbitrary $ W^{2,n} $-domains} (i.e.\ open sets which are locally subgraphs of $ W^{2,n} $-functions) when $ \varphi $ is a symmetric function of the weak principal curvatures. Moreover, we prove our result under a more general hypothesis than uniform ellipticity, namely the \emph{degenerate ellipticity condition} \eqref{intro degenerate ellipticity}, cf.\ Theorem \ref{intro sphere theorem}.
Instead of using a moving plane method,  we extend the Montiel-Ros integral-geometric approach to prove our result. In recent years Montiel-Ros argument has been generalized to some classes of non-smooth  geometric sets, namely  sets of finite perimeter with \emph{bounded} distributional mean curvature (see \cite{DelgadinoMaggi} and \cite{DeRosaetall}) and sets of positive reach  (cf.\ \cite{HugSantilli}). On the other hand, the aforementioned examples show that $ W^{2,n}  $-domains exhibit some very different behaviour than the sets treated in \cite{DelgadinoMaggi}, \cite{DeRosaetall}, or in  \cite{HugSantilli}. As explained below, this requires a substantially novel approach.


\subsection{Legendrian cycles and sphere theorem}

Here we discuss the generalization of the Montiel-Ros method (see \cite{MontielRos}) to $ W^{2,n} $-domains. In the smooth setting this method is based on a clever combination of the Heintze-Karcher inequality (see \cite{HeintzeKarcher} and \cite{MontielRos}) with  the variational formulae for the higher-order mean curvature integrals of a $ C^2 $-domain (cf.\ \cite{Hsiung} and \cite{Reilly1972}). As noted by Fu (cf.\ \cite{Fu98}),  the variational formulae are strictly related with the  structure of Legendrian cycle carried by the unit-normal bundle of a smooth submanifold.  We recall that an integer multiplicity locally rectifiable $ n $-current $ T $ of $ \mathbf{R}^{n+1} \times \mathbf{R}^{n+1} $  with compact support in $ \mathbf{R}^{n+1}  \times \mathbf{S}^n $  is called \emph{Legendrian cycle} of $ \mathbf{R}^{n+1} $ if and only if $ \partial T = 0 $ and $ T \restrict \alpha  = 0 $, where $ \alpha $ is the contact form in $ \mathbf{R}^{n+1} \times \mathbf{R}^{n+1} $ (see section \ref{section: Notation and background}). Indeed, it is a simple exercise to prove that the unit-normal bundle of a $ C^2 $-domain carries a natural structure of Legendrian cycle. It is also well known that if $ f $ is a $ W^{2,n} $-function on an open subset $ U $ of $ \mathbf{R}^{n+1} $ then there exists an integral current $ \mathbb{D}(f) $ (also denoted by $ [df] $) of $ U \times \mathbf{R}^n $  with zero distributional boundary (i.e.\ an integral cycle), which serves as a substitute for the graph of the differential of $ f $ (see \cite{Fu89}, \cite[(4.1) pag.\ 332]{JerrardMongeAmpere} and \cite{FuAlexandrov}).  This current is called \emph{differential cycle of $ f $.} The construction of this integral cycle can be naturally extended to construct a Legendrian cycle associated with the graph of $ f $. However,  this information alone is not sufficient to extend the Montiel-Ros method, as such extension seems to require a crucial geometric property for the carrier $ W $ of this Legendrian cycle: namely that the segments $ a + t u $ with $ t \geq 0 $, at least for "many" points $(a, u) \in W$, must be distance-minimizing segments near $ 0 $. This observation leads us to consider the proximal unit-normal bundle, which is defined for an arbitrary set $ C \subseteq \mathbf{R}^{n+1} $ as
\begin{equation*}
	\nor(C) = \{(x,u) \in \overline{C} \times \mathbf{S}^n: \dist(x+su,C) = s \; \textrm{for some $ s > 0 $}\}.
\end{equation*} 
It is always true that $ \nor(C) $ is a \emph{Legendrian rectifiable set} of $ \mathbf{R}^{n+1}  \times \mathbf{S}^n $ (see Definition \ref{def legendrian set}), namely it can be $ \mathcal{H}^n $ almost everywhere covered by a countable union of $ n $-dimensional $ C^1 $-submanifold of $ \mathbf{R}^{n+1}  \times \mathbf{S}^n $ and the contact form $ \alpha $ vanishes on the approximate tangent plane of $ \nor(C) $ at $ \mathcal{H}^n $ almost all points; cf.\ Lemma \ref{lem: Santilli20}. On the other hand, it is \emph{not} always true that $ \mathcal{H}^n\restrict \nor(C) $ is a Radon measure over $ \mathbf{R}^{n+1} \times \mathbf{S}^n $ (for instance even when $ C $ coincides with the closure of a smooth submanifold with bounded mean curvature, cf.\ Lemma \ref{lem Brakke example}), henceforth, $ \nor(C) $ cannot always carry an integer-multiplicity rectifiable current.

Our first and central result asserts that the proximal unit normal bundle of a $ W^{2,n} $-domain  carries a natural structure of Legendrian cycle. More precisely, denoting by  $ \pi_0 : \mathbf{R}^{n+1} \times \mathbf{R}^{n+1} \rightarrow \mathbf{R}^{n+1} $ the projection onto the first factor, cf.\ \eqref{projections}, and by $ E' $ a volume form of $ \mathbf{R}^{n+1} $, cf.\ \eqref{n form and coform}, we prove the following result.
\begin{thm}[\protect{cf.\ Theorem \ref{W2n functions normal cycles} and Theorem \ref{W2n domains}}]\label{intro Theorem Legendrian}
	If $ \Omega \subseteq \mathbf{R}^{n+1} $ is a bounded $W^{2,n} $-domain then $ \mathcal{H}^n(\nor(\Omega)) < \infty $ and there exists a unique $ n $-dimensional Legendrian cycle $ T $ such that
	$$ T = (\mathcal{H}^n \restrict \nor(\Omega)) \wedge \eta,  $$
	where $ \eta $ is a $ \mathcal{H}^n \restrict \nor(\Omega) $ measurable $ n $-vectorfield such that
	$$ | \eta(x, u)| =1, \quad \textrm{$ \eta(x, u) $ is simple}, $$
	$$ \textrm{$\Tan^n(\mathcal{H}^n \restrict \nor(\Omega), (x, u)) $ is associated with $ \eta(x, u) $} $$
	and 
	$$ \langle \big[ {\textstyle\bigwedge_n} \pi_0\big](\eta(x,u)) \wedge u, E'\rangle > 0 $$
	for $ \mathcal{H}^n $ a.e.\ $(x,u) \in \nor(\Omega) $. We write $ T = N_\Omega $.
\end{thm}

A simple functional-analytic reformulation of Theorem \ref{intro Theorem Legendrian}, which can be proved along the same lines of Theorem \ref{W2n functions normal cycles}, states what follows.
 
 \begin{thm}\label{intro theorem Lagrangian}
 Suppose $ f \in W^{2,n}(U) $. Then $ \mathbb{D}(f) $ is a unit-density integral cycle carried over $\{(x, \nabla f(x)) : x \in \mathcal{S}(f) \} $, where $ \mathcal{S}(f)  $ is the set of pointwise \emph{twice}-differentiabily points of  $ f $ (cf.\ Definition \ref{def twice diff points}), and $ \nabla f $ is the pointwise gradient of $ f $.
 \end{thm}
 
\noindent In this direction we recall that $ \mathcal{L}^n(U \setminus \mathcal{S}(f)) =0 $ by Lemma \ref{W2n twice diff}, and the subtlety of Theorem \ref{intro theorem Lagrangian} becomes more transparent if we observe that $ \mathcal{S}(f) $ cannot be replaced by the set of pointwise differentiability points $ \Diff(f) $ of $ f $: indeed, there exists $ f \in C^1\big([-1,1]^n\big) \cap W^{2,n}\big((-1,1)^n\big) $ such that $$ [-1,1]^n \subseteq \nabla f([-1,1] \times \{0\}^{n-1}), $$ cf.\ \cite{Roskovec}, and denoting by $ \overline{\nabla f} $ the graph map of the gradient of $ f $, we use Lemma \ref{W2n functions Lusin} to conclude that $\mathcal{H}^n\big(\overline{\nabla f}(U) \setminus \overline{\nabla f}(\mathcal{S}(f))\big) > 0 $.



Combining Theorem \ref{intro Theorem Legendrian} with the variational formulae for Legendrian cycles in \cite{Fu98}, we can extend Reilly variational formulae (see \cite{Reilly1972}) to $ W^{2,n} $-domains, whence we deduce the Minkowski-Hsiung formulae in our setting (see Theorems \ref{Reilly variational formualae} and \ref{Minkowski formulae}). Moreover, the Heintze-Karcher inequality for $ W^{2,n} $-domains (see Theorem \ref{Heintze-Karcher}) can be deduced from the general inequality \cite[Theorem 3.20]{HugSantilli} employing some of the structural properties of the proximal unit-normal bundle (see Theorem \ref{W2n domains}\eqref{W2n domains 2}-\eqref{W2n domains 3}),  that already play a role in Theorem \ref{intro Theorem Legendrian}.
	
Combining these results we can eventually prove our generalization of Alexandrov sphere theorem.

\begin{thm}[\protect{cf.\ Theorem \ref{Alexandrov} and Remark \ref{Alexandrov rmk}}]\label{intro sphere theorem}
	A bounded and connected $ W^{2,n} $-domain $ \Omega \subseteq \mathbf{R}^{n+1} $ must be a round ball, provided there exist  $ k \in \{2, \ldots , n\} $ and  $ \lambda \in \mathbf{R} $ such that $$ \sigma_k(\rchi_{\Omega,1}(p), \ldots , \rchi_{\Omega, n}(p)) = \lambda $$
	and 
	\begin{equation}\label{intro degenerate ellipticity}
		\frac{\partial \sigma_k}{\partial t_i}(\rchi_{\Omega, 1}(p), \ldots , \rchi_{\Omega, n}(p)) \geq  0 \quad \textrm{for $ i = 1, \ldots , n $} 
	\end{equation}
	for $ \mathcal{H}^n $ a.e.\ $ p \in \partial \Omega $.
\end{thm}

\noindent If $ k = 1 $ the result would reduce to the smooth Alexandrov's sphere theorem for constant mean curvature hypersurfaces, since the condition $ H_{\Omega,1}(z) = \lambda $ for $ \mathcal{H}^n $ a.e.\ $ z \in \partial \Omega $ implies that $ \partial \Omega $ is smooth by Allard's regularity theorem (notice that by Theorem \ref{Reilly variational formualae} the function $ H_{\Omega,1} $ is the generalized mean curvature of $ \partial \Omega $ in the sense of varifolds, see \cite{Allard72}). No analogous regularity result is available when $ k \geq 2 $.

\subsection{The support of Legendrian cycles} Theorem \ref{intro Theorem Legendrian} finds natural application in other problems, beyond the rigidity questions that we have considered so far. In Section \ref{Section support} we employ it to answer a question implicit in \cite{RatajZaehle2015}. In \cite[Remark 2.3]{RatajZaehle2015} the authors asked if there exist $ n $-dimensional Legendrian cycles in $ \mathbf{R}^{n+1} $ whose support is not locally $ \mathcal{H}^n $-rectifiable, or even has positive $ \mathcal{H}^{n+1} $-measure. Combining Theorem \ref{intro Theorem Legendrian} with an observation by J.\ Fu in \cite{FuAlexandrov} about the existence of $ W^{2,n} $-functions whose differential has a graph dense in $ \mathbf{R}^n \times \mathbf{R}^n $, we prove the following result.
\begin{thm}
There exist $n $-dimensional Legendrian cycles of $ \mathbf{R}^{n+1} $ whose support has positive $ \mathcal{H}^{2n} $-measure.
\end{thm}

\subsection{The Nabelpunktsatz} In the final section of this paper we study the problem of extending the classical umbilicality theorem (or Nabelpunktsatz). The classical proof of this theorem works for hypersurfaces that are at least $ C^3 $-regular. A  proof for $ C^2 $-hypersurfaces is given in \cite{SouamToubiana2006} (see also \cite{Pauly} and \cite{MantegazzaUmbilical}). Considering more general hypersurfaces with curvatures defined only almost everywhere, the question about the validity of the Nabelpunktsatz goes back to the classical paper of Busemann and Feller \cite{BusemannFeller}, where they also pointed out the existence of non-spherical convex $ C^{1} $-hypersurfaces which are umbilical at almost every point; see also Remark \ref{rmk primitive cantor}. In \cite{DeRosaetall}  the Nabelpunksatz is extended to $ C^{1,1} $-hypersurfaces. Here we obtain the following far reaching generalization of this result.
\begin{thm}[cf.\ {\protect Theorem \ref{Nabelpunktsatz}}]
The Nabelpunktsatz holds for	almost everywhere umbilical $ W^{2,1} $-graphs,  provided the Lusin condition (N) holds for the graph function (cf. Definition \ref{Lusin condition})
\end{thm}

\noindent The Lusin condition (N) is necessary in order to guarantee the existence of weak curvatures on the graphs (i.e.\ second order rectifiability) and it is automatically verified for graphs of $ W^{2,p} $-functions, with $ p > \frac{n}{2} $; cf.\ Remark \ref{rmk Lusin condition (N) for f}.

\subsection{Acknowledgments} The first author is partially supported by INDAM-GNSAGA of the Instituto Nazionale di Alta Matematica "Francesco Severi", and by PRIN project no.\ 20225J97H5. The first author wishes to thank Slawomir Kolasinski for interesting discussions about the proof of Lemma \ref{lem exterior derivative Lipschitz-Killing}, and Ulrich Menne to have shared \cite[Appendix B]{menne2024sharplowerboundmean}, which plays an important role in this paper. 
	
	\section{Notation and background}\label{section: Notation and background}
	
	 Given a set of parameters $ \{p_1, p_2, \ldots p_n\}$, we denote a \emph{generic} positive constant depending only $ p_1, \ldots , p_n $  by $ c(p_1, \ldots , p_n) $. 
	
	If $ f : S \rightarrow T $ is a function we define 
	\begin{equation}\label{graph map}
		\overline{f}  : S \rightarrow S \times T, \qquad \overline{f}(x) = (x, f(x)). 
	\end{equation} 
The characteristic function of a set $ X $ is $\bm{1}_X$ and the Grassmannian of $ m $-dimensional subspaces of $ \mathbf{R}^k $ is $ \mathbf{G}(k,m) $.	Moreover, we often use the following projection maps 
	\begin{equation}\label{projections}
		\pi_0  : \mathbf{R}^{n+1} \times \mathbf{R}^{n+1} \rightarrow \mathbf{R}^{n+1} \quad \pi_1  : \mathbf{R}^{n+1} \times \mathbf{R}^{n+1} \rightarrow \mathbf{R}^{n+1} 
	\end{equation} 
	defined as $ \pi_0(x,u) = x $ and $ \pi_1(x,u) = u $.
	
	In this paper we use the symbol $ \bullet $ to denote \emph{scalar products}. In particular we fix a scalar product $ \bullet $ on $ \mathbf{R}^{n+1} $, 
	\begin{equation}\label{orthonormal basis}
		\textrm{an orthonormal basis $ e_1, \ldots , e_{n+1} $ of $ \mathbf{R}^{n+1} $ and its dual basis $ e'_1, \ldots, e'_{n+1}$.}
	\end{equation}
	 For a subset $ S $ of an Euclidean space, $ \overline{S} $ is the closure of $ S $. We use the symbols $ \Der $ and $ \nabla $ for the classical differential and the gradient.   If $ f : U \rightarrow \mathbf{R} $ is a continuous function defined on an open set $ U $, we denote the set of $ x \in U $ where $ f $ is pointwise differentiable by
	$$ \Diff(f). $$

	\subsection{Basic notions from geometric measure theory}
	
		In this paper we use standard notation from geometric measure theory, for which we refer to \cite{Fed69}. For reader's convenience we recall some basic notions here.

 For a subset $ X \subseteq  \mathbf{R}^m $ and a positive integer $ \mu $  we define $ \Tan^\mu(\mathcal{H}^\mu \restrict X,a) $ to be the set of all $ v \in \mathbf{R}^m $ such that 
$$ \Theta^{\ast \mu}\big(\mathcal{H}^\mu \restrict X  \cap \{x : |r(x-a) - v | < \epsilon\; \textrm{for some $ r > 0 $}\}, a\big) > 0 $$
for every $ \epsilon > 0 $. This is a cone with vertex at $ 0 $ and we set 
$$ \Nor^\mu(\mathcal{H}^\mu \restrict X,a) = \{v \in \mathbf{R}^m : v \bullet u \leq 0 \; \textrm{for $ u \in  \Tan^\mu(\mathcal{H}^\mu \restrict X,a) $}\}. $$

Suppose $ X \subseteq \mathbf{R}^{m} $ and $ f $ maps a subset of $ \mathbf{R}^m $ into $ \mathbf{R}^k $. Given a positive integer $ \mu $ and $ a \in \mathbf{R}^m $ we say that $ f $ is \emph{$ \mathcal{H}^\mu \restrict X $ approximately differentiable at $ a $} (cf.\ \cite[3.2.16]{Fed69}) if and only if there exists a map $ g : \mathbf{R}^m \rightarrow \mathbf{R}^k $ pointwise differentiable at $ a $ such that 
$$ \Theta^\mu(\mathcal{H}^\mu \restrict X \cap \{b: f(b) \neq g(b)\}, a) =0. $$
In this case (see \cite[3.2.16]{Fed69}) $ f $ determines the restriction of $ \Der g(a) $ on the approximate tangent cone $ \Tan^\mu(\mathcal{H}^\mu \restrict X, a) $ and we define 
$$ \ap \Der f(a) = \Der g(a)| \Tan^\mu(\mathcal{H}^\mu \restrict X,a). $$

Suppose $ X \subseteq \mathbf{R}^m $ and $ \mu $ is a positive integer. We say that $ X $ is \emph{countably $ \mathcal{H}^\mu $-rectifiable} if there exist countably many  $ \mu $-dimensional $ C^1 $-submanifolds $ \Sigma_i $ of $ \mathbb{R}^m $ such that $$ \mathcal{H}^\mu\big(X \setminus\bigcup_{i =1}^\infty \Sigma_i \big) = 0. $$ It is well known that \emph{if $ X $ is countably $ \mathcal{H}^\mu $-rectifiable with $ \mathcal{H}^\mu(X) < \infty $, then $ \Tan^\mu(\mathcal{H}^\mu\restrict X, a) $ is a $ \mu $-dimensional plane at $ \mathcal{H}^\mu $ a.e.\ $ a \in X $, and every Lipschitz function $ f : X \rightarrow \mathbf{R}^k $ has an $\mathcal{H}^\mu \restrict X $-approximate differential $ \ap \Der f(a) : \Tan^\mu(\mathcal{H}^\mu \restrict X, a)  \rightarrow \mathbb{R}^k $ at $ \mathcal{H}^\mu $ a.e.\ $ a \in X $.} At such points $ a $ we define for $ h \in \{1, \ldots , k\} $ the $ h $-dimensional approximate Jabobian of $ f $ 
$$ J_h^X f (a) = \sup \big\{ \big| [{\textstyle \bigwedge_h \ap \Der f(a)}](\xi) \big| : \xi \in {\textstyle \bigwedge_h \Tan^\mu(\mathcal{H}^\mu\restrict X, a)}, \, | \xi | = 1   \big\} $$
(see \eqref{natural map} for the definition of $ {\textstyle \bigwedge_h \ap \Der f(a)} $).	The approximate Jacobian naturally appears in area and coarea formula for $ f $; see \cite[3.2.20, 3.2.22]{Fed69}.

\subsection{Differential forms and currents} 

Let $ V $ be a vector space. We denote by $ v_1 \wedge \cdots \wedge v_m $ the simple $ m $-vector obtained by the exterior multiplication of vectors $ v_1, \ldots , v_m $ in $ V $ and $ \bigwedge_m V $ is the vector space generated by all simple $ m $ vectors of $ V $. Each linear map $ f : V \rightarrow V' $ can be uniquely extended to a linear map 
\begin{equation}\label{natural map}
{\textstyle \bigwedge_m} f : {\textstyle \bigwedge_m} V \rightarrow {\textstyle \bigwedge_m} V' 
\end{equation}
such that $ {\textstyle \bigwedge_m} f(v_1 \wedge \cdots \wedge v_m) = f(v_1) \wedge \cdots \wedge f(v_m) $ for every $ v_1, \ldots , v_m \in V $ (cf.\ \cite[1.3.1-1.3.3]{Fed69}).

 The vector space of all alternating $ m $-linear functions $ f : V^m \rightarrow \mathbf{R} $ (i.e.\ $ f(v_1, \ldots, v_m) =0 $ whenever $ v_1, \ldots v_m \in V $ and $ v_i = v_j $ for some $ i \neq j $) is denoted by $ \bigwedge^m V $. There is a natural isomorphism between $  \bigwedge^m V $ and the space of all linear $ \mathbf{R} $-valued maps on $ \bigwedge_m V $  (cf. \cite[1.4.1-1.4.3]{Fed69}). It is often convenient to use the following customary notation (see \cite{Fed69}):
$$ \langle \xi, h \rangle = h(\xi) \quad \textrm{whenever $ \xi \in {\textstyle\bigwedge_m} V $ and $ h \in {\textstyle\bigwedge^m} V $}. $$ 
Using the notation introduced in \eqref{orthonormal basis} we define 
\begin{equation}\label{n form and coform}
E= e_1 \wedge \cdots \wedge e_{n+1} \in {\textstyle \bigwedge_{n+1}} \mathbf{R}^{n+1} \quad \textrm{and} \quad E' = e'_1 \wedge \cdots \wedge e'_{n+1} \in {\textstyle \bigwedge^{n+1}}\mathbf{R}^{n+1} 
\end{equation} 
If $ V $ is an inner product space, then both $ {\textstyle \bigwedge_m V} $ and $ {\textstyle \bigwedge^m V} $ can be endowed with natural scalar products, whose \emph{associated norms} are denoted by $|\cdot | $; see \cite[1.7.5]{Fed69}.

Suppose $ U \subseteq \mathbf{R}^p$ is open and $ k \geq 0 $. A  $k$-form is a smooth map $ \phi : U \rightarrow \bigwedge^k \mathbf{R}^p $ (if $ k =0 $ we set $\bigwedge^0 \mathbf{R}^p = \mathbf{R} $). Following \cite[4.1.1, 4.1.7]{Fed69}, we denote by $ \mathcal{E}^k(U) $  the space of all smooth $ k $-forms on $ U $ and we denote by $ \mathcal{D}^k(U) $ the space of all $ k $-forms with compact support in $ U $. If $ \phi \in \mathcal{E}^k(U) $ we denote by $ d\, \phi $ the \emph{exterior derivative} of $ \phi $ (cf.\ \cite[4.1.6]{Fed69}). Moreover, if $ f $ is a smooth function mapping $ U $ into $ \mathbf{R}^q $ and $ \psi $ is a $ k $-form defined on an open subset $ V $ of $\mathbf{R}^q $ with $ f(U) \subseteq V $, then we define the $ k $-form $ f^{\#} \psi $ on $ U $ by the formula
$$ \langle v_1 \wedge \cdots \wedge v_k, f^{\#} \psi(x) \rangle = \langle {\textstyle \bigwedge_{k}}\Der f(x)(v_1 \wedge \cdots \wedge v_k), \psi(f(x)) \rangle$$
for $ x \in U $ and $ v_1, \ldots , v_k \in \mathbf{R}^p $. We refer to \cite[4.1.6]{Fed69} for the basic properties of $ f^\# $. Functions mapping a subset of $ U $ into $ \bigwedge_k(\mathbf{R}^p) $ are called \emph{$ k $-vectorfields}.

 Suppose $ U \subseteq \mathbf{R}^p$ is open and $ k \geq 0 $. A $ k $-current is a continuous $ \mathbf{R} $-valued linear map $ T $ on $ \mathcal{D}^k(U) $, with respect to the standard topology (cf.\ \cite[4.1.1]{Fed69}) and we denote the space of all $ k $-currents on $ U $ by $ \mathcal{D}_k(U) $. We say that a sequence $ T_\ell \in \mathcal{D}_k(U)$ \emph{weakly converges} to $ T\in \mathcal{D}_k(U) $ if and only if
$$ T_\ell(\phi) \to T(\phi)\qquad \textrm{for all $ \phi \in \mathcal{D}^k(U) $.} $$
If $ T $ is a $ k $-current on $ U $, then  \emph{the boundary of $ T $} is the $ (k-1) $-current $ \partial T \in \mathcal{D}_{k-1}(U) $ given by 
$$ \partial T(\phi) = T(d\, \phi) \qquad \textrm{for all $ \phi \in \mathcal{D}^k(U) $}, $$
while the \emph{support of $ T $} is defined as
$$ \spt(T) = U \setminus \bigcup\big\{ V : \textrm{$ V \subseteq U $ open and $ T(\phi) =0 $ for all $ \phi \in \mathcal{D}^k(V) $}  \}. $$
 If $ T \in \mathcal{D}_k(U) $ and $ \spt(T) $ is a compact subset of $ U $, then $ T $ can be uniquely extended to a continuous linear map on $ \mathcal{E}^k(U) $. If $ \psi \in \mathcal{E}^h(U) $, $ T \in \mathcal{D}_k(U) $ and $ h \leq k $ we set 
$$ (T \restrict \psi)(\phi) = T(\psi \wedge \phi) \qquad \textrm{for all $ \phi \in \mathcal{D}^{k-h}(U) $}. $$
If $ T \in \mathcal{D}_k(U) $, $ V $ is an open subset of $ \mathbf{R}^q $ and $ f : U \rightarrow V $ is a smooth map such that $ f | \spt (T) $ is proper, then noting that $ \spt f^\# \phi \subseteq f^{-1}(\spt\, \phi) $ and $ f^{-1}(\spt\, \phi) \cap \spt\, T $ is a compact subset of $ U $ for each $ \phi \in \mathcal{D}^k(V) $, we  define $ f_{\#}T \in \mathcal{D}_k(V) $  by the formula
\begin{equation}\label{eq push forward current}
f_\#T(\phi) = T\big[ \gamma \wedge f^\#\phi\big]  
\end{equation} 
whenever $ \phi \in \mathcal{D}^k(V) $ and $ \gamma \in \mathcal{D}^0(U) $ with $ f^{-1}(\spt \phi) \cap \spt T \subseteq {\rm interior}\,[\gamma^{-1}(\{1\})] $.  If $ \spt(T) $ is a compact subset of $ U $ then $ f_\#T(\phi) = T(f^\# \phi) $ whenever $ \phi \in \mathcal{E}^k(V) $. We refer to \cite[4.1.7]{Fed69} for the basic properties of the map $ f_\# $.

	We say that a $ k $-current $ T \in \mathcal{D}_k(U) $ is a \emph{integer multiplicity locally rectifiable $ k $-current of $ U $} provided
\begin{equation}\label{eq representaion of rectifiabl currents}
	T(\phi) = \int_{M}\langle \eta(x), \phi(x) \rangle \, d\mathcal{H}^k(x) \qquad \textrm{for all $ \phi \in \mathcal{D}^k(U) $,} 
\end{equation} 
where $ M \subseteq U $ is $ \mathcal{H}^k $-measurable and countably $ \mathcal{H}^k $-rectifiable and $ \eta $ is an $ \mathcal{H}^k \restrict M $-measurable $ k $-vectorfield such that
\begin{enumerate}
\item 	 $ \int_{K \cap M} | \eta|\, d\mathcal{H}^k < \infty $ for every compact subset $ K $ of $ U $, 
\item  $ \eta(x) $ is a simple and $ | \eta(x) | $ is a positive integer for $ \mathcal{H}^k $ a.e.\ $ x \in  M $,
\item  $ \Tan^k(\mathcal{H}^k \restrict M,x) $ is associated with $ \eta(x) $ for $ \mathcal{H}^k $ a.e.\ $ x \in M $.
\end{enumerate}
We refer to $ M $ as \emph{carrier of $ T $}.

\subsection{Legendrian currents} Here we introduce the central notion of Legendrian cycle and we collect some fundamental facts.

Let $  \alpha  \in \mathcal{E}^1(\mathbf{R}^{n+1} \times \mathbf{R}^{n+1}) $ be the \emph{contact $ 1 $-form of $ \mathbf{R}^{n+1} $}, which is defined by the formula $$ \langle (y,v),\alpha(x,u) \rangle = y \bullet u \quad \textrm{for $ (x,u), (y,v) \in \mathbf{R}^{n+1} $.}  $$
\begin{definition}\label{def legendrian set}
Let $ M \subseteq \mathbf{R}^{n+1} \times \mathbf{S}^n $ be a countably $ \mathcal{H}^n $ rectifiable set. We say that $ M $ is a \emph{Legendrian rectifiable set} if and only if for every $ Q \subseteq M $ with $ \mathcal{H}^n(Q) < \infty $ we have that 
$$ \alpha(x,u) | \Tan^n(\mathcal{H}^n \restrict Q, (x,u)) =0 $$
for $ \mathcal{H}^n $ a.e.\ $(x,u) \in Q $.
\end{definition}
\begin{remark}
	In relation with Definition \ref{def legendrian set} we recall  that $ \Tan^n(\mathcal{H}^n\restrict Q, (x,u)) $ is a $ n $-dimensional plane for $ \mathcal{H}^n $ a.e.\ $(x,u) \in Q $. 
\end{remark}
\begin{definition}\label{def Legendrian cycle}
	Let $ W \subseteq \mathbf{R}^{n+1} $ be an open set and let $ T $ be an integer-multiplicity locally rectifiable $ n $-current of $ W \times \mathbf{R}^{n+1} $ with $ \spt(T) \subseteq W \times \mathbf{S}^n $. 
	
	
	We say that $ T $ is \emph{Legendrian cycle of $ W $} if $ T \restrict \alpha =0 $ and $ \partial T =0 $.
\end{definition}
\begin{remark}
	If $ T $ is a Legendrian cycle and $ M $ is a carrier of $ T $, then $ M $ is a Legendrian rectifiable set (with $ \mathcal{H}^n(K \cap M) < \infty $ for every $ K \subseteq \mathbf{R}^{n+1} \times \mathbf{R}^{n+1} $ compact).
\end{remark}

\begin{lemma}\label{lem patching of legendrian cycles}
	Suppose $W_1, \ldots , W_m \subseteq \mathbf{R}^{n+1} $ are bounded open sets and $ T \in \mathcal{D}_{n}(\mathbf{R}^{n+1} \times \mathbf{R}^{n+1}) $ such that $ T\restrict (W_i \times \mathbf{R}^{n+1}) $ is a Legendrian cycle of $ W_i $ for every $ i = 1, \ldots , m $ and $ \spt(T) $ is a compact subset of $ \bigcup_{i=1}^m W_i \times \mathbf{S}^n $. Then $ T   $ is a Legendrian cycle of $ \mathbf{R}^{n+1} $.
\end{lemma}
\begin{proof}
	For each $ i = 1, \ldots , m $ choose an open set $ V_i $  with compact closure in $ W_i $ and $ f_i \in C^\infty(\mathbf{R}^{n+1}) $ such that $ \spt(f_i) $ is a compact subset of $ W_i $, $ \sum_{i=1}^m f_i(x) = 1 $ for every $ x \in \bigcup_{i=1}^m V_i $ and $ \spt(T) \subseteq \bigcup_{i=1}^m V_i \times \mathbf{S}^n $. Then $ T = T \restrict (\sum_{i=1}^m  f_i) $, $ \sum_{i=1}^m d f_i =0 $ on $ \bigcup_{i=1}^m V_i $, 
	$$ \langle \phi, T\restrict \alpha \rangle = \sum_{i=1}^m \langle f_i\phi, T \restrict \alpha \rangle =0 $$
	and 
	$$ \partial T(\phi) = \sum_{i=1}^m \partial T(f_i \phi) + T\Big[\Big({\textstyle\sum_{i=1}^m}df_i\Big)\wedge \phi \Big] =0 $$
	for every $ \phi \in \mathcal{D}^{n-1}(\mathbf{R}^{n+1} \times \mathbf{R}^{n+1}) $.
\end{proof}

For the next definition recall \eqref{orthonormal basis} and \eqref{n form and coform}.
\begin{definition}[Lipschitz-Killing forms (cf.\ \protect{\cite{Zaehle86}})]\label{Lipschitz Killing}
For $ k \in \{0, \ldots , n\} $	the \emph{$ k $-th Lipschitz-Killing form of $ \mathbf{R}^{n+1} $}, $ \varphi_k \in \mathcal{E}^n(\mathbf{R}^{n+1} \times \mathbf{R}^{n+1})$,  is defined by the formula
	\begin{flalign*}
		\langle \xi_1 \wedge \cdots \wedge \xi_n, \varphi_k(x,u) \rangle =\sum_{\sigma \in \Sigma_{n,k}} \langle \pi_{\sigma(1)}(\xi_1)\wedge \cdots \wedge \pi_{\sigma(n)}(\xi_n) \wedge u, E'\rangle,
	\end{flalign*} 
	for every $ \xi_1, \ldots , \xi_n \in \mathbf{R}^{n+1}\times \mathbf{R}^{n+1}$, where
	$$ \Sigma_{n,k} = \bigg\{ \sigma: \{1, \ldots , n\} \rightarrow \{0,1\}: \sum_{i=1}^n \sigma(i) = n-k \bigg\} $$
	and $ \pi_0 $ and $ \pi_1 $ are the projections defined in \eqref{projections}.
\end{definition}

 For $ 1 \leq k \leq m $ we denote by $ \Lambda(m,k) $ the set of all increasing mappings from $ \{1, \ldots, k\} $ into $ \{1, \ldots, m\} $.

\begin{lemma}[\protect{cf.\ \cite[Lemma 3.1]{Fu98}}]\label{lem exterior derivative Lipschitz-Killing}
	The exterior derivatives of the Lipschitz-Killing differential forms satisfy the following equations:
	\begin{equation}\label{lem exterior derivative Lipschitz-Killing 1}
		\langle \xi_1 \wedge \cdots \wedge \xi_{n+1}, d\, \varphi_k(x,u) \rangle  = \langle \xi_1 \wedge \cdots \wedge\xi_{n+1},  \alpha (x,u) \wedge (n-k+1) \varphi_{k-1}(x,u) \rangle 
	\end{equation} 
	for $ k = 1, \ldots, n $	and 
\begin{equation}\label{lem exterior derivative Lipschitz-Killing 2}
\langle \xi_1 \wedge \cdots \wedge \xi_{n+1}, d\, \varphi_0 (x,u) \rangle =0
\end{equation}
	whenever $ \xi_1, \ldots , \xi_{n+1} \in \Tan(\mathbf{R}^{n+1} \times \mathbf{S}^n, (x,u)) $ and $ (x,u) \in \mathbf{R}^{n+1} \times \mathbf{S}^n $.
\end{lemma}

\begin{proof}
	We fix $ (x,u)\in \mathbf{R}^{n+1} \times \mathbf{S}^n $.
	
	Suppose $ k \geq 0 $ and notice that $ \varphi_k : \mathbf{R}^{n+1} \times \mathbf{R}^{n+1} \rightarrow {\textstyle \bigwedge^n}(\mathbf{R}^{n+1} \times \mathbf{R}^{n+1}) $ is a linear map. Henceforth,  we compute (cf.\ \cite[pag.\ 352]{Fed69})
	\allowdisplaybreaks{\begin{flalign}\label{lem exterior derivative Lipschitz-Killing eq1}
		&	\langle \xi_1 \wedge \cdots \wedge \xi_{n+1}, d\,\varphi_k(x,u) \rangle \\
		& \quad  = \sum_{j=1}^{n+1} (-1)^{j-1}\, \big\langle \xi_1 \wedge \cdots \wedge \xi_{j-1} \wedge \xi_{j+1} \wedge \cdots \wedge \xi_{n+1}, \langle \xi_j , \Der \varphi_k(x,u)\rangle \big\rangle \notag \\
		& \quad   = \sum_{j=1}^{n+1} (-1)^{j-1}\, \big\langle \xi_1 \wedge \cdots \wedge \xi_{j-1} \wedge \xi_{j+1} \wedge \cdots \wedge \xi_{n+1}, \varphi_k(\xi_j) \big\rangle \notag\\
		& \quad  = \sum_{j=1}^{n+1} (-1)^{j-1} \sum_{\sigma \in \Sigma_{n,k}} \langle \pi_{\sigma(1)}(\xi_1) \wedge \cdots \wedge \pi_{\sigma(j-1)}(\xi_{j-1}) \wedge \pi_{\sigma(j)}(\xi_{j+1}) \wedge \cdots \notag\\
		&  \qquad \qquad \qquad \qquad \qquad \qquad \qquad \qquad \qquad\qquad \qquad \qquad \qquad \, \cdots \wedge  \pi_{\sigma(n)}(\xi_{n+1}) \wedge \pi_1(\xi_j), E' \rangle \notag \\
		& \quad   = (-1)^n \sum_{j=1}^{n+1}  \sum_{\sigma \in \Sigma_{n,k}} \langle \pi_{\sigma(1)}(\xi_1) \wedge \cdots \wedge \pi_{\sigma(j-1)}(\xi_{j-1}) \wedge \pi_1(\xi_j) \wedge  \pi_{\sigma(j)}(\xi_{j+1}) \wedge \cdots \notag \\
		&  \qquad \qquad \qquad \qquad \qquad \qquad \qquad \qquad \qquad\qquad \qquad \qquad \qquad \qquad  \qquad\cdots \wedge  \pi_{\sigma(n)}(\xi_{n+1}), E' \rangle \notag 
	\end{flalign}}
	for $ \xi_1, \ldots, \xi_{n+1} \in \mathbf{R}^{n+1} \times \mathbf{R}^{n+1} $, whence we readily deduce \eqref{lem exterior derivative Lipschitz-Killing 2}.
	Moreover, if $ p_u : \mathbf{R}^{n+1} \rightarrow \mathbf{R}^{n+1}$ is the orthogonal projection onto $ {\rm span}\{u\} $ we use the permutation formula (cf.\ \cite[1.4.2]{Fed69}) to compute
	\allowdisplaybreaks{\begin{flalign}\label{lem exterior derivative Lipschitz-Killing eq2}
		&	\langle \xi_1 \wedge \cdots \wedge \xi_{n+1}, \alpha(x,u) \wedge \varphi_{k-1}(x,u) \rangle\\
		&  = \sum_{j=1}^{n+1}(-1)^{j-1} \, \langle \xi_j, \alpha(x,u) \rangle \, \langle \xi_1 \wedge \cdots \wedge  \xi_{j-1} \wedge \xi_{j+1} \wedge \cdots \wedge \xi_{n+1}, \varphi_{k-1}(x,u) \rangle \notag \\
		&  = (-1)^n \sum_{j=1}^{n+1}  \sum_{\sigma \in \Sigma_{n,k-1}} \langle \pi_{\sigma(1)}(\xi_1) \wedge \cdots \wedge \pi_{\sigma(j-1)}(\xi_{j-1}) \wedge p_u(\pi_0(\xi_j)) \wedge  \pi_{\sigma(j)}(\xi_{j+1}) \wedge \cdots \notag  \\
		& \qquad \qquad \qquad \qquad \qquad \qquad \qquad \qquad \qquad\qquad \qquad \qquad\qquad \qquad \qquad\cdots  \wedge  \pi_{\sigma(n)}(\xi_{n+1}), E' \rangle \notag 
	\end{flalign}}
	whenever  $ \xi_1, \ldots, \xi_{n+1} \in \mathbf{R}^{n+1} \times \mathbf{R}^{n+1} $ and $ k \geq 1 $.  Suppose $ \tau_1, \ldots, \tau_n \in u^\perp $ is an orthonormal set and
	we define 
	$$ v_i = (\tau_i, 0) \quad \textrm{for $ i = 1, \ldots , n $,} \quad v_i = (0, \tau_{i-n}) \quad \textrm{for $ i = n+1, \ldots, 2n $}, \quad v_{2n+1} = (u,0), $$
which form an orthonormal basis of $ \mathbf{R}^{n+1} \times u^\perp $. Then we define 
	$$ v_\lambda = v_{\lambda(1)} \wedge \cdots \wedge v_{\lambda(n+1)} \qquad \textrm{whenever $ \lambda \in \Lambda(2n+1, n+1) $} $$
	and, recalling that $ \{v_\lambda : \lambda \in \Lambda(2n+1, n+1) \} $ is a basis of $ \bigwedge_{n+1}(\mathbf{R}^{n+1} \times u^\perp) $ (cf. \cite[1.3.2]{Fed69}), we notice that \eqref{lem exterior derivative Lipschitz-Killing 1}  reduces to check
 \begin{equation}\label{lem exterior derivative Lipschitz-Killing eq3}
 	(n-k+1)\, \langle v_\lambda, \alpha(x,u) \wedge  \varphi_{k-1}(x,u) \rangle  = \langle v_\lambda, d\, \varphi_k(x,u) \rangle \quad \textrm{for $ \lambda \in \Lambda(2n+1, n+1) $}.
 \end{equation}   
 Firstly, we notice that if $ \lambda \in \Lambda(2n+1, n+1) $ and $ 2n+1 \notin {\rm Im}(\lambda) $ then one can easily check from \eqref{lem exterior derivative Lipschitz-Killing eq1} and \eqref{lem exterior derivative Lipschitz-Killing eq2} that both sides of \eqref{lem exterior derivative Lipschitz-Killing eq3} must be equal to zero.

We fix now $ \lambda \in \Lambda(2n+1, n+1) $ such that $ \lambda (n+1) = 2n+1 $. Then $ \pi_1(v_{2n+1}) =0 $ and we employ \eqref{lem exterior derivative Lipschitz-Killing eq1} to compute
	\allowdisplaybreaks{\begin{flalign*}
		&\langle v_\lambda, d\,\varphi_k(x,u) \rangle \\
		&    = (-1)^n \sum_{j=1}^{n}  \sum_{\sigma \in \Sigma_{n,k}} \langle \pi_{\sigma(1)}(v_{\lambda(1)}) \wedge \cdots \wedge \pi_{\sigma(j-1)}(v_{\lambda(j-1)}) \wedge \pi_1(v_{\lambda(j)}) \wedge  \pi_{\sigma(j)}(v_{\lambda(j+1)}) \wedge \cdots \notag \\
		&  \qquad \qquad \qquad \qquad \qquad\qquad \qquad \qquad \qquad \qquad  \qquad\cdots \wedge  \pi_{\sigma(n-1)}(v_{\lambda(n)}) \wedge  \pi_{\sigma(n)}(v_{2n+1}), E' \rangle \notag \\
		&   = (-1)^n \sum_{j \in \lambda^{-1}\{n+1, \cdots, 2n\}} \sum_{\substack{\sigma \in \Sigma_{n,k}\\ \sigma(n) =0}} \langle \pi_{\sigma(1)}(v_{\lambda(1)}) \wedge \cdots \wedge \pi_{\sigma(j-1)}(v_{\lambda(j-1)}) \wedge \tau_{\lambda(j)-n}  \wedge   \notag \\
		&  \qquad \qquad \qquad \qquad \qquad  \qquad \qquad \qquad  \qquad \ \ \, \wedge \pi_{\sigma(j)}(v_{\lambda(j+1)}) \wedge \cdots  \wedge   \pi_{\sigma(n-1)}(v_{\lambda(n)}) \wedge  u, E' \rangle, \\
	\end{flalign*}}
	while, since $ p_u(\pi_0(v_{\lambda(j)})) =0 $ for $ j = 1, \ldots, n $, we obtain from \eqref{lem exterior derivative Lipschitz-Killing eq2}
	\begin{flalign*}
			\langle v_\lambda, \alpha(x,u) \wedge \varphi_{k-1}(x,u) \rangle   = (-1)^n  \sum_{\sigma \in \Sigma_{n,k-1}} \langle \pi_{\sigma(1)}(v_{\lambda(1)}) \wedge \cdots \wedge  \pi_{\sigma(n)}(v_{\lambda(n)}) \wedge u, E' \rangle.
	\end{flalign*} 
	Therefore, if  $  \mathcal{H}^0\big(\lambda^{-1}\{1, \ldots, n\}\big) \neq k-1 $ then
	$$ \langle v_\lambda, \alpha(x,u) \wedge \varphi_{k-1}(x,u) \rangle =0 = \langle v_\lambda, d\,\varphi_k(x,u) \rangle. $$
	Finally, if $ \mathcal{H}^0\big(\lambda^{-1}\{1, \ldots, n\}\big) = k-1 $ then
	\begin{flalign*}
		&(n-k+1)\,	\langle v_\lambda, \alpha(x,u) \wedge \varphi_{k-1}(x,u) \rangle \\
		& \quad  = (n-k+1)\,(-1)^n \, \langle \tau_{\lambda(1)} \wedge \cdots \wedge \tau_{\lambda(k-1)}\wedge \tau_{\lambda(k) -n} \wedge \cdots \wedge \tau_{\lambda(n) -n}\wedge u , E' \rangle \\
		& \quad  = \langle v_\lambda, d\,\varphi_k(x,u) \rangle.
	\end{flalign*} 
\end{proof}


We recall that a \emph{local variation} $(F_t)_{t \in I} $ of $ \mathbf{R}^{n+1} $ is a smooth map $ F : I \times \mathbf{R}^{n+1} \rightarrow \mathbf{R}^{n+1} $, where $ I $ is an open interval of $\mathbf{R}$ with $ 0 \in I $, such that $ F_0 = F(0, \cdot) $ is the identity of $ \mathbf{R}^{n+1} $ and $ F_t = F(t, \cdot) : \mathbf{R}^{n+1} \rightarrow \mathbf{R}^{n+1} $ is a diffeomorphism for every $ t \in I $. For such local variation we define the initial velocity vector field $ V $ by
$$ V(x) = \lim_{t \to 0} \frac{F_t(x) - x}{t} \quad \textrm{for $ x \in \mathbf{R}^{n+1} $.} $$
Moreover, if  $ F :  U \rightarrow  V $ is a $ C^2 $-diffeomorphism between open subsets of $ \mathbf{R}^{n+1} $, we define the $ C^1 $-diffeomorphism $ \Psi_F  :U \times \mathbf{S}^n \rightarrow V \times \mathbf{S}^n $ by
\begin{equation}\label{PsiF}
	\Psi_F(x,u) = \bigg(F(x), \frac{(\Der F(x)^{-1})^\ast(u)}{| (\Der F(x)^{-1})^\ast(u) |}\bigg) \quad \textrm{for $(x,u) \in U \times \mathbf{S}^n $}.
\end{equation} 

We define $ \mathbf{R}^{n+1}_0 = \mathbf{R}^{n+1} \setminus \{0\} $. For a local variation $ (F_t)_{t \in I} $ (where $ I $ is an open interval of $ \mathbf{R} $ with $ 0 \in I $) we define (cf.\ \eqref{PsiF}) the smooth map $ h : \mathbf{R}^{n+1} \times \mathbf{R}^{n+1}_0 \times I \rightarrow \mathbf{R}^{n+1} \times \mathbf{R}^{n+1} $ by
$$ h(x,u,t) = \Psi_{F_t}(x,u) \qquad \textrm{for $(x,u,t)   \in   \mathbf{R}^{n+1} \times \mathbf{R}^{n+1}_0 $} \times I $$
and we notice that $ h(x,u,0) = (x,u) $ for $ (x,u)\in \mathbf{R}^{n+1} \times \mathbf{R}^{n+1}_0 $.  Moreover we define  $$ p : \mathbf{R}^{n+1} \times \mathbf{R}^{n+1} \times \mathbf{R} \rightarrow \mathbf{R}^{n+1} \times \mathbf{R}^{n+1}, \qquad  p(x,u,t) = (x,u),  $$
$$ q : \mathbf{R}^{n+1} \times \mathbf{R}^{n+1} \times \mathbf{R} \rightarrow  \mathbf{R}, \qquad  q(x,u,t) = t, $$
$$ P : \mathbf{R}^{n+1} \times \mathbf{R}^{n+1} \rightarrow  \mathbf{R}^{n+1} \times \mathbf{R}^{n+1} \times \mathbf{R}, \qquad P(x,u) = (x,u,0). $$

\begin{lemma}[\protect{cf.\ \cite{Fu98}}]\label{Lemma Fu}
Suppose $ T $ is a Legendrian cycle of $ \mathbf{R}^{n+1} $ with $ \spt (T) $ compact, $ (F_t)_{t \in I} $  is a  local variation of $ \mathbf{R}^{n+1} $ with initial velocity vector field $ V $ and $ \theta_V : \mathbf{R}^{n+1} \times \mathbf{R}^{n+1} \rightarrow \mathbf{R} $ is given by $ \theta_V(x,u) = V(x) \bullet u $ for $(x,u) \in \mathbf{R}^{n+1} \times \mathbf{R}^{n+1} $.

Then (see \eqref{PsiF})
	$$ \frac{d}{dt}\big[(\Psi_{F_t})_{\#}T \big](\varphi_{i}) \Big|_{t=0} = (n+1-i)T(\theta_V \wedge \varphi_{i-1}) \quad \textrm{for $ i = 1, \ldots , n $} $$
	and 
	$$  \frac{d}{dt}\big[(\Psi_{F_t})_{\#}T \big](\varphi_{0})  \Big|_{t=0}  =0. $$
\end{lemma}

\begin{proof}
	Firstly, a simple direct computation leads to
	\begin{equation}\label{Lemma Fu eq1}
		h^\# \alpha \circ P = \big(p^\# \alpha + \theta_V \, dq\big) \circ P
	\end{equation}
	and 
	\begin{equation}\label{Lemma Fu eq2}
		\big(h^\# \varphi_k  \wedge \, dq\big) \circ P = \big(p^\# \varphi_k \wedge \, dq\big) \circ P.
	\end{equation}

	Suppose now $ M \subseteq \mathbf{R}^{n+1} \times \mathbf{S}^n $ is a countably $ \mathcal{H}^n $-rectifiable set and $ \eta $ is a $ \mathcal{H}^n \restrict M $-measurable simple $ n $-vectorfield such that $ | \eta(x,u) | $ is a positive integer for $ \mathcal{H}^n $ a.e.\ $(x,u) \in M $,
	$$ T (\phi) = \int_M \langle \eta(x,u), \phi(x,u) \rangle \, d\mathcal{H}^n(x,u) \quad \textrm{for each $ \phi \in \mathcal{D}^n(\mathbf{R}^{n+1} \times \mathbf{R}^{n+1}) $}, $$
	and $ \Tan^n(\mathcal{H}^n \restrict M, (x,u)) $ is associated with $ \eta(x,u) $ for $ \mathcal{H}^n $ a.e.\ $(x,u) \in M $.
	For $ t > 0 $ we define $ \llbracket 0,t \rrbracket \in \mathcal{D}_1(\mathbf{R}) $ by the formula
	$$ \llbracket 0,t \rrbracket(\beta) = \int_0^1 \langle t, \beta(st) \rangle \, ds = \int_0^t \langle 1, \beta(s) \rangle \, ds \quad \textrm{for $ \beta \in \mathcal{E}^1(\mathbf{R}) $.} $$
	Denoting by $ T \times \llbracket 0,t \rrbracket \in \mathcal{D}_{n+1}(\mathbf{R}^{n+1} \times \mathbf{R}^{n+1} \times \mathbf{R}) $ the cartesian product of $ T $ and $ \llbracket 0,t \rrbracket $, we employ \cite[4.1.8]{Fed69} to compute 
	$$ \big( T \times \llbracket 0,t \rrbracket\big)(\phi) = \int_M \int_0^t \big\langle \zeta(x,u,s), \phi(x,u,s) \big\rangle\, ds\, d\mathcal{H}^n(x,u) $$
	whenever $ \phi \in \mathcal{E}^{n+1}(\mathbf{R}^{n+1} \times \mathbf{R}^{n+1} \times \mathbf{R}) $, where
	$$ \zeta(x,u,s) =\big({\textstyle \bigwedge_n} P\big)(\eta(x,u)) \wedge w_{2n+3} \quad \textrm{for $ \mathcal{H}^n $ a.e.\ $ (x,u) \in M  $ and  for $ s \in (0,t) $.}$$
	Employing \cite[4.1.8, 4.1.9]{Fed69} and taking into account that $ \partial T = 0 $  we derive the homotopy formula
	$$
	\big(\Psi_{F_t}\big)_\# T - T    = (-1)^n\,  \partial \big[h_\#(T \times \llbracket 0,t \rrbracket)\big].
	$$
	Since $ \spt \big( h_\#\big(T \times \llbracket 0,t \rrbracket \big)\big) \subseteq \mathbf{R}^{n+1} \times \mathbf{S}^n $ we use Lemma \ref{lem exterior derivative Lipschitz-Killing} to compute
	\begin{flalign*}
		&	\big[ \big(\Psi_{F_t}\big)_\# T - T \big](\varphi_k)   = (-1)^n\,(n-k+1)\, h_\#(T \times \llbracket 0,t \rrbracket)(\alpha \wedge \varphi_{k-1}) \\
		& \qquad 	 = (-1)^n\, (n-k+1)\, \int_M \int_0^t \big\langle \zeta(x,u,s), h^\# \alpha \wedge h^\#\varphi_{k-1}(x,u,s) \big\rangle\, ds\, d\mathcal{H}^n(x,u)
	\end{flalign*}
	whence we infer that 
	\begin{flalign*}
		&\lim_{t \to 0} \frac{\big[ \big(\Psi_{F_t}\big)_\# T - T \big](\varphi_k)}{t} \\
		&  \qquad  =  (-1)^n\,(n-k+1)\, \int_M  \big\langle \zeta(x,u,0), h^\# \alpha \wedge h^\#\varphi_{k-1}(x,u,0) \big\rangle\,  d\mathcal{H}^n(x,u).
	\end{flalign*}
	Using \eqref{Lemma Fu eq1} and \eqref{Lemma Fu eq2} we deduce that
	\begin{flalign*}
		& \big\langle \zeta(x,u,0), h^\# \alpha \wedge h^\#\varphi_{k-1}(x,u,0) \big\rangle \\
		& \qquad = \langle \zeta(x,u,0), p^\# \alpha \wedge h^\#\varphi_{k-1}(x,u,0) \rangle + \langle \zeta(x,u,0), \theta_V(x,u)\, dq \wedge p^\#\varphi_{k-1}(x,u,0) \rangle 
	\end{flalign*}
	for $ \mathcal{H}^n $ a.e.\ $ (x,u) \in M $. Moreover, noting that $ p(w_{2n+3}) =0 $ and $ \langle \tau, \alpha(x,u) \rangle =0 $ whenever $ \tau \in \Tan^n(\mathcal{H}^n \restrict M, (x,u)) $ for $ \mathcal{H}^n $ a.e.\ $ (x,u) \in M $ by \cite[Theorem 9.2]{RatajZaehlebook}, we obtain that 
	$$ \langle \zeta(x,u,0), p^\# \alpha \wedge h^\#\varphi_{k-1}(x,u,0) \rangle  =0 \qquad \textrm{for $ \mathcal{H}^n $ a.e.\ $(x,u) \in M $} $$
	and employing the shuffle formula we compute 
	$$  \langle \zeta(x,u,0), \theta_V(x,u)\, dq \wedge p^\#\varphi_{k-1}(x,u,0) \rangle  = (-1)^n\,\theta_V(x,u)\, \langle \eta(x,u), \varphi_{k-1}(x,u) \rangle $$ 
	for $ \mathcal{H}^n $ a.e.\ $(x,u) \in M $. Moreover, applying Lemma \ref{lem exterior derivative Lipschitz-Killing}, we obtain
	$$\frac{d}{dt}\big[(\Psi_{F_t})_{\#}T \big](\varphi_{0})  \Big|_{t=0}  =(-1)^n\,\lim_{t\to0}\frac{[h_\#(T \times \llbracket 0,t \rrbracket)\big](d\varphi_0)}{t}=0\,.$$
\end{proof}

	\subsection{The proximal unit normal bundle} The following notion  plays a key role in this paper.
	
	\begin{definition}[\protect{cf.\ \cite[pag.\ 212]{RockafellarWets}}]\label{def proximal unit normal bundle}
		 If $\varnothing \neq C \subseteq \mathbf{R}^{n+1} $ we define the \emph{proximal unit normal bundle} of $ C $ as
		$$ \nor(C) = \{(x,\nu)\in \overline{C} \times \mathbf{S}^{n}: \dist(x+s\nu, C) = s \; \textrm{for some $ s > 0 $}\}. $$
	\end{definition}

\begin{remark}\label{rmk prox normal bundle}
	Notice that $ \nor(C) = \nor(\overline{C}) $. We recall that $ \nor(C) $ is a Borel set and it is always countably $ \mathcal{H}^n $-rectifiable; see \cite[Remark 4.3]{SantilliAnnali}\footnote{The unit normal bundle of a closed set $ C $ in \cite{SantilliAnnali} is denoted with $N(C)$.}. However, $ \mathcal{H}^n \restrict \nor(C) $ might not be a Radon measure, even when $ C $ is the closure of a smooth submanifold with bounded mean curvature; cf.\ Lemma \ref{lem Brakke example} and Remark \ref{rmk Brakke example}.
\end{remark}

The following lemma is a simple extension of well known results for sets of positive reach (cf.\ \cite[Lemma 4.23 and 4.24]{RatajZaehlebook}). 
\begin{lemma}\label{lem: Santilli20}
Suppose $ C \subseteq \mathbf{R}^{n+1} $. For $ \mathcal{H}^n $ a.e.\ $ (x,u) \in \nor(C) $ there exist numbers $$-\infty < \kappa_{1}(x,u) \leq \ldots \leq \kappa_{n}(x,u)\leq \infty $$ and vectors  $ \tau_1(x,u), \ldots , \tau_n(x,u)$  such that $ \{\tau_1(x,u), \ldots , \tau_n(x,u), u\} $ is an orthonormal basis of $\mathbf{R}^{n+1} $ and the vectors 
$$ \zeta_i(x,u) = \bigg(\frac{1}{\sqrt{1 + \kappa_i(x,u)^2}} \tau_i(x,u), \frac{\kappa_i(x,u)}{\sqrt{1 + \kappa_i(x,u)^2}}\tau_i(x,u) \bigg), \quad i = 1, \ldots, n, $$
form an orthonormal basis of $ \Tan^n(\mathcal{H}^n \restrict Q, (x,u)) $ for every $ \mathcal{H}^n $-measurable set $ Q \subseteq\nor(C) $ with $ \mathcal{H}^n(Q) < \infty $ and for $ \mathcal{H}^n $ a.e.\ $(x,u) \in Q $ (We set $ \frac{1}{\infty} = 0 $ and $ \frac{\infty}{\infty} = 1 $). In particular, 
\begin{center}
$ \nor(C) $ is a Legendrian rectifiable set.
\end{center}

Moreover, the maps $ \kappa_{1}, \ldots , \kappa_n $ can be chosen to be $ \mathcal{H}^n \restrict \nor(C) $-measurable and they are $ \mathcal{H}^n \restrict \nor(C) $-almost uniquely determined.
\end{lemma}

\begin{proof}
The existence part of the statement and the measurability  property are discussed in \cite[Section 3]{HugSantilli} (see in particular \cite[Remark 3.7]{HugSantilli}). Uniqueness can be proved as in \cite[Lemma 4.24]{RatajZaehlebook}.
\end{proof}

\begin{definition}\label{def principal curvatures}
If $ C \subseteq \mathbf{R}^{n+1} $ we denote by $ \kappa_{C,1}, \ldots , \kappa_{C,n} $ the $ \mathcal{H}^n \restrict \nor(C) $ measurable maps given by Lemma \ref{lem: Santilli20}.
\end{definition}

The following Heintze-Karcher type inequality for arbitrary closed sets is proved in \cite{HugSantilli}.

\begin{theorem}[\protect{cf.\ \cite[Theorem 3.20]{HugSantilli}}]\label{HK general}
	Let $ C \subseteq \mathbf{R}^{n+1} $ be a bounded closed set with non empty interior. Let $ K = \mathbf{R}^{n+1}  \setminus {\rm interior}(C) $ and assume that 
	$$ \sum_{i=1}^n \kappa_{K,i}(x,u) \leq 0 \quad \textrm{for $ \mathcal{H}^n $ a.e.\ $(x,u) \in \nor(K)$.} $$
	Then 
	$$ (n+1)\mathcal{L}^{n+1}({\rm interior}(C)) \leq  \int_{\nor(K)}J_n^{\nor(K)}\pi_0(x,u)\,\frac{n}{|\sum_{i=1}^n \kappa_{K,i}(x,u)|}\, d\mathcal{H}^n(x,u). $$
	Moreover, if the equality holds and there exists $ q < \infty $ such that $ |\sum_{i=1}^n \kappa_{K,i}(x,u)| \leq q $ for $ \mathcal{H}^n $ a.e.\ $ (x,u) \in \nor(K) $, then $ \textup{interior}(C)$ is a finite union of disjointed (possibly mutually tangent) open balls.
\end{theorem}



\subsection{$W^{2,n}$-functions}

Suppose $ U \subseteq \mathbf{R}^n $ is open. We denote by $ W^{k,p}(U) $ (resp.\ $ W^{k,p}_\loc(U) $) the usual Sobolev space of $ k $-times weakly differentiable functions, whose distributional derivatives up to order $ k $ belong to the Lebesgue space $ L^p(U) $ (resp.\ $ L^p_\loc(U) $); cf.\ \cite[Chapter 7]{GilbargTrudinger}.  We denote by $ \bm{\nabla}f $ and $ \der^i f $  the distributional gradient and the distributional $ i $-th differential of a Sobolev function $ f $. 

We state now two results on the fine properties of $ W^{2,n} $-functions, that play an important role in this paper. We start with some definitions.

\begin{definition}
	Given $ U \subseteq \mathbf{R}^n $ open set and $ f : U \rightarrow \mathbf{R} $ continuous function, we define $ \Gamma^+(f, U) $ as the set of $ x \in U $ for which there exists $ p \in \mathbf{R}^n $ such that
	$$ f(y) \leq f(x) + p \bullet (y-x)  \quad \forall  y \in U. $$
\end{definition}

\begin{definition}\label{def twice diff points}
	Suppose $ f : U \rightarrow \mathbf{R} $ is a continuous function. We define $ \mathcal{S}^\ast(f)$, respectively $ \mathcal{S}_\ast(f)$, as the set of $ x \in U $ where there exists a polynomial function $ P $ of degree at most $ 2 $ such that $ P(x) = f(x) $ and 
	$$ \limsup_{y \to x} \frac{f(y) - P(y)}{|y-x|^2} < \infty, \qquad \textrm{respectively} \quad \liminf_{y \to x} \frac{f(y) - P(y)}{|y-x|^2} > - \infty. $$
	Moreover, we set $ \mathcal{S}(f) $ to be the set of $ x \in U $ where there exists a polynomial function $ P $ of degree at most $ 2 $ such that $ P(x) = f(x) $ and 
	$$ \lim_{y \to x} \frac{f(y) - P(y)}{|y-x|^2} =0. $$
\end{definition}

\begin{lemma}\label{W2n twice diff}
	If $ f \in C(U) \cap W^{2,n}(U) $, then $ \mathcal{L}^n\big(U \setminus \mathcal{S}(f)\big) =0 $. 
\end{lemma}

\begin{proof}
	This is a special case of \cite[Proposition 2.2]{Caffarelli96}, which is attributed to Calderon and Zygmund.  This result can also be proved by a simple adaptation of the method of \cite{Trudinger89} (cf.\ \cite{Valentini}). An interesting generalization is given in \cite{FuAlexandrov}.
\end{proof}

The following oscillation estimate plays a key role in Lemma \ref{W2n functions Lusin} and was proved by Ulrich Menne, see \cite[Appendix B]{menne2024sharplowerboundmean}.

\begin{lemma}\label{W2n basic estimate}
	Let $ a \in \mathbf{R}^n $, $ r > 0 $, $ f \in C(B_r(a)) \cap W^{2,n}(B_r(a)) $ and $ g \in C^2(B_r(a)) $ such that 
	$$ g(a) = f(a), \qquad f(x) \geq g(x) \quad \textrm{for every $ x \in B_r(a) $.} $$
	
	Then there exists a constant $ c(n) $, depending only on $ n $, such that 
	\begin{equation}\label{W2n basic estimate gradient}
		\| \der f - \Der g(a) \|_{L^n(B_r(a))} \leq c(n) r \Big( \| \der^2 f \|_{L^n(B_r(a))} + r \| \Der^2 g \|_{L^\infty(B_r(a))}\Big), 
	\end{equation}
	\begin{equation}\label{W2n basic estimate function}
		\| f - L_a \|_{L^\infty(B_r(a))} \leq c(n) r \Big( \| \der^2 f \|_{L^n(B_r(a))} + r \| \Der^2 g \|_{L^\infty(B_r(a))}\Big),
	\end{equation}
	where $ L_a(x) = f(a) + \Der g(a)(x-a) $ for $ x \in \mathbf{R}^n $. In particular, $ f $ is pointwise differentiable at $  a$ with $ \Der f(a) = \Der g(a) $.
\end{lemma}

\begin{proof}
	See \cite[Lemma B.3]{menne2024sharplowerboundmean}. In particular,  \eqref{W2n basic estimate gradient} and \eqref{W2n basic estimate function} follows from the estimate of \cite[Lemma B.3]{menne2024sharplowerboundmean}, while the pointwise differentiability of $ f $ at $ a $ directly follows from \eqref{W2n basic estimate function}.
\end{proof}

\begin{remark}\label{rmk diff points of W2n}
	It follows from  Lemma \ref{W2n basic estimate} that if $ a \in \mathcal{S}^\ast(f) \cup \mathcal{S}_\ast(f) $, then $ a $ is  a Lebesgue point of $ \der f $, the map $ f $ is pointwise differentiable at $ a $, and 
	$$  \Der f(a) = \der f(a). $$
\end{remark}

We conclude with a Lusin-type result for $ W^{1,1}$-functions, which can be easily deduced from well known results of Calderon-Zygmund and Federer. We provide a detailed proof since we were unable to find this precise statement in classical references.

\begin{lemma}\label{lem weakly vs approx diff Sobolev maps}
	Suppose $ U $ is an open subset of $ \mathbf{R}^n $ and $ f \in W^{1,1}_\loc(U, \mathbf{R}^k)$. Then $ f $ is $ \mathcal{L}^n \restrict U $-approximately differentiable at $ \mathcal{L}^n $ a.e.\ $ x \in U $ with $ \ap \Der f(x) = \der f(x) $. In particular, there exists countably many $ \mathcal{L}^n $-measurable subsets $ A_i $ of $ U $ such that $ \mathcal{L}^n\big(U \setminus \bigcup_{i=1}^\infty A_i\big) =0 $ and $ \Lip (f|A_i) < \infty $ for every $ i \geq 1 $.
\end{lemma}

\begin{proof}
	By \cite[Theorem 12]{CalderonZygmund} (or \cite[Theorem 3.4.2]{Ziemerbook}) we have that
	\begin{equation}\label{lem weakly vs approx diff Sobolev maps eq}
		\lim_{r \to 0} r^{-n-1} \int_{B(x,r)} | f(y) - f(x) - \der f(x)(y-x)|\, d\mathcal{L}^n(y) =0 
	\end{equation} 
	for $ \mathcal{L}^n $ a.e.\ $ x \in U $. Fix now $ x \in U $ such that \eqref{lem weakly vs approx diff Sobolev maps eq} holds, define the affine function $ L_x(y) = f(x)  + \der f(x)(y-x) $ for $ y \in \mathbf{R}^n $ and notice that 
	$$ \frac{\epsilon \mathcal{L}^n(B(x,r) \cap \{y : |f(y)- L_x(y)| \geq \epsilon \, r\})}{r^n} \leq r^{-n-1} \int_{B(x,r)} |f(y)- L_x(y)|\, d\mathcal{L}^ny $$
	for every $ \epsilon > 0 $.  Henceforth,  by \eqref{lem weakly vs approx diff Sobolev maps eq} and \cite[Lemma 2.7]{Santilli19},  
	$$ \Theta^n\big(\mathcal{L}^n \restrict \{y : | f(y)- L_x(y)| \geq 2 \epsilon | y-x|\}, x\big) =0 \quad \textrm{for every $ \epsilon > 0 $}, $$
	whence we conclude that $ f $ is $ \mathcal{L}^n \restrict U $-approximately differentiable at $ x $ with $ \ap \Der f(x) = \der f(x) $ applying \cite[pag.\ 253]{Fed69} with $ \phi $ and $ m $ replaced by $ \mathcal{L}^n \restrict U $ and $  n $. Now we can use \cite[3.1.8]{Fed69} to infer the existence of the countable cover $ A_i $.
\end{proof}

\section{Legendrian cycles over $W^{2,n} $-graphs}

	The main result of this section (Theorem \ref{W2n functions normal cycles}) proves that the Legendrian cycle of associated with the subgraph of a $W^{2,n}$ function is carried over its proximal unit normal bundle. This is the key result to extend the Alexandrov sphere theorem to $ W^{2,n} $-domains.

\begin{definition}
	Suppose $ \psi : \mathbf{R}^n \rightarrow \mathbf{R}^{n+1} $ is defined as $$  \psi(y) = \frac{(-y,1)}{\sqrt{1 + |y|^2}}. $$
	If $ f \in W^{2,n}(U) \cap C(U) $ we define $$ \Phi_f(x) = (x,f(x), \psi(\nabla f(x))) \in U \times \mathbf{R} \times \mathbf{S}^n \quad \textrm{for every $ x \in \Diff(f) $.} $$
\end{definition}

\begin{remark}\label{rmk map eta est}
Let $ \mathbf{S}^n_+ = \{(z, t) \in \mathbf{R}^n \times \mathbf{R} : |z|^2 + t^2 = 1, \; t > 0\} $. Notice that $ \psi $ is a diffeomorphism onto $ \mathbf{S}^n_+ $ with inverse given by 
$$ \varphi : \mathbf{S}^n_+ \rightarrow \mathbf{R}^n, \quad \varphi(z,t) = - \frac{z}{t}, $$
and $ \| \Der \psi (y) \| \leq 2 $ for $  y \in \mathbf{R}^n $. 
\end{remark}

\begin{remark}\label{rmk approximate differentiability of Phi f}
We recall that $  \bm{\nabla} f $ is $ \mathcal{L}^n \restrict U $-approximately differentiable at $ \mathcal{L}^n $ a.e.\ $ a \in U $ by Lemma \ref{lem weakly vs approx diff Sobolev maps}. Moreover, $ \nabla f(a) = \bm{\nabla} f(a) $ for $ \mathcal{L}^n $ a.e.\ $ a \in U $ by Remark \ref{rmk diff points of W2n}. If $ a \in \Diff(f) $ and $ \nabla f $ is $ \mathcal{L}^n \restrict U $-approximately differentiable at $ a $, then $ \Phi_f $ is $ \mathcal{L}^n \restrict U $-approximately differentiable at $ a $ and 
$$ \ap \Der \Phi_f(a)(\tau) = \big(\tau, \Der f(a)(\tau), \langle \ap \Der(\nabla f)(a)(\tau), \Der \psi(\nabla f(a))\rangle \big) $$
for every $ \tau \in \mathbf{R}^n $. In particular $ \ap \Der \Phi_f(a) $ is injective for $ \mathcal{L}^n $ a.e.\ $ a \in U $ and, recalling \eqref{rmk map eta est} and noting that $ \der(\nabla f) = \ap \Der (\nabla f) $ by Lemma \ref{lem weakly vs approx diff Sobolev maps}, we infer that
\begin{flalign*} 
	\int_U \| \ap \Der \Phi_f \|^n \, d\mathcal{L}^n \leq c(n) \bigg(\mathcal{L}^n(U) + \int_U \| \Der f \|^n \, d\mathcal{L}^n + \int_U \| \der^2 f \|^n \, d\mathcal{L}^n \bigg).
\end{flalign*}
It follows that $ \Phi_f $ is a $W^{1,n}$-map over $ U $.
\end{remark}

\begin{remark}\label{rmk measure theory}
	The following basic fact of measure theory is used 	in the proof of Lemma \ref{W2n functions Lusin}. \emph{Suppose $ \mu $ is a measure over a set $ X $, $ C $ is a positive constant, and $ \{E_j: j \in S\} $ is a countable family of $ \mu $-measurable sets, such that $ \mathcal{H}^0(\{j \in S : E_j \cap E_i \neq \varnothing\}) \leq C $ for every $ i \in S $. Then 
		$$ \sum_{i \in S} \mu(E_i) \leq C \mu\bigg( \bigcup_{i \in S} E_i\bigg). $$}
\end{remark}

The following result is proved using a Rado-Reichelderfer type argument; cf.\ \cite{Maly99} and references therein.
\begin{lemma}\label{W2n functions Lusin} If $f\in C(U)\cap W^{2,n}_\loc(U)$ then $ \mathcal{H}^n(\Phi_f(Z)) =0 $ for every $ Z \subseteq \mathcal{S}^\ast(f) $ such that $ \mathcal{L}^n(Z) =0 $.
\end{lemma}

\begin{proof}
Given $ \mu > 0 $ and $ V \subseteq U $, we define $ X(V, \mu) $ as the set of $ x \in V $ for which there exists a polynomial function $ Q $ of degree at most $ 2 $ such that $ f(y) \leq Q(y) $ for every $ y \in V $, $ f(x) = Q(x) $, $ \| \Der Q(x) \| \leq \mu $ and $ \| \Der^2 Q \| \leq \mu $. If $ D \subseteq U $ is a countable dense subset of $ U $ and  $  I(c) = \{s \in \mathbf{Q} : B_s(c) \subseteq U\} $ for every $  c\in D $, then we notice that
$$ \mathcal{S}^\ast(f) \subseteq \bigcup_{c \in D} \bigcup_{s \in I(c)}\bigcup_{i =1}^\infty X(B_s(c), i). $$
Henceforth, it is sufficient to show that $ \mathcal{H}^n(\Phi_f(Z)) =0 $ whenever $ Z \subseteq X(U,\mu) $ with $ \mathcal{L}^n(Z) =0 $, for some $ \mu > 0 $.  Notice that $ f $ is pointwise differentiable at each point of the set $ X(V,\mu) $, by Lemma \ref{W2n basic estimate}.

Now we prove the following estimates: given $ c \in U $ and $ 0 < r < 1 $ such that $ \overline{B_{3r}(c)} \subseteq U $, then
\begin{equation}\label{W2n functions Lusin gradient estimate}
\| \Der f(a) - \Der f(b) \| \leq c(n)\big( \|\der^2 f\|_{L^n(B_{3r}(c))} + \mu r\big)
\end{equation} 
\begin{equation}\label{W2n functions Lusin function estimate}
|f(a) - f(b) | \leq c(n) r \big( \|\der^2 f\|_{L^n(B_{3r}(c))} + \mu \big)
\end{equation}
for every $ a, b \in X(U,\mu)  \cap B_{r}(c) $. 	We fix $ a,  b \in X(U,\mu) \cap B_r(c) $, $ a \neq b $, and we define
$$ s = \frac{|a-b|}{2} \quad \textrm{and} \quad d = \frac{a+b}{2}. $$
We notice that $ s \leq r $, 
$$ B_s(d) \subseteq B_{2s}(a) \cap B_{2s}(b) \quad \textrm{and} \quad B_{2s}(a) \cup B_{2s}(b) \subseteq B_{3r}(c); $$
consequently it follows from \eqref{W2n basic estimate gradient} of Lemma \ref{W2n basic estimate} that
\begin{flalign*}
	\| \der f -\Der f(e) \|_{L^n(B_s(d))} & \leq 	\| \der f - \Der f(e) \|_{L^n(B_{2s}(e))}\\
	& \leq  c(n) s \big( \|\der^2 f\|_{L^n(B_{2s}(e))} + \mu s \big)\\
	& \leq c(n) s \big( |\der^2 f\|_{L^n(B_{3r}(c))} + \mu r\big),
\end{flalign*}
for $ e \in \{a,b\} $, whence we infer 
\begin{flalign*}
	\bm{\alpha}(n)^{1/n}s \| \Der f(a) - \Der f(b) \| & \leq \| \der f -\Der f(a) \|_{L^n(B_s(d))} + \| \der f - \Der f(b) |_{L^n(B_s(d))}\\
	& \leq c(n) s \big( \|\der^2 f\|_{L^n(B_{3r}(c))} + \mu r\big)
\end{flalign*}
and \eqref{W2n functions Lusin gradient estimate} is proved. Moreover, combining \eqref{W2n basic estimate function} of Lemma \ref{W2n basic estimate} with \eqref{W2n functions Lusin gradient estimate}
\begin{flalign*}
&	|f(a) - f(b) | \\
& \qquad  \leq  \| f - L_a\|_{L^\infty(B_s(d))} + \| f - L_b \|_{L^\infty(B_s(d))} + s \| \Der f(a) \| + s \|  \Der f(b) \| \\
	 &  \qquad \leq  \| f - L_a\|_{L^\infty(B_{2s}(a))} + \| f - L_b \|_{L^\infty(B_{2s}(b))} + 2\mu r  \\
	 & \qquad  \leq  c(n) r \big( \|\der^2 f\|_{L^n(B_{3r}(c))} + \mu \big).
\end{flalign*}

We consider the function $ \overline{f} \times \nabla f $ mapping $ x \in \Diff(f)$ into $ (\overline{f}(x), \nabla f(x)) \in \mathbf{R}^{2n+1} $. Then it follows from \eqref{W2n functions Lusin gradient estimate} and \eqref{W2n functions Lusin function estimate} that 
\begin{equation}\label{W2n functions Lusin diamater estimate}
	\diam\big[(\overline{f} \times \nabla f)\big(B_r(c) \cap X(U,\mu)\big)\big] \leq c(n,\mu) \big( \|\der^2 f\|_{L^n(B_{3r}(c))} + r\big). 
\end{equation}
Let $ Z \subseteq X(U,\mu) $ bounded and $ \mathcal{L}^n(Z) =0 $. Given $ \epsilon > 0 $, we choose an open set $ V \subseteq U $ such that $ Z \subseteq V $, and 
\begin{equation}\label{W2n functions Lusin diamater epsilon}
	\mathcal{L}^n(V) \leq \epsilon, \quad \| \der^2 f \|_{L^n(V)}^n \leq \epsilon.
\end{equation}  
We define $ R : Z \rightarrow \mathbf{R} $ and $ \rho : Z \rightarrow \mathbf{R} $ as $$ R(x)  = \inf\bigg\{1, \frac{\dist(x, \mathbf{R}^n \setminus V)}{4}\bigg\}, \quad \textrm{for $ x \in Z $,} $$
$$ \rho(x) = \diam ((\overline{f} \times \nabla f)(B_{R(x)}(x) \cap X(U,\mu))), \quad \textrm{for $ x \in Z $.} $$
We notice that $ R $ is a Lipschitzian function with $ \Lip(R) \leq \frac{1}{4} $ and, noting that $ B_{3R(x)}(x) \subseteq V $ for every $ x \in Z $ and combining \eqref{W2n functions Lusin diamater estimate} and \eqref{W2n functions Lusin diamater epsilon}, we obtain
\begin{equation}\label{W2n functions Lusin diamater rho}
\rho(x) \leq c(n,\mu) \big( \|\der^2 f\|_{L^n(B_{3R(x)}(x))} + R(x)\big) \leq c(n,\mu) \epsilon^{1/n},
\end{equation}  
for $ x \in Z $. We prove now the following claim: \emph{there exists $ C \subseteq Z $ countable such that  $$ \textrm{$\{B_{R(y)/5}(y): y \in C\} $ is disjointed}, $$ 
	$$ Z \subseteq \bigcup_{y \in C}B_{R(y)}(y), $$
	and 
	$$ \mathcal{H}^0\big(\{y \in C: B_{3R(y)}(y) \cap B_{3R(x)}(x) \neq \varnothing\}\big) \leq c(n), \quad \textrm{for every $ x \in Z $.} $$
}
Applying Besicovitch covering theorem \cite[Theorem 2.17]{AFP00} (see also the remark at the beginning of page 52), there exists a positive constant $ \xi(n) $ depending only on $ n $, and there exist $ Z_1, \ldots , Z_{\xi(n)}  \subseteq Z $ such that 
$$ Z \subseteq \bigcup_{i=1}^{\xi(n)} \bigcup_{x \in Z_i}B_{R(x)/5}(x), $$
and $ \{B_{R(x)/5}(x) : x \in Z_i\} $ is disjointed for every $ i = 1, \ldots , \xi(n) $. We now apply \cite[Lemma 3.1.12]{Fed69} with $ S = Z_i $, $ U = Z $, $ h = \frac{R}{5}$, $ \lambda = \frac{1}{20} $ and $ \alpha = \beta = 15 $, to infer that
$$ \mathcal{H}^0 \big( \{y \in Z_i : B_{3R(y)}(y) \cap B_{3R(x)}(x) \neq \varnothing\}\big) \leq c(n), \quad \textrm{for every $ i = 1, \ldots , \xi(n) $}. $$
We define $ Z' = \bigcup_{i=1}^{\xi(n)} Z_i $, and we notice that $ Z \subseteq \bigcup_{x \in Z'} B_{R(x)/5}(x) $  and 
$$ \mathcal{H}^0\big(\{y \in Z' : B_{3R(y)}(y) \cap B_{3R(x)}(x) \neq \varnothing\}\big) \leq \xi(n)c(n), \quad \textrm{for every $ x \in Z $.} $$
Now we apply Vitali covering theorem to find a countable set $ C \subseteq Z' $ such that $ \{B_{R(x)/5}(x): x \in C\} $ is disjointed and 
$$ Z \subseteq \bigcup_{x \in C}B_{R(x)}(x), $$
which proves the claim.

Denoting with $\phi_\delta $ the size $ \delta $ approximating measure of the $ n $-dimensional Hausdorff measure $ \mathcal{H}^n $ of $ \mathbf{R}^{2n+1} $ (cf.\ \cite[2.10.1, 2.10.2]{Fed69}),  and combining \eqref{W2n functions Lusin diamater rho} with the claim above and with Remark \ref{rmk measure theory}, we have
\begin{flalign*}
	&\phi_{c(n, \mu)\epsilon^{1/n}} ((\overline{f} \times \nabla f)(Z))  \leq c(n) \sum_{y \in C}  \rho(y)^n\\
	& \qquad  \leq c(n,\mu) \sum_{y \in C} \Big( \|\der^2 f \|_{L^n(B_{3R(y)}(y))} + R(y)\Big)^n\\
	& \qquad \leq c(n,\mu) \sum_{y \in C} R(y)^n + c(n,\mu) \sum_{y \in C}\int_{B_{3R(y)}(y)}\|\der^2 f \|^n\, d\mathcal{L}^n \\
	&\qquad  \leq c(n, \mu)\mathcal{L}^n(V) + c(n, \mu) \int_V \| \der^2 f \|^n\, d\mathcal{L}^n \\
	& \qquad \leq c(n,\mu) \epsilon.
\end{flalign*}
Henceforth, letting $ \epsilon \to 0 $ we deduce that $ \mathcal{H}^n((\overline{f} \times \nabla f)(Z)) =0 $. Since $ \Phi_f = (\bm{1}_{\mathbf{R}^{n+1}} \times \psi) \circ (\overline{f} \times \nabla f) $ (see Remark \ref{rmk map eta est}), we conclude that $ \mathcal{H}^n(\Phi_f(Z)) =0 $.  
\end{proof}

\begin{lemma}\label{lem no vertical touching balls}
Suppose $ U \subseteq \mathbf{R}^n $ is open, $ \gamma > \frac{1}{2} $, $ f \in C^{0, \gamma}(U) $, $ x \in U $, $ \nu \in \mathbf{S}^n \subseteq \mathbf{R}^n \times \mathbf{R} $ such that $ B^{n+1}(\overline{f}(x)+s\nu, s) \cap \overline{f}(U) = \varnothing $ for some $ s > 0 $.
	
	Then $ \nu \notin \mathbf{R}^n \times \{0\} $. In particular, this is always true for $ f \in W^{2,n}_\loc(U) \cap C(U) $.
\end{lemma}
\begin{proof}
We prove the assertion by contradiction. Suppose $ 0 \in U $ and  $ f(0) =0 $; hence there exists $ c \in U \setminus \{0\} $ such that $$ B^{n+1}(c,|c|) \cap G = \varnothing \quad \textrm{and} \quad  K : = \overline{ B^{n+1}(c,|c|) \cap (\mathbf{R}^n \times \{0\})} \subseteq U. $$ Suppose $ L > 0 $ such that $ | f(x) - f(y) | \leq L |x-y|^\gamma $ for every $ x, y \in K $, and define $$ h_{\pm}(x) = \pm \sqrt{|c|^2 -|x-c|^2} = \pm \sqrt{2 c \bullet x - | x |^2} $$ for $ x \in K $. Henceforth, either $ h_+(x) \leq f(x) $ for every $ x \in K $, or $ f(x) \leq h_-(x) $ for every $ x \in K $. In both cases, replacing $ x $ with $ t c $ and $ 0 < t < 1 $, one obtains
	$$ t^{1-2\gamma}(2-t)\leq L^2 |c|^{2\gamma-2} $$
	for $ 0 < t < 1 $. This is clearly impossible, since $ 1-2\gamma < 0$.
\end{proof}

\begin{definition}
	If $ U \subseteq \mathbf{R}^n $ is an open set and $ f : U \rightarrow \mathbf{R} $ is a function, we define 
	$$ E_f = \{(x,u) \in U \times \mathbf{R} :  u \leq f(x)\}  $$ 
	and 
	$$ N_f = \nor(E_f) \cap (U \times \mathbf{R} \times \mathbf{R}^{n+1}). $$
\end{definition}

\begin{lemma}\label{W2n functions normal bundle}
If $ U $ is a bounded open set and $ f \in W^{2,n}(U) \cap C(U) $ then 
\begin{equation}\label{W2n functions normal bundle par}
	N_f \cap (A \times \mathbf{R} \times \mathbf{R}^{n+1})  =  \Phi_f\big[ A \cap \mathcal{S}^\ast(f) \big]
\end{equation}
 for every $ A \subseteq U $, and 
\begin{equation}\label{W2n functions normal bundle area formula}
	\int_{N_f} \beta\, d\mathcal{H}^n = \int_U \beta(\Phi_f(x))\,  \ap J_n\Phi_f(x)\, d\mathcal{L}^n(x)
\end{equation}   
whenever $ \beta : U \times \mathbf{R} \times \mathbf{R}^{n+1} \rightarrow \mathbf{R} $ is a $ \mathcal{H}^n $-measurable non-negative function. In particular, $ \mathcal{H}^n(N_f) < \infty $ and
\begin{equation}\label{W2n functions normal bundle tangent}
\textrm{$ \ap \Der \Phi_f(x)[\mathbf{R}^n] = \Tan^n(\mathcal{H}^n \restrict N_f, \Phi_f(x))$} 
\end{equation}  
for $ \mathcal{L}^n $ a.e.\ $ x \in \mathcal{S}^\ast(f) $.
\end{lemma}

\begin{proof}

Suppose $(z, \nu)  \in N_f$, where  $ z = (x, f(x)) $ with $ x \in A $, and $ s > 0 $ such that $ B^{n+1}(z+s\nu, s) \cap E_f= \varnothing $. Since  $ \nu \notin \mathbf{R}^n \times \{0\} $ by Lemma \ref{lem no vertical touching balls}, we can easily see that there exists an open set $ W \subseteq U \times \mathbf{R} $ with $ z\in W $, an open set $ V \subseteq U $ with $ x \in V $, and a smooth function $ g : V \rightarrow \mathbf{R} $ such that  $ f(x) = g(x)$ and $$ W \cap B^{n+1}(z+s\nu, s) = \{(y,u) : y \in V, \; u > g(y)\}; $$  in particular $ f(y) \leq g(y) $ for every $ y \in V $. It follows that $ x \in \mathcal{S}^\ast(f) $, $ x \in \Diff(f) $ and $ \Der f(x) = \Der g(x) $ by Lemma \ref{W2n basic estimate}. Noting that 
$$ \nu = \frac{(-\nabla g(x), 1)}{\sqrt{1 + | \nabla g(x)|^2}} $$
we conclude $(z,\nu) = \Phi_f(x) $ and $ N_f \cap (A \times \mathbf{R} \times \mathbf{R}^{n+1}) \subseteq \Phi_f(\mathcal{S}^\ast(f)\cap A) $.  

The opposite inclusion is clear, since for every $ x \in \mathcal{S}^\ast(f) $ there exists a polynomial function $ P $ and an open neighbourhood $ V $ of $ x $, such that $ P(x) = f(x) $ and $ P(y) \geq f(y) $ for every $ y \in V $.

We prove now the area formula in \eqref{W2n functions normal bundle area formula}. Since $ \Phi_f $ is a $ W^{1,n} $-map (see Remark \ref{rmk approximate differentiability of Phi f}), we apply Lemma \ref{lem weakly vs approx diff Sobolev maps} with $ f $ replaced by $ \Phi_f $ and we find countably many $ \mathcal{L}^n $-measurable sets $ B_i $ of $ U $ such that $ \Lip \big(\Phi_f|B_i\big) < \infty $ for every $ i \geq 1 $ and
$$ \mathcal{L}^n\bigg(U \setminus \bigcup_{i=1}^\infty B_i \bigg) =0. $$
 Then  we set
$$ A_1 = B_1 \cap \mathcal{S}^\ast(f), \qquad  A_i = B_i \cap \mathcal{S}^\ast(f) \setminus {\textstyle\bigcup_{\ell = 1}^{i-1}} B_\ell \quad \textrm{for $ i \geq 2 $} $$
and we notice that 
\begin{equation}\label{W2n functions decomposition of N}
\mathcal{L}^n \big( U \setminus \textstyle{\bigcup_{i=1}^\infty} A_i\big)=0 \quad \textrm{and} \quad \mathcal{H}^n\big(N_f \setminus \textstyle{\bigcup_{i=1}^\infty} \Phi_f(A_i) \big) =0 
\end{equation}    
by Lemma \ref{W2n twice diff}, Lemma \ref{W2n functions Lusin} and \eqref{W2n functions normal bundle par}. If  $ \beta : U \times \mathbf{R} \times \mathbf{R}^{n+1} \rightarrow \mathbf{R} $ is a $ \mathcal{H}^n $-measurable non-negative function, firstly we notice that (cf.\ \cite[2.4.8]{Fed69}) 
$$ \int_{U} \beta(\Phi_f(x))\, \ap J_n \Phi_f(x)\, d\mathcal{L}^n(x)  = \sum_{i=1}^\infty \int_{A_i}\beta(\Phi_f(x))\,\ap J_n\Phi_f(x)\, d\mathcal{L}^n(x), $$
then, recalling \cite[2.10.43]{Fed69}, we use the area formula for Lipschitzian maps in \cite[3.2.5]{Fed69} and the injectivity of $ \Phi_f $ to compute
$$ \int_{A_i}\beta(\Phi_f(x))\,\ap J_n\Phi_f(x)\, d\mathcal{L}^n(x)  =\int_{\Phi_f(A_i)}\beta(y)\, d\mathcal{H}^n(y) \qquad \textrm{for $ i \geq 1 $} $$
 and we use \eqref{W2n functions decomposition of N} to conclude
\begin{flalign*}
	 \int_{U} \beta(\Phi_f(x))\, \ap J_n \Phi_f(x)\, d\mathcal{L}^n(x) & =  \sum_{i=1}^\infty \int_{\Phi_f(A_i)}\beta(y)\, d\mathcal{H}^n(y)\\
	 & = \int_{N_f}\beta(y)\, d\mathcal{H}^n(y).
\end{flalign*}
Choosing $ \beta = 1 $ we conclude that $ \mathcal{H}^n(N_f) < \infty $ and we deduce that $ \Tan^n(\mathcal{H}^n \restrict N_f, (z,\nu)) $ is a $ n $-dimensional plane for $ \mathcal{H}^n $ a.e.\ $(z, \nu) \in N_f $. Let $ D_i $ be the set of $x \in A_i $ such that  $\bm{\Theta}^n(\mathcal{L}^n\restrict \mathbf{R}^n \setminus A_i, x) = 0 $, $ \ap \Der \Phi_f(x) $ is injective and $ \Tan^n(\mathcal{H}^n \restrict N_f, \Phi_f(x)) $ is a $ n $-dimensional plane. Noting \cite[2.10.19]{Fed69} and Remark \ref{rmk approximate differentiability of Phi f}, we deduce that 
$$  \mathcal{H}^n \big(A_i \setminus D_i\big)  =0. $$
Noting that $ \Tan^n(\mathcal{L}^n \restrict A_i , x) = \mathbf{R}^n $ for $ x \in D_i $, and noting that $ \Psi|A_i $ is a bi-lipschitz homeomorphism onto $ \Phi_f(A_i) $, we employ \cite[Lemma B.2]{SantilliAnnali} to conclude
\begin{flalign*}
\ap \Der \Phi_f(x)[\mathbf{R}^n] & = \Der \Psi_i(x)\big[\Tan^n(\mathcal{L}^n \restrict A_i, x)\big]\\
& \subseteq \Tan^n(\mathcal{H}^n \restrict \Psi_i(A_i), \Phi_f(x)) \subseteq \Tan^n(\mathcal{H}^n \restrict N_f, \Phi_f(x))
\end{flalign*} 
for every $ x \in D_i $. Since $ \Tan^n(\mathcal{H}^n \restrict N_f, \Phi_f(x)) $ is a $ n $-dimensional plane and $ \ap D \Phi_f(x) $ is injective for every $ x \in D_i $, we conclude that 
$$ \ap D \Phi_f(x)[\mathbf{R}^n] = \Tan^n(\mathcal{H}^n \restrict N_f, \Phi_f(x)) \quad \textrm{for every $ x \in D_i $}. $$
\end{proof}


\begin{theorem}\label{W2n functions normal cycles}
If $ U \subseteq \mathbf{R}^{n} $ is a bounded open set and $ f \in W^{2,n}(U) \cap C(U) $, then there exists a Borel  $ n $-vectorfield $ \eta $ on $ N_f $  such that
$$ \textrm{$ (\mathcal{H}^n \restrict N_f) \wedge \eta $ is a Legendrian cycle of $ U \times \mathbf{R} $} $$
and, for $ \mathcal{H}^n $ a.e.\ $(z,\nu) \in N_f $, 
$$ | \eta(z, \nu)| =1, \quad \textrm{$ \eta(z, \nu) $ is simple}, $$
$$ \textrm{$\Tan^n(\mathcal{H}^n \restrict N_f, (z, \nu)) $ is associated with $ \eta(z, \nu) $}, $$
$$ \langle \big[ {\textstyle\bigwedge_n} \pi_0\big](\eta(z,\nu)) \wedge \nu, E' \rangle > 0. $$
\end{theorem}

\begin{proof}
We identify $ \mathbf{R}^{n+1}  \simeq \mathbf{R}^{n} \times \mathbf{R}$  and we consider the orthonormal basis  $ \epsilon_1, \ldots , \epsilon_n $ of $\mathbf{R}^n $ such that $ (\epsilon_i,0) = e_i $ for $ i =1, \ldots, n $ (cf.\ \eqref{orthonormal basis}). We use the notation $ \Der_1, \ldots \Der_n $ and $ \ap \Der_1 , \ldots , \ap \Der_n $ for the partial derivatives and the approximate partial derivatives with respect to $ \epsilon_1, \ldots , \epsilon_n $, respectively.	We notice by \cite[3.1.4]{Fed69} that $ \ap \Der_i \Phi_f $ is a $ \mathcal{L}^n \restrict U $-measurable map for $ i = 1, \ldots , n $. Henceforth, by the classical Lusin theorem (cf.\ \cite[2.3.5, 2.3.6]{Fed69}) there exists a Borel map $ \xi_i : U \rightarrow \mathbf{R}^{n} \times \mathbf{R} \times \mathbf{R}^{n+1} $ such that $ \xi_i $ is $ \mathcal{L}^n \restrict U $ almost equal to $ \ap \Der_i \Phi_f $ for $ i = 1, \ldots , n $.
	
Since $ \int_U | \xi_1 \wedge \cdots \wedge \xi_n |\, d\mathcal{L}^n < \infty $ by Remark \ref{rmk approximate differentiability of Phi f}, we define   
\begin{equation}\label{W2n functions legendrian cycle}
 T(\phi) = \int_U \langle \xi_1(x) \wedge \cdots \wedge \xi_n(x), \phi(\Phi_f(x)) \rangle\, d\mathcal{L}^n(x) 
\end{equation}
for $ \phi \in \mathcal{D}^n(U \times \mathbf{R} \times \mathbf{R}^{n+1}) $ and we notice that  $ T \in \mathcal{D}_n(U \times \mathbf{R} \times \mathbf{R}^{n+1}) $. We choose now a sequence $ f_k \in C^\infty(U)\cap W^{2,n}(U) $ such that $f_k \to f$ in $ W^{2,n}(U) $, $ f_k(x) \to f(x) $ and $ \der f_k(x) \to \der f(x) $ for $ \mathcal{L}^n $ a.e.\ $ x \in U $; see \cite[Theorem 7.9]{GilbargTrudinger}. Since $ \Phi_{f_k} : U \rightarrow U \times \mathbf{R} \times \mathbf{R}^{n+1} $ is a smooth proper map,  we define $$ T_k  = (\Phi_{f_k})_{\#}\big( (\mathcal{L}^n \restrict U) \wedge \epsilon_1 \wedge \cdots \wedge \epsilon_n  \big) \in \mathcal{D}_n(U \times \mathbf{R} \times \mathbf{R}^{n+1}) $$ and we prove that
\begin{equation}\label{W2n functions normal cycles eq1}
\textrm{$T_k \to T $ in $ \mathcal{D}_n(U \times \mathbf{R} \times \mathbf{R}^{n+1}) $.}
\end{equation}
Noting that
$$T_k(\phi)  = \int_{U} \langle \Der_1 \Phi_{f_k}(x) \wedge \cdots \wedge \Der_n \Phi_{f_k}(x), \phi(\Phi_{f_k}(x)) \rangle\, d\mathcal{L}^n(x)$$
whenever $ \phi \in   \mathcal{D}^n(U \times \mathbf{R} \times \mathbf{R}^{n+1}) $, we estimate
\begin{flalign}\label{W2n functions normal cycles eq6}
|T_k&(\phi)  - T(\phi) | \\
& \leq \int_U \big|  \langle \Der_1 \Phi_{f_k}(x) \wedge \cdots \wedge \Der_n \Phi_{f_k}(x) - \xi_1(x) \wedge \cdots \wedge \xi_n(x), \phi(\Phi_{f_k}(x)) \rangle\big|\, d\mathcal{L}^n(x)\notag \\
&  \quad + \int_U \big|\langle \xi_1(x) \wedge \cdots \wedge \xi_n(x), \phi(\Phi_{f_k}(x)) - \phi(\Phi_f(x)) \rangle\big| \, d\mathcal{L}^n(x)\notag \\
& \leq \| \phi \|_{L^\infty(U)} \int_U \big|  \Der_1 \Phi_{f_k}(x) \wedge \cdots \wedge \Der_n \Phi_{f_k}(x)   - \xi_1(x) \wedge \cdots \wedge \xi_n(x)\big|\, d\mathcal{L}^n(x) \notag\\
& \quad + \int_U  \textup{ap}\,J_n\Phi_f(x)\, \| \phi(\Phi_f(x)) - \phi(\Phi_{f_k}(x)) \|\, d\mathcal{L}^n(x) \notag
\end{flalign}
for $ \phi \in \mathcal{D}^n(U \times \mathbf{R} \times \mathbf{R}^{n+1})  $. Moreover,  noting that $  \ap J_n \Phi_f \in L^1(U) $ (see Remark \ref{rmk approximate differentiability of Phi f}) and $$ \ap J_n \Phi_f(x) \, \| \phi(\Phi_f(x)) - \phi(\Phi_{f_k}(x)) \| \leq  2 \| \phi \|_{L^\infty(U \times \mathbf{R} \times \mathbf{R}^n)} \ap J_n\Phi_f(x) $$
for $ \mathcal{L}^n $ a.e.\ $ x \in U $ and for every $ k \geq 1 $, it follows from the dominated convergence theorem that 
\begin{equation}\label{W2n functions normal cycles eq2}
	\lim_{k \to \infty}\int_U  \ap J_n\Phi_f(x)\, \| \phi(\Phi_f(x)) - \phi(\Phi_{f_k}(x)) \|\, d\mathcal{L}^n(x) =0. 
\end{equation} 
We observe that 
\begin{flalign*}
	&\Der_1 \Phi_{f_k} \wedge \cdots \wedge \Der_n \Phi_{f_k}   - \xi_1 \wedge \cdots \wedge \xi_n \\
	& \quad  = \sum_{i=1}^{n} \xi_1 \wedge \cdots \wedge \xi_{i-1} \wedge \big(\Der_i \Phi_{f_k}-\xi_i\big) \wedge \Der_{i+1} \Phi_{f_k} \wedge \cdots \wedge \Der_n\Phi_{f_k}
\end{flalign*} 
and we use the generalized Holder's inequality to estimate 
\begin{flalign}\label{W2n functions normal cycles eq7}
&\int_U \big|  \Der_1 \Phi_{f_k} \wedge \cdots \wedge \Der_n \Phi_{f_k}   - \xi_1 \wedge \cdots \wedge \xi_n\big|\, d\mathcal{L}^n\\
& \quad  \leq \sum_{i=1}^n \int_U \big|\xi_1 \wedge \cdots \wedge \xi_{i-1}\big|\cdot \big| \xi_i - \Der_i \Phi_{f_k}\big|\cdot \big|  \Der_{i+1} \Phi_{f_k} \wedge \cdots \wedge \Der_n\Phi_{f_k}\big|\,d\mathcal{L}^n \notag \\
& \quad \leq \sum_{i=1}^n \int_U \| \ap\Der \Phi_f \|^{i-1}\cdot \| \ap\Der \Phi_f - \Der \Phi_{f_k}\| \cdot \| \Der\Phi_{f_k} \|^{n-i}\, d\mathcal{L}^n \notag \\
& \quad \leq  \sum_{i=1}^n \bigg(\int_U \| \ap\Der \Phi_f \|^{n}\, d\mathcal{L}^n \bigg)^{\frac{i-1}{n}}\cdot \bigg( \int_U  \| \ap\Der \Phi_f - \Der \Phi_{f_k}\|^{n}\, d\mathcal{L}^n \bigg)^{\frac{1}{n}}\cdot \bigg(\int_U \| \Der\Phi_{f_k} \|^{n}\, d\mathcal{L}^n\bigg)^{\frac{n-i}{n}}.\notag 
\end{flalign}
Moreover, by \eqref{rmk map eta est},
\begin{flalign*}
	\int_U & \| \Der \Phi_{f_k}  - \ap\Der \Phi_f \|^n \, d\mathcal{L}^n\\ 
&  \leq c(n)\bigg( \int_U \| \Der \overline{f}_k - \der \overline{f} \|^n \, d\mathcal{L}^n + \int_U \| \Der (\psi \circ \nabla f_k) - \der ( \psi \circ \bm{\nabla}f)  \|^n \, d\mathcal{L}^n\bigg)\\
&  \leq c(n)\bigg( \int_U \| \Der \overline{f}_k - \der \overline{f} \|^n \, d\mathcal{L}^n + \int_U \| \Der \psi (\nabla f_k)\|^n \|\Der^2 f_k - \der^2 f_k \|^n\, d\mathcal{L}^n \\
& \qquad \qquad \qquad \qquad \qquad \qquad \qquad \qquad + \int_U \| \Der \psi(\nabla f_k) - \Der \psi(\bm{\nabla} f) \|^n \| \der^2 f \|^n\, d\mathcal{L}^n\bigg)\\
& \leq c(n) \bigg( \int_U \| \Der \overline{f}_k - \der \overline{f} \|^n \, d\mathcal{L}^n + \int_U  \|\Der^2 f_k - \der^2 f_k \|^n\, d\mathcal{L}^n \\
& \qquad \qquad \qquad \qquad \qquad \qquad \qquad \qquad + \int_U \| \Der \psi(\nabla f_k) - \Der \psi(\bm{\nabla} f) \|^n \| \der^2 f \|^n\, d\mathcal{L}^n\bigg)
\end{flalign*}
and 
$$  \lim_{k \to \infty}\int_U \| \Der \psi(\nabla f_k) - \Der \psi(\bm{\nabla} f) \|^n \| \der^2 f \|^n\, d\mathcal{L}^n = 0 $$
 by dominated convergence theorem. Consequently $ \| \Der \Phi_{f_k}  - \der \Phi_f \|_{L^n(U)} \to 0 $, and combining \eqref{W2n functions normal cycles eq6}, \eqref{W2n functions normal cycles eq2} and \eqref{W2n functions normal cycles eq7} we obtain \eqref{W2n functions normal cycles eq1}.

Since $ \partial T_k =0 $ for every $ k \geq 1 $, we readily infer from \eqref{W2n functions normal cycles eq1} that $ \partial T =0 $. Define $ G = [\Phi_f|\mathcal{S}^\ast(f)]^{-1} : N_f \rightarrow U $  and notice that $ G $ is simply the restriction on $ N_f $ of the linear function that maps a point of $ \mathbf{R}^{n+1} \times \mathbf{R}^{n+1} $ onto its first $ n $ coordinates; in particular $ G $ is a Borel map (recall that $ N_f $ is a Borel set). We employ Lemma \ref{W2n functions normal bundle} to see that 
\begin{equation} \label{W2n functions normal cycles eq3}
	T(\phi) = \int_{N_f} \langle \xi\big[G(z,\nu)\big], \phi(z,\nu) \rangle\, d\mathcal{H}^n(z, \nu) 
\end{equation} 
for every $ \phi \in \mathcal{D}^n(U \times \mathbf{R} \times \mathbf{R}^{n+1}) $, where  $ \xi : U \rightarrow    \bigwedge_n(\mathbf{R}^{n+1} \times \mathbf{R}^{n+1})$ is the Borel map  defined as 
$$ \xi(x) = \frac{\xi_1(x) \wedge \cdots \wedge \xi_n(x)}{\big|\xi_1(x) \wedge \cdots \wedge \xi_n(x)\big|}. $$
Then we define the Borel map $$ \eta : N_f \rightarrow    {\textstyle\bigwedge_n}(\mathbf{R}^{n+1} \times \mathbf{R}^{n+1}), \quad  \eta = \xi \circ G, $$
and we readily infer that $ \Tan^n(\mathcal{H}^n \restrict N_f, (z, \nu)) $ is associated with $ \eta(z, \nu) $  for $ \mathcal{H}^n $ a.e.\ $(z, \nu) \in N_f $ by Lemma \ref{W2n functions normal bundle}, and $ (\mathcal{H}^n \restrict N_f) \wedge \eta $ is a Legendrian cycle by Lemma \ref{lem: Santilli20}. Finally, we use Remark \ref{rmk approximate differentiability of Phi f} and shuffle formula (see \cite[pag.\ 19]{Fed69}) to compute
\begin{flalign*}
\big\langle \big[\textstyle{\bigwedge_n \pi_0}\big]&(\xi(x)) \wedge \psi(\nabla f(x)), E' \big\rangle\\
&  = \frac{1}{\ap J_n \Phi_f(x)} \, \langle  \Der_1\overline{f}(x) \wedge \cdots \wedge \Der_n\overline{f}(x) \wedge \psi(\nabla f(x)), E' \rangle \\
& = \frac{e'_{n+1}(\psi(\nabla f(x)))}{\ap J_n \Phi_f(x)}  > 0
\end{flalign*}
for $ \mathcal{L}^n $ a.e.\ $ x \in U $.
\end{proof}

\section{The support of Legendrian cycles}\label{Section support}

Suppose $ C \subseteq \mathbf{R}^{n+1} $. We define $\Unp(C) $ as the set of $ x \in \mathbf{R}^{n+1} \setminus \overline{C} $ such that there exists a \emph{unique} $ y \in \overline{C} $ with $ \dist(x, C) = | y-x| $. It is well known  that $ \mathbf{R}^{n+1} \setminus (\overline{C} \cup \Unp(C)) $ is the set of points in $ \mathbf{R}^{n+1} \setminus \overline{C} $ where $ \dist(\cdot, C) $ is not differentiable, see \cite[Lemma 2.41(c)]{KolSan} and references therein. In particular, Rademacher theorem ensures that 
\begin{equation}\label{support eq1}
\mathcal{L}^{n+1}( \mathbf{R}^{n+1} \setminus (\overline{C}\cup \Unp(C))) =0.
\end{equation} 
The nearest point projection $\xi_C  $ is multivalued function mapping $ x  \in \mathbf{R}^{n+1} $ onto 
$$ \xi_C(x) = \{a \in C: | a-x| = \dist(x, C)\}. $$
Notice that $ \xi_C| \Unp(C) $ is single-valued and we define 
$$ \nu_C(x) = \frac{x-\xi_C(x)}{\dist(x,C)} \quad \textrm{and} \quad \psi_C(x) =(\xi_C(x), \nu_C(x)) $$
for $ x \in \Unp(C) $. It is  well known (see \cite[Theorem 4.8(4)]{Fed59}) that $ \xi_C $, $ \nu_C $ and $ \psi_C $ are continuous functions over $ \Unp(C) $ and it is easy to see that 
\begin{equation}\label{support eq2}
	\nor(C) = \psi_C(\Unp(C)).
\end{equation}
We also define 
$$ \rho_C(x) = \sup\{s > 0: \dist(a + s(x-a), C) = s \dist(x, C)\}  $$
for $ x \in \mathbf{R}^{n+1} \setminus \overline{C} $ and $ a \in \xi_C(x) $. This definition does not depend on the choice of $ a \in \xi_C(x) $, the function $ \rho_C : \mathbf{R}^{n+1} \setminus \overline{C} \rightarrow [1, \infty] $ is upper-semicontinuous and we set
$$ \Cut(C) = \{x \in \mathbf{R}^{n+1} \setminus \overline{C} : \rho_C(x) =1\};  $$
see \cite[Remark 2.32 and Lemma 2.33]{KolSan}.  Finally  we define 
$$ S_t(C) = \{x \in \mathbf{R}^{n+1} : \dist(x, C) = t\} \qquad \textrm{for $ t > 0 $} $$
and  we recall from \cite[Lemma 4.2(53)]{HugSantilli} that 
\begin{equation}\label{support eq3}
	\mathcal{H}^n(S_t(C) \cap \Unp(C) \cap \Cut(C)) =0 \qquad \textrm{for every $ t > 0 $.} 
\end{equation}
We are ready to prove the following lemma.

\begin{lemma}\label{dense support lem 1}
	If $ C \subseteq \mathbf{R}^{n+1} $ and $ W \subseteq \mathbf{R}^{n+1} \times \mathbf{R}^{n+1} $ is an open set such that $ W \cap \nor(C) \neq \varnothing $, then 
	$ \mathcal{H}^n(W \cap \nor(C)) > 0. $
\end{lemma}

\begin{proof}
Notice that $ \mathbf{R}^{n+1} \setminus \overline{C} \neq \varnothing $ since $ \nor(C) \neq \varnothing $. It follows from the continuity of $ \psi_C $ that $ \psi_C^{-1}(W \cap \nor(C)) $ is relatively open in $ \Unp(C) $. Henceforth,  we conclude from \eqref{support eq1} that 
$$ \mathcal{L}^{n+1}\big(\psi_C^{-1}(W \cap \nor(C))\big) > 0. $$
We define $ T = \psi_C^{-1}(W \cap \nor(C)) $. Since $ J_1 \dist(\cdot, C)  $ is $ \mathcal{L}^{n+1} $ almost equal to the constant function $ 1 $ on $ \mathbf{R}^{n+1} \setminus \overline{C} $, we use coarea formula to compute
$$ 0 < \mathcal{L}^{n+1}(T) = \int\mathcal{H}^n(T \cap S_t(C))\, dt $$
we infer there exists $ \tau > 0 $ so that $  \mathcal{H}^n(T \cap S_\tau(C)) > 0 $ and we use \eqref{support eq3} to conclude 
$$ \mathcal{H}^n\big((T \cap S_\tau(C)) \setminus \Cut(C)\big) > 0. $$
Consequently there exists $ s > \tau $ so that 
$$ \mathcal{H}^n \big(T \cap S_\tau(C) \cap \{\rho_C \geq s/ \tau\} \big) > 0. $$
Since $ \psi_C | S_\tau(C) \cap \{\rho_C \geq s/ \tau\} $ is a bi-lipschitz homeomorphism by \cite[Theorem 3.16]{HugSantilli}, 
we conclude that 
$$ \mathcal{H}^n \big(\psi_C\big[T \cap S_\tau(C) \cap \{\rho_C \geq s/ \tau\}\big] \big) > 0,  $$
whence we infer that $ \mathcal{H}^n(W \cap \nor(C)) > 0 $.
\end{proof}

\begin{remark}
In relation with Lemma \ref{dense support lem 1}, the following example is particularly appropriate.	

Suppose $ 0 < \alpha < 1 $ and $f : \mathbf{R}^{n} \rightarrow \mathbf{R} $ is a $ C^{1, \alpha} $-function such that
	$$ \mathcal{H}^n(\{(x, f(x)) : x \in \mathbf{R}^n\} \cap B) =0 $$
	whenever $ B $ is a $ n $-dimensional $ C^2 $-submanifold of $ \mathbf{R}^{n+1} $. The existence of this type of functions is proved in \cite{Kohn}. We define $ M = \{(x, f(x)) : x \in \mathbf{R}^n\} $, we choose $ \nu : M \rightarrow \mathbf{S}^n $ a unit-normal vector field of $ M $ of class $ C^{0, \alpha} $ and we define 
	$$ N^{+} = \{(x, \nu(x)) : x \in M\} \quad \textrm{and} \quad   N^{-} = \{(x, -\nu(x)) : x \in M\}. $$
	Clearly, $ N^+ $ and $ N^- $ are disjointed closed $ n $-dimensional $ C^{0,\alpha} $-submanifolds without boundary  of $ \mathbf{R}^{n+1} \times \mathbf{S}^n $ and 
	$ \nor(M) \subseteq N^+ \cup N^-$. Moreover, we notice that $ \mathcal{H}^n(\pi_0(\nor(M))) =0 $ since $ \pi_0(\nor(M)) $ is $ \mathcal{H}^n $-rectifiable of class $ 2 $ by \cite{MenneSantilli}. Noting that $ \pi_0| N^+ $ and $ \pi_0| N^- $ are homeomorphisms and recalling Remark \ref{rmk prox normal bundle}, we conclude that \emph{$ \nor(M) $ is a countably $ \mathcal{H}^n $-rectifiable subset of $ N^+ \cup N^- $ with empty relative interior.} 
\end{remark}

For the next proof we recall that, for a subset $ C \subseteq \mathbf{R}^{n+1} $, the normal cone $ \Nor(C,z) $ (see \cite[3.1.21]{Fed69}) coincides with the cone of \emph{regular normals} of $ C $ at $ z $ introduced in \cite[Definition 6.3]{RockafellarWets} (and denoted there by $ \hat{N}_C(z) $), while $ \nor(C, z) $ is the cone of proximal normals of unit length of $ C $ at $ z $ defined in \cite[Example 6.16]{RockafellarWets}.
 
\begin{lemma}\label{dense support lem 2}
Suppose $ U \subseteq \mathbf{R}^n $ is open and $ f \in C(U) $ such that $ \overline{\nabla f}[\Diff(f)] $ is a dense subset of $ U \times \mathbf{R}^n $. Then $ \nor(E_f) $ is dense in $ G \times \mathbf{S}^n_+ $, where $ G = \{(x, f(x)) : x \in U\} $.
\end{lemma}
\begin{proof}
First, we observe that 
$$ \psi(\nabla f(x)) \in  \Nor(E_f, \overline{f}(x)) \qquad \textrm{for every $ x \in \Diff(f) $}. $$
Since $  \psi $ is a diffeomorphism of $ \mathbf{R}^n $ onto $ \mathbf{S}^n_+ $ and $ f $ is continuous,  the set  $ \{(\overline{f}(x), \psi(\nabla f(x))) : x \in U\} $ is dense in $ G \times \mathbf{S}^n_+ $; consequently we infer that $ \nor(E_f)  $ is dense in $ G \times \mathbf{S}^n_+ $ by standard approximation of regular normals, see \cite[Exercise 6.18(a)]{RockafellarWets}.
\end{proof}

Fu in \cite[pag.\ 2260]{FuAlexandrov} observed that there exist continuous functions as in Lemma \ref{dense support lem 2} that belong to $ W^{2,n}(U) $. Consequently, combining Lemma \ref{dense support lem 1}, Lemma \ref{dense support lem 2} and Theorem \ref{W2n functions normal cycles}, we conclude that there exists $ n $-dimensional Legendrian cycles (of open subsets on $ \mathbf{R}^{n+1}$) whose support has positive $ \mathcal{H}^{2n} $-measure. This answers a question implicit in \cite[Remark 2.3]{RatajZaehle2015}.

\section{Reilly-type variational formulae for $ W^{2,n} $-domains}

In this section we  study the structure of the unit normal bundle of a $ W^{2,n} $-domain (see Theorem \ref{W2n domains}), and we prove the variational formulae for their mean curvature functions (see Theorem \ref{Reilly variational formualae}). The latter extends the well known variational formulae obtained by Reilly in \cite{Reilly1972} for smooth domains. As a corollary Minkowski-Hsiung formulae are also proved; see Theorem \ref{Minkowski formulae}.

\begin{definition}[Viscosity boundary]
Suppose $ \Omega \subseteq \mathbf{R}^{n+1} $ be an open set. We define $ \partial^v_+ \Omega $ to be the set of all $ p \in \partial \Omega $ such that there exists $ \nu \in \mathbf{S}^n $ and $ r > 0 $ such that 
$$ B^{n+1}(p + r \nu, r) \cap \Omega = \varnothing \quad \textrm{and} \quad B^{n+1}(p - r \nu, r) \subseteq \Omega. $$
[Notice $ \{p\} = \partial B(p+r\nu, r) \cap \partial B(p-r\nu, r) $.] Clearly for each $ p \in \partial^v_+ \Omega $ the unit vector $ \nu $ is unique. This defines an exterior unit-normal vector field on $ \partial^v_+ \Omega $,  
$$ \nu_\Omega : \partial^v_+ \Omega \rightarrow \mathbf{S}^n. $$
\end{definition}

We introduce the notion of second order rectifiability. Suppose $ X \subseteq \mathbf{R}^m $ and $ \mu $ is a positive integer such that $ \mathcal{H}^\mu(X) < \infty $. We say that $ X $ is \emph{$\mathcal{H}^\mu$-rectifiable of class $ 2 $} if and only if there exists countably many $ \mu $-dimensional submanifolds $ \Sigma_i \subseteq \mathbf{R}^m $ of class $ 2 $ such that 
$$ {\textstyle \mathcal{H}^\mu\big( X \setminus \bigcup_{i=1}^\infty \Sigma_i\big) =0}. $$

\begin{lemma}\label{lem approx diff unit normal}
	Suppose $ X \subseteq \mathbf{R}^{n+1} $ is $ \mathcal{H}^n $-measurable and $ \mathcal{H}^n $-rectifiable of class $ 2 $, and $ \nu : X \rightarrow \mathbf{S}^{n}$ is a  $ \mathcal{H}^n \restrict X  $-measurable map such that 
	$$ \nu(a) \in \Nor^n(\mathcal{H}^n \restrict X, a) \quad \textrm{for $ \mathcal{H}^n $ a.e.\ $ a \in X $.} $$
	
	Then there exist countably many  $ \mathcal{H}^n$-measurable sets $ X_i \subseteq X $ such that $ \mathcal{H}^n\big(X \setminus \bigcup_{i =1}^\infty X_i\big) =0 $ and \mbox{$ \Lip(\nu|X_i) < \infty $;} moreover, $ \nu $ is $ \mathcal{H}^n \restrict X $-approximately differentiable at $ \mathcal{H}^n $ a.e.\ $ a \in X $ and $ \ap \Der \nu(a) $ is a symmetric endomorphism of $ \Tan^n(\mathcal{H}^n \restrict X, a) $.
\end{lemma}

\begin{proof}
	Suppose $ \{\Sigma_i\}_{i \geq 1} $ is a countable family of $ C^2 $-hypersurfaces such that $$ \mathcal{H}^n\big(X \setminus {\textstyle\bigcup_{i=1}^\infty} \Sigma_i\big) =0 $$ and  $ \eta_i : \Sigma_i \rightarrow \mathbf{S}^n $  is a continuously differentiable unit-normal vector field with $ \Lip(\eta_i) < \infty $ for $ i \geq 1 $. By \cite[2.10.19(4)]{Fed69} 
	$$ \Theta^n(\mathcal{H}^n \restrict X \setminus \Sigma_i, a) = \Theta^n(\mathcal{H}^n \restrict \Sigma_i \setminus X, a) =0 \quad \textrm{for $ \mathcal{H}^n $ a.e.\ $ a \in \Sigma_i \cap X $,} $$
	whence we infer that $ \Tan^n(\mathcal{H}^n \restrict X, a) = \Tan(\Sigma_i,a) $ for $ \mathcal{H}^n $ a.e.\ $ a \in \Sigma_i \cap X $.  In particular, if we define $$ \Sigma_i^+ =\{a \in \Sigma_i \cap X : \eta_i(a) =  \nu(a)\} \quad \textrm{and}  \quad  \Sigma_i^- =\{a \in \Sigma_i \cap X : \eta_i(a) = - \nu(a)\} $$
	we infer that $ \mathcal{H}^n\big(\Sigma_i \cap X \setminus (\Sigma_i^+ \cup \Sigma_i^-)\big) =0 $ for each $ i \geq 1 $.
	Moreover, employing again \cite[2.10.19(4)]{Fed69} we infer that 
	\begin{equation*}
		 \Theta^n(\mathcal{H}^n \restrict X \setminus \Sigma_i^\pm, a) =0 \quad \textrm{for $ \mathcal{H}^n $ a.e.\ $ a \in \Sigma_i^\pm $} 
	\end{equation*}
	whence we deduce that $ \nu $ is $ \mathcal{H}^n \restrict X $ approximately differentiable at $ a $ with  
	$$ \ap \Der \nu(a) = \pm \Der\eta_i(a) $$
for $ \mathcal{H}^n $ a.e.\ $ a \in \Sigma_i^\pm $.	Finally, since $ \Der \eta_i(a) $ is symmetric, we conclude the proof.
\end{proof}

We introduce now the class of $ W^{2,n} $ domains in a slightly more general fashion than the notion given in the Introduction. 

	\begin{definition}\label{def W2n domains}
	An open set $ \Omega\subseteq \mathbf{R}^{n+1} $ is a $ W^{2,n} $-domain if and only there exists a couple $(\Omega', F) $, where 
	\begin{enumerate}
		\item $ \Omega' \subseteq \mathbf{R}^{n+1} $ is an open set such that for each $ p \in \partial \Omega' $ there exist $ \epsilon > 0 $, $ \nu \in \mathbf{S}^n $, a bounded open set $ U \subseteq \nu^\perp $ with $ 0 \in U $ and a continuous function $ f \in W^{2,n}(U)$ with $ f(0) =0 $ such that 
		$$ \{p + b + \tau \nu : b \in U, \, -\epsilon < \tau \leq f(b)\} = \overline{\Omega'} \cap \{p + b + \tau \nu: b \in U, \, -\epsilon < \tau < \epsilon \}, $$
		\item $ F $ is a $ C^2 $-diffeomorphism defined over an open set $ V \subseteq  \mathbf{R}^{n+1} $ such that $ \overline{\Omega'} \subseteq V $,
		\item $ F(\Omega') = \Omega $.
	\end{enumerate}
\end{definition}

\begin{remark}
	This class of domains is invariant under images of $ C^2 $-diffeomorphisms, which is clearly a necessary condition in order to provide a natural framework to generalize Reilly's variational formulae. We do not know if we really need to introduce the diffeomorphism $ F $ in the definition above; in other words, if $ \Omega' $ belongs to the class $ \mathcal{S} $ of domains satisfying only condition (1) of Definition \ref{def W2n domains}, is it true that $ F(\Omega') $ belongs to $ \mathcal{S} $ too?
\end{remark}    

We collect some basic properties of $ W^{2,n} $-domains.
\begin{lemma}\label{W2n domains basic lemma}
If $ \Omega \subseteq \mathbf{R}^{n+1} $ is a $W^{2,n} $-domain, then the following statements hold.
\begin{enumerate}
\item\label{W2n domains basic lemma 1} $ \mathcal{H}^n(\partial \Omega \setminus \partial^v_+\Omega) =0 $ and $ K \cap \partial \Omega $ is $\mathcal{H}^n $-rectifiable of class $ 2 $ for every compact set $ K \subseteq \mathbf{R}^{n+1} $.
\item\label{W2n domains basic lemma 2} For $ \mathcal{H}^n $ a.e.\ $ p \in \partial \Omega $, 
 $$ \Tan^n(\mathcal{H}^n \restrict \partial \Omega, p) = \Tan(\partial \Omega, p)  = \nu_\Omega(p)^\perp. $$
\item\label{W2n domains basic lemma 3} For every $ p \in \partial^v_+ \Omega $,
 $$ \Tan^{n+1}(\mathcal{L}^{n+1} \restrict  \Omega, p) = \Tan(\Omega, p)  = \{v \in \mathbf{R}^{n+1}: v \bullet \nu_\Omega(p) \leq 0\}. $$
\end{enumerate}
\end{lemma}

\begin{proof}
	Suppose $ \Omega = F(\Omega') $, where $ \Omega' $ and $ F $ are as in Definition \ref{def W2n domains}. Clearly, $ F(\partial \Omega') = \partial \Omega $ and $ F(\partial^v_+ \Omega') = \partial^v_+ \Omega $. Therefore, assertion \eqref{W2n domains basic lemma 1}  follows from Theorem \ref{W2n twice diff}, Theorem \ref{Nabelpunktsatz} and Remark \ref{rmk Lusin condition (N) for f}.  
	
	If $ p \in \partial^v_+ \Omega $, we have $ \Tan(\partial \Omega, p)  \subseteq \nu_\Omega(p)^\perp; $ since $  \Tan^n(\mathcal{H}^n \restrict \partial \Omega, p) \subseteq \Tan(\partial \Omega, p) $ for every $ p \in \partial \Omega $ and $ \Tan^n(\mathcal{H}^n \restrict \partial \Omega, p) $ is a $ n $-dimensional plane for $ \mathcal{H}^n $ a.e.\ $ p \in \partial \Omega $, we obtain \eqref{W2n domains basic lemma 2}. Finally it follows from definitions that 
$ \Tan(\Omega, p)  = \{v \in \mathbf{R}^{n+1}: v \bullet \nu_\Omega(p) \leq 0\} $ and $ \Tan^{n+1}(\mathcal{L}^{n+1} \restrict  \Omega, p) = \{v \in \mathbf{R}^{n+1}: v \bullet \nu_\Omega(p) \leq 0\} $
for every $ p \in \partial^v_+ \Omega $.
\end{proof}

 By Lemma \ref{lem approx diff unit normal} the map $ \nu_\Omega $ is $ \mathcal{H}^n \restrict \partial \Omega $-approximately differentiable with a symmetric approximate differential $ \ap \Der \nu_\Omega(x) $ at $ \mathcal{H}^n $ a.e.\ $ x \in \partial \Omega $. Consequently we introduce the following definition.

\begin{definition}[Approximate principal curvatures]
	Suppose $ \Omega \subseteq \mathbf{R}^{n+1} $ is a  $ W^{2,n} $-domain. The \emph{approximate principal curvatures of $ \Omega $} are the $ \mathbf{R} $-valued $ (\mathcal{H}^n \restrict \partial \Omega) $-measurable maps 
	$$ \rchi_{\Omega, 1}, \ldots , \rchi_{\Omega, n}, $$
	defined so that $ \rchi_{\Omega, 1}(p) \leq \ldots \leq \rchi_{\Omega, n}(p) $ are the eigenvalues of $ \ap \Der \nu_\Omega(p) $ for $ \mathcal{H}^n $ a.e.\ $ p \in \partial \Omega $.
\end{definition} 

We prove now the main structure theorem for the unit normal bundle $ \nor(\Omega) $ of a $ W^{2,n} $-domain.

\begin{theorem}\label{W2n domains}
	If $ \Omega \subseteq \mathbf{R}^{n+1} $ is a $W^{2,n}$-domain then the following statements hold.
	\begin{enumerate}
	\item\label{W2n domains 1} $ \mathcal{H}^n(\overline{\nu_\Omega}(Z)) =0 $ whenever $ Z \subseteq \partial^v_+ \Omega $ with $ \mathcal{H}^n(Z) =0 $.
	\item\label{W2n domains 2} $ \mathcal{H}^n\big(\nor(\Omega)
	\setminus \overline{\nu_\Omega}(\partial^v_+ \Omega)\big) =0 $.
	\item\label{W2n domains 3}  $ \kappa_{\Omega, i}(x, u) = \rchi_{\Omega, i}(x) $ for every $ i =1, \ldots , n $ and for $ \mathcal{H}^n $ a.e.\ $ (x,u)\in \nor(\Omega) $. In particular, $ \kappa_{\Omega,i}(x,u) < \infty $ for $ \mathcal{H}^n $ a.e.\ $ (x,u)\in \nor(\Omega) $.
	\item\label{W2n domains 4} If $ \partial \Omega $ is compact, then $ \mathcal{H}^n(\nor(\Omega))< 	\infty $ and there exists a unique Legendrian cycle $ T $ of $ \mathbf{R}^{n+1} $ such that
	$$ T = (\mathcal{H}^n \restrict \nor(\Omega)) \wedge \eta,  $$
	where $ \eta $ is a $ \mathcal{H}^n \restrict \nor(\Omega) $ measurable $ n $-vectorfield such that
	$$ | \eta(x, u)| =1, \quad \textrm{$ \eta(x, u) $ is simple}, $$
	$$ \textrm{$\Tan^n(\mathcal{H}^n \restrict \nor(\Omega), (x, u)) $ is associated with $ \eta(x, u) $} $$
	and 
	$$ \langle \big[ {\textstyle\bigwedge_n} \pi_0\big](\eta(x,u)) \wedge u, E' \rangle > 0 $$
	for $ \mathcal{H}^n $ a.e.\ $(x,u) \in \nor(\Omega) $. In this case, $ \eta = \zeta_1 \wedge \ldots \wedge \zeta_n $, where 
	$$ \zeta_i = \Bigg(\frac{1}{\sqrt{1 + \kappa_{\Omega, i}^2}} \tau_i, \frac{\kappa_{\Omega,i}}{\sqrt{1 + \kappa_{\Omega,i}^2}}\tau_i \Bigg) \quad \textrm{for $i = 1, \ldots, n $} $$
	and $ \tau_1(x,u), \ldots , \tau_n(x,u) $ are an orthonormal basis of $ u^\perp $ such that $ \tau_1(x,u) \wedge \ldots \wedge \tau_n(x,u) \wedge u = E $ for $ \mathcal{H}^n $ a.e.\ $(x,u) \in \nor(\Omega) $.
	\end{enumerate}
\end{theorem}

\begin{proof}
		Suppose $ \Omega = F(\Omega') $, where $ \Omega' $ and $ F $ are as in Definition \ref{def W2n domains}. We recall the definition of $ \Psi_F $ from \eqref{PsiF} and   notice that
		\begin{equation}\label{W2n domains eq1}
			\Psi_F(\nor(\Omega')) = \nor(\Omega) 
		\end{equation} 
	by \cite[Lemma 2.1]{Santilli21}.
	Since $ F(\partial^v_+ \Omega') = \partial^v_+ \Omega $, we readily infer from \eqref{W2n domains eq1}  that 
	$$ \Psi_F(x, \nu_{\Omega'}(x)) = (F(x), \nu_\Omega(F(x))) \quad \textrm{for every $ x \in \partial^v_+\Omega' $} $$
	and 
		\begin{equation}\label{W2n domains eq2}
		\Psi_F\big(\overline{\nu_{\Omega'}}(F^{-1}(S))\big) = \overline{\nu_\Omega}(S) \quad \textrm{for every $ S \subseteq \partial^v_+ \Omega $}.
	\end{equation}

To prove the assertions in \eqref{W2n domains 1} and \eqref{W2n domains 2} we notice, firstly, that they are true for $ \Omega' $ as a consequence of Lemma \ref{W2n functions Lusin} and \eqref{W2n functions normal bundle par} of Lemma \ref{W2n functions normal bundle}; then we apply \eqref{W2n domains eq1}  and \eqref{W2n domains eq2}.

To prove \eqref{W2n domains 3} we first employ Lemma \ref{lem approx diff unit normal} to find a countable family $ X_i \subseteq \partial^v_+ \Omega $ such that $ \mathcal{H}^n\big(\partial \Omega \setminus \bigcup_{i=1}^\infty X_i\big) =0 $ and $ \Lip(\nu_\Omega|X_i)< \infty $ for every $ i \geq 1 $; then we define $ Y_i $ to be the set of $ x \in X_i $ such that $ \nu_\Omega $ is $ \mathcal{H}^n\restrict \partial \Omega $ approximately differentiable at $ x $, $ \Tan^n(\mathcal{H}^n\restrict \partial \Omega, x)  $  and $ \Tan^n(\mathcal{H}^n\restrict \nor(\Omega), \overline{\nu_\Omega}(x)) $ are $ n $-dimensional planes, and $ \Theta^n(\mathcal{H}^n\restrict \partial \Omega \setminus X_i, x) =0 $. We notice that $ \overline{\nu_\Omega}| X_i $ is bi-lipschitz and, since $\Tan^n(\mathcal{H}^n\restrict \nor(\Omega), (x,u)) $ is a $ n $-dimensional plane for $ \mathcal{H}^n $ a.e.\ $ (x,u) \in \nor(\Omega) $, we conclude that 
\begin{equation*}
	\mathcal{H}^n(X_i \setminus Y_i) =0 \quad \textrm{for every $ i \geq 1 $.}
\end{equation*} 
It follows from \eqref{W2n domains 1}  and \eqref{W2n domains 2} that 
\begin{equation}\label{W2n domains eq3}
	\mathcal{H}^n\big(\nor(\Omega) \setminus {\textstyle\bigcup_{i=1}^\infty \overline{\nu_\Omega}(Y_i)}\big) =0.
\end{equation}
We fix now $ x \in Y_i $. Then there exists a map $ g : \mathbf{R}^{n+1} \rightarrow \mathbf{R}^{n+1} \times \mathbf{R}^{n+1} $ pointwise differentiable at $ x $ such that $ \Theta^n(\mathcal{H}^n \restrict \partial \Omega \setminus \{g = \overline{\nu_\Omega}\}, x)  =0 $ and $ \ap \Der \overline{\nu_\Omega}(x) = \Der g(x) |  \Tan^n(\mathcal{H}^n\restrict \partial \Omega, x) $; noting that $ \ap \Der \overline{\nu_\Omega}(x) $ is injective, $ g | X_i \cap \{g = \overline{\nu_\Omega}\} $ is bi-lipschitz and  
$$ \Tan^n(\mathcal{H}^n \restrict \partial \Omega, x) = \Tan^n(\mathcal{H}^n \restrict X_i \cap \{g = \overline{\nu_\Omega}\}, x), $$
we readily infer by \cite[Lemma B.2]{SantilliAnnali} that
\begin{equation*}
\ap\Der \overline{\nu_\Omega}(x)[\Tan^n(\mathcal{H}^n\restrict \partial \Omega, x)] = \Tan^n(\mathcal{H}^n\restrict \nor(\Omega), \overline{\nu_\Omega}(x)).
\end{equation*}  
Henceforth, if $ \tau_1, \ldots , \tau_n $ is an orthonormal basis of $ \Tan^n(\mathcal{H}^n \restrict \partial \Omega, x) $ with $ \ap \Der \nu_\Omega(x)(\tau_i) = \rchi_{\Omega, i}(x) \tau_i $ for $ i = 1, \ldots , n $, we conclude that 
$$ \Bigg\{\bigg(\frac{1}{\sqrt{1 + \rchi_{\Omega, i}(x)^2}}\tau_i, \frac{\rchi_{\Omega, i}(x)}{\sqrt{1 + \rchi_{\Omega, i}(x)^2}}\tau_i\bigg): i = 1, \ldots, n\Bigg\} $$
is an orthonormal basis of $ \Tan^n(\mathcal{H}^n \restrict \nor(\Omega), \overline{\nu_\Omega}(x)) $. Since $x $ is arbitrarily chosen in $ Y_i $, thanks to \eqref{W2n domains eq3},  we deduce from the uniqueness stated in Lemma \ref{lem: Santilli20} that 
$$ \kappa_{\Omega, i}(x,u) = \rchi_{\Omega,i}(x) \quad \textrm{for $ \mathcal{H}^n $ a.e.\ $ (x,u) \in \nor(\Omega)$}. $$

Finally we prove \eqref{W2n domains 4}. By Lemma \ref{lem: Santilli20} we can choose maps $ \tau_1, \ldots , \tau_n $ defined $ \mathcal{H}^n $ a.e.\ on $ \nor(\Omega') $ such that $ \tau_{1}(x,u), \ldots , \tau_n(x,u), u $ is an orthonormal basis of $ \mathbf{R}^{n+1} $,
\begin{equation}\label{W2n domains eq4}
\tau_1(x,u) \wedge \cdots \tau_n(x,u) \wedge u = e_1 \wedge \cdots \wedge e_{n+1} \quad \textrm{for $ \mathcal{H}^n $ a.e.\ $(x,u)\in \nor(\Omega') $}
\end{equation} 
and  the vectors 
$$ \zeta_i'(x,u) = \bigg(\frac{1}{\sqrt{1 + \kappa_{\Omega', i}(x,u)^2}} \tau_i(x,u), \frac{\kappa_{\Omega',i}(x,u)}{\sqrt{1 + \kappa_{\Omega',i}(x,u)^2}}\tau_i(x,u) \bigg), \quad i = 1, \ldots, n, $$
form an orthonormal basis of $ \Tan^n(\mathcal{H}^n \restrict \nor(\Omega'), (x,u)) $ for $ \mathcal{H}^n $ a.e.\ $(x,u)\in \nor(\Omega') $. Then we define 
$$ \eta' = \zeta_1' \wedge \cdots \wedge \zeta_n' $$
and notice that 
	$$ | \eta'(x, u)| =1, \quad \textrm{$ \eta'(x, u) $ is simple}, $$
$$ \textrm{$\Tan^n(\mathcal{H}^n \restrict \nor(\Omega'), (x, u)) $ is associated with $ \eta'(x, u) $} $$
and (see \eqref{projections} and \eqref{n form and coform}) 
\begin{equation}\label{W2n domains eq6}
	\langle \big[ {\textstyle\bigwedge_n} \pi_0\big](\eta'(x,u)) \wedge u, E' \rangle > 0 \quad \textrm{(by \eqref{W2n domains eq4})}
\end{equation}  
for $ \mathcal{H}^n $ a.e.\ $(x,u) \in \nor(\Omega') $. If $ p \in \partial \Omega' $, $ \epsilon > 0 $, $ \nu \in \mathbf{S}^n $,  $ U \subseteq \nu^\perp $ is a bounded open set with $ 0 \in U $ and  $ f \in W^{2,n}(U)$ is a continuous function with $ f(0) =0 $ such that 
$$ \{p + b + \tau \nu : b \in U, \, -\epsilon < \tau \leq f(b)\} = \overline{\Omega'} \cap C_{U,\epsilon}, $$
where $ C_{U, t} = \{p + b + \tau \nu: b \in U, \, -t < \tau < t \} $ for each $ 0 < t \leq \infty $, then we observe that 
$$ N_f = \nor(\Omega') \cap (C_{U, \epsilon} \times \mathbf{S}^n)$$ 
where $ N_f =  \nor(E_f) \cap (C_{U, \infty} \times \mathbf{S}^n) $ and $ E_f = \{p + b + \tau \nu : b \in U, \, -\infty < \tau \leq f(b)\} $. It follows from \eqref{W2n domains eq6} and Theorem \ref{W2n functions normal cycles} that $ \eta'| \big[ \nor(\Omega') \cap (C_{U,\epsilon} \times \mathbf{S}^n)\big] $ is $ \mathcal{H}^n $ almost equal to a Borel $ n $-vectorfield defined over $ \nor(\Omega') \cap (C_{U,\epsilon} \times \mathbf{S}^n) $ and  $ (\mathcal{H}^n \restrict \big[\nor(\Omega') \cap (C_{U,\epsilon} \times \mathbf{S}^n)\big]) \wedge \eta' $ is a $ n $-dimensional Legendrian cycle of $ C_{U, \epsilon} $.
 Henceforth,  we define the integer multiplicity locally rectifiable $ n $-current $$ T' = \big(\mathcal{H}^n \restrict \nor(\Omega')\big) \wedge \eta' $$ and we conclude by Lemma \ref{lem patching of legendrian cycles} that $ T' $ is a Legendrian cycle of $ \mathbf{R}^{n+1} $.  
 
 We define now $ \psi = \Psi_F | \nor(\Omega') $ and, recalling \eqref{W2n domains eq1} and noting that 
\begin{equation}
\ap \Der \psi(\psi^{-1}(y,v)) = \Der \Psi_F(\psi^{-1}(y,v))
\end{equation}  
for $ \mathcal{H}^n $ a.e.\ $ (y,v) \in \nor(\Omega) $, we define 
$$ \eta(y,v) = \frac{\big[{\textstyle \bigwedge_n \ap \Der \psi(\psi^{-1}(y,v))}\big] \eta'(\psi^{-1}(y,v))}{J^{\nor(\Omega')}_n \psi(\psi^{-1}(y,v))} $$ 
for $ \mathcal{H}^n $ a.e.\ $ (y,v) \in \nor(\Omega) $. Since $ \Psi_F $ is a diffeomorphism we have that $ \eta(y,v) \neq 0 $ for $ \mathcal{H}^n $ a.e.\ $ (y,v) \in \nor(\Omega) $. We now apply \cite[4.1.30]{Fed69} with $ U $, $ K $, $ W $, $ \xi $, $ G $ and $ g $ replaced by $ \mathbf{R}^{n+1} \times \mathbf{R}^{n+1} $, $ \partial \Omega' \times \mathbf{S}^n $, $ \nor(\Omega') $, $ \Psi_F $ and $ \psi $ respectively. We infer that  
$$ (\Psi_F)_{\#}\big[(\mathcal{H}^n \restrict \nor(\Omega')) \wedge \eta' \big] = \big(\mathcal{H}^n \restrict \nor(\Omega)\big) \wedge \eta  $$
and that $ | \eta(y,v) | = 1 $ and $ \Tan^n(\mathcal{H}^n \restrict \nor(\Omega), (y,v)) $ is associated with $ \eta(y,v) $ for $ \mathcal{H}^n $ a.e.\ $ (y,v) \in \nor(\Omega) $.
Clearly $ \big(\mathcal{H}^n \restrict \nor(\Omega)\big) \wedge \eta $ is a cycle, and $\big[ \big(\mathcal{H}^n \restrict \nor(\Omega)\big) \wedge \eta\big]  \restrict \alpha =0 $
by Lemma \ref{lem: Santilli20}. Finally, if $ \bm{\ast} $ is the Hodge-star operator with respect to $ E $ (cf.\ Remark \ref{rmk hodge star}), since $ \tau_1(x,u) \wedge \cdots \wedge \tau_n(x,u) = (-1)^n\, (\bm{\ast}u) $ and 
\begin{flalign*}
&	\big[{\textstyle \bigwedge_n} \pi_0 \big](\eta(\Psi_F(x,u)))  \\
& \quad = \frac{1}{J_n^{\nor(\Omega')}\psi(x,u)}\, \Bigg(\prod_{i=1}^{n}\frac{1}{\sqrt{1 + \kappa_{\Omega', i}(x,u)^2}}\Bigg)\, \big[ \Der F(x)(\tau_1(x,u)) \wedge \cdots \wedge \Der F(x)(\tau_n(x,u))\big] \\
& \quad = \frac{(-1)^n}{J_n^{\nor(\Omega')}\psi(x,u)}\, \Bigg(\prod_{i=1}^{n}\frac{1}{\sqrt{1 + \kappa_{\Omega', i}(x,u)^2}}\Bigg)\, \big[ {\textstyle \bigwedge_n} \Der F(x)\big](\bm{\ast} u)
\end{flalign*}
for $ \mathcal{H}^n $ a.e.\ $(x,u) \in \nor(\Omega') $, it follows by Remark \ref{rmk hodge star} below and \eqref{W2n domains eq1} that 
either 
$$ 	\langle \big[ {\textstyle\bigwedge_n} \pi_0\big](\eta(y,v)) \wedge v, E' \rangle > 0 \quad \textrm{for $ \mathcal{H}^n $ a.e.\ $(y,v) \in \nor(\Omega) $} $$
or 
$$ 	\langle \big[ {\textstyle\bigwedge_n} \pi_0\big](\eta(y,v)) \wedge v, E' \rangle < 0 \quad \textrm{for $ \mathcal{H}^n $ a.e.\ $(y,v) \in \nor(\Omega) $}. $$
This settles the existence part in statement \eqref{W2n domains 4}. Uniqueness easily follows from the defining  conditions of $ T $ and the representation of $ \eta $ follows from Lemma \ref{lem: Santilli20}.
\end{proof}

\begin{remark}\label{rmk hodge star}
		Let $ \bm{\ast} : \mathbf{R}^{n+1} \rightarrow \bigwedge_{n}\mathbf{R}^{n+1} $ be the Hodge-star operator, taken with respect to $ E $; cf.\  \cite[1.7.8]{Fed69}. We notice that if $ u \in \mathbf{S}^n $ and $ \tau_1, \ldots , \tau_n $ is an orthonomal basis of $ u^\perp $ such that $ u \wedge  \tau_1 \wedge \cdots \wedge \tau_n  = E $, then it follows from the shuffle formula \cite[pag.\ 18]{Fed69} that
	$$ \bm{\ast} u = \tau_1 \wedge \cdots \wedge \tau_n. $$
	Using this remark, we prove that \emph{if $ F : \mathbf{R}^{n+1} \rightarrow \mathbf{R}^{n+1} $ is a diffeomorphism, then either
		$$ \langle \big[ {\textstyle\bigwedge_n} \Der F(x)\big](\bm{\ast} u) \wedge (\Der F(x)^{-1})^\ast(u), E'\rangle > 0 \quad \textrm{for every $(x,u) \in \mathbf{R}^{n+1} \times \mathbf{S}^n $} $$
		or 
		$$ \langle \big[ {\textstyle\bigwedge_n} \Der F(x)\big](\bm{\ast} u) \wedge (\Der F(x)^{-1})^\ast(u), E' \rangle < 0 \quad \textrm{for every $(x,u) \in \mathbf{R}^{n+1} \times \mathbf{S}^n $}. $$} By contradiction, assume that there exists $ (x,u) \in \mathbf{R}^{n+1} \times \mathbf{S}^n $ such that 
$$ \langle \big[ {\textstyle\bigwedge_n} \Der F(x)\big](\bm{\ast} u) \wedge (\Der F(x)^{-1})^\ast(u), E' \rangle = 0 $$
and choose an orthonormal basis $ \tau_1, \ldots, \tau_n $ of $ u^\perp $ such that $ u \wedge  \tau_1 \wedge \cdots \wedge \tau_n  = e_1 \wedge \cdots \wedge e_{n+1} $. Henceforth, $ \Der F(x)(\tau_1) \wedge \cdots \wedge \Der F(x)(\tau_n) \wedge (\Der F(x)^{-1})^\ast(u) =0 $ and, since $ \{\Der F(x)(\tau_i) : i = 1, \ldots n\} $ are linearly independent, we conclude that there exists $ c_1, \ldots , c_n \in \mathbf{R} $ such that 
$$ (\Der F(x)^{-1})^\ast(u) = \sum_{i=1}^n c_i\, \Der F(x)(\tau_i). $$
Applying $ \Der F(x)^{-1} $ to both sides  and taking the scalar product with $ u $, we get 
$$ \big[\Der F(x)^{-1} \circ (\Der F(x)^{-1})^\ast \big](u) \bullet u =0, $$
whence we infer that $ (\Der F(x)^{-1})^\ast(u)  =0 $, a contradiction.
\end{remark}

\begin{definition}
Suppose $  \Omega \subseteq \mathbf{R}^{n+1} $ is a $ W^{2,n} $-domain. We denote by $ N_\Omega $ the Legendrian cycle given by Theorem \ref{W2n domains}\eqref{W2n domains 4}.
\end{definition}

\begin{remark}\label{rmk invariance under diff}
The proof of Theorem \ref{W2n domains}\eqref{W2n domains 4} proves that \emph{if $ F : U \rightarrow V $ is a $ C^2 $-diffeomorphism between open subsets of $ \mathbf{R}^{n+1} $ and $  \Omega $ is a bounded $ W^{2,n}$-domain such that $ \overline{\Omega} \subseteq U $, then 
$$ (\Psi_F)_{\#}(N_\Omega) = N_{F(\Omega)}. $$}
	\end{remark}

\begin{definition}[$ r $-th elementary symmetric function]\label{symmetric function}
Suppose $ r \in \{1, \ldots , n\} $. The $ r $-th symmetric function $ \sigma_r : \mathbf{R}^n \rightarrow \mathbf{R} $ is  defined as 
$$ \sigma_r(t_1, \ldots , t_n) =  \frac{1}{{n \choose r}}\sum_{\lambda \in \Lambda_{n, r}} t_{\lambda(1)}\cdots t_{\lambda(r)}  $$
	where $ \Lambda_{n,r} $ is the set of all increasing functions from $ \{1, \ldots , r\} $ to $ \{1, \ldots , n\} $. We set 
	$$ \sigma_{0}(t_1, \ldots , t_n) = 1 \quad \textrm{for $ (t_1, \ldots , t_n) \in \mathbf{R}^n $.} $$
\end{definition}

\begin{definition}[$r$-th mean curvature function]\label{rth mean curvature function}
	Suppose $ \Omega \subseteq \mathbf{R}^{n+1} $ is a  $ W^{2,n} $-domain and $ r \in \{0, \ldots, n\} $. Then we define the \emph{$ r $-th mean curvature function of $\Omega$} as 
	$$ H_{\Omega,r}(z) = \sigma_r(\rchi_{\Omega, 1}(z), \ldots , \rchi_{\Omega, n}(z)) $$
	for $ \mathcal{H}^n $ a.e.\ $ z \in \partial \Omega $.
\end{definition}

\begin{lemma}\label{rmk curvature measures}
If $ \Omega \subseteq \mathbf{R}^{n+1} $ is a bounded $ W^{2,n} $-domain and $ \phi $ is a smooth $ \mathbf{R} $-valued function over $ \mathbf{R}^{n+1} \times \mathbf{S}^n $ then
		$$ [N_\Omega \restrict \varphi_{n-k}](\phi) = {n \choose k} \int_{\partial \Omega} H_{\Omega, k}(x)\, \phi(x, \nu_\Omega(x))\, d\mathcal{H}^n(x) \quad \textrm{for $ k = 0, \ldots , n $}. $$
\end{lemma}

\begin{proof}
We know by Theorem \ref{W2n domains}\eqref{W2n domains 4} that $ N_{\Omega} = (\mathcal{H}^n \restrict \nor(\Omega)) \wedge (\zeta_1\wedge \ldots \wedge \zeta_n) $. Noting that 
$$ J_n^{\nor(\Omega)}\pi_0(x,u) = \prod_{i=1}^n \frac{1}{\sqrt{1 + \kappa_{\Omega,i}(x,u)^2}} \quad \textrm{for $ \mathcal{H}^n $ a.e.\ $(x,u) \in \nor(\Omega) $}, $$
we employ Theorem \ref{W2n domains}\eqref{W2n domains 3} to compute
\begin{flalign*}
 [N_\Omega \restrict \varphi_{n-k}](\phi) = {n \choose k} \int_{\nor(\Omega)} J_n^{\nor(\Omega)}\pi_0(x,u)\,\phi(x,u)\,H_{\Omega, k}(x)\, d\mathcal{H}^n(x,u),
\end{flalign*}
whence we conclude using area formula in combination with Theorem \ref{W2n domains}\eqref{W2n domains 2} and Lemma \ref{W2n domains basic lemma}\eqref{W2n domains basic lemma 1}.
\end{proof}

\begin{definition}[$r $-th total curvature measure]\label{r-totCurvMeasOmega}
If $ \Omega \subseteq \mathbf{R}^{n+1} $ is a bounded $ W^{2,n} $-domain and $ r = 0 , \ldots, n $, we define 
$$ \mathcal{A}_r(\Omega) = \int_{\partial \Omega} H_{\Omega, r}\, d\mathcal{H}^n. $$
\end{definition}

Now we can quickly derive the following extension of Reilly's variational formulae (cf.\ \cite{Reilly1972}) to $W^{2,n}$-domain.

\begin{theorem}\label{Reilly variational formualae}
	Suppose $ \Omega \subseteq \mathbf{R}^{n+1} $ is a bounded $W^{2,n} $-domain and $ (F_t)_{t \in I} $ is a local variation of $ \mathbf{R}^{n+1} $ with initial velocity vector field $ V $. Then 
	$$ \frac{d}{dt}\mathcal{A}_{k-1}(F_t(\Omega)) \Big|_{t =0} = (n-k + 1) \int_{\partial \Omega} H_{\Omega, k}\, (\nu_\Omega \bullet V)\, d\mathcal{H}^n \quad \textrm{ for $ k = 1, \ldots , n $} $$ 
	and 
	\begin{equation}\label{Reilly last formula}
		\frac{d}{dt}\mathcal{A}_n(F_t(\Omega)) \Big|_{t =0} = 0.
	\end{equation} 
\end{theorem}

\begin{proof}
	Combining Remark \ref{rmk invariance under diff} and Lemma \ref{rmk curvature measures} we obtain 
	$$ \big[(\Psi_{F_t})_{\#}N_\Omega  \big](\varphi_{n-k+1}) = N_{F_t(\Omega)}(\varphi_{n-k+1}) = {n \choose k-1} \mathcal{A}_{k-1}(F_t(\Omega)) $$ 
	for $ k = 1 , \ldots, n+1 $. Hence we use Lemma \ref{Lemma Fu} and again Lemma \ref{rmk curvature measures} to compute
	$$  \frac{d}{dt}\big[(\Psi_{F_t})_{\#}N_\Omega  \big](\varphi_{n-k+1}) \Big|_{t=0} = k\,{n \choose k}\int_{\partial \Omega} (V(x) \bullet \nu_\Omega(x))\, H_{\Omega, k}(x)\, d\mathcal{H}^n(x) $$
	for $ k = 1, \ldots , n $ and
	$$  \frac{d}{dt}\big[(\Psi_{F_t})_{\#}N_\Omega  \big](\varphi_{0}) \Big|_{t=0} = 0. $$
\end{proof}

\begin{remark}
	 If $ \Omega $ is  a $ C^2 $-domain, then  \eqref{Reilly last formula} follows from the Gauss-Bonnet theorem.  The validity of the Gauss-Bonnet theorem for bounded $ W^{2,n} $-domains is an interesting open question, and \eqref{Reilly last formula} seems to point to a possible positive answer.
\end{remark}

The following integral formulae can be easily deduced from Theorem \ref{Reilly variational formualae} by  a standard procedure. For the $ C^2 $-regular domains these formulae are classic, see \cite{Hsiung}.

\begin{corollary}\label{Minkowski formulae}
	If $ \Omega \subseteq \mathbf{R}^{n+1} $ is a bounded  $ W^{2,n} $-domain and $ r \in \{1, \ldots, n\} $ then 
$$ \int_{\partial \Omega} H_{\Omega, r-1}(x)\, d\mathcal{H}^n(x) =  \int_{\partial \Omega} (x \bullet \nu_\Omega(x))\, H_{\Omega, r}(x)\, d\mathcal{H}^n(x). $$
\end{corollary}

\begin{proof}
	We consider the local variation $ F_t(x) = e^t\, x $ for $ (x,t) \in \mathbf{R}^n \times \mathbf{R} $ and we notice that
	$$ \Tan^n(\mathcal{H}^n \restrict \partial \Omega, x) = \Tan^n(\mathcal{H}^n \restrict F_t(\partial \Omega), F_t(x)) $$
	$$  \nu_{F_t(\Omega)}(F_t(x)) = \nu_\Omega(x) \quad \textrm{and} \quad \rchi_{F_t(\Omega),i}(F_t(x)) = e^{-t} \rchi_{\Omega,i}(x) $$
		for $ \mathcal{H}^n $ a.e.\ $ x \in \partial \Omega $ and $ i = 1, \ldots , n $.
		Henceforth, we compute by area formula
		\begin{flalign*}
		\mathcal{A}_{r-1}(F_t(\Omega)) & = \int_{\partial F_t(\Omega)} H_{F_t(\Omega), r-1}\, d\mathcal{H}^n \\
		& = e^{-(r-1)t}  \int_{F_t(\partial \Omega)} H_{\Omega, r-1}(F_t^{-1}(y))\, d\mathcal{H}^n(y)\\
		& = e^{(n-r+1)t}  \int_{\partial \Omega}  H_{\Omega, r-1}(x)\, d\mathcal{H}^n(x)
		\end{flalign*}
and we apply Theorem \ref{Reilly variational formualae}.
\end{proof}

\begin{corollary}\label{Minkowski fomrulae cor}
Suppose $ \Omega \subseteq \mathbf{R}^{n+1} $ is a bounded  $ W^{2,n} $-domain, $ k \in \{1, \ldots, n\} $ and 
\begin{equation}\label{k-1 mean convexity}
	 H_{\Omega, i}(z) \geq 0 \quad \textrm{for $ i = 1, \ldots , k-1 $ and for $ \mathcal{H}^n $ a.e.\ $ z \in \partial \Omega $.} 
\end{equation}

Then there exists $ P \subseteq \partial \Omega $ such that $ \mathcal{H}^n(P) > 0 $ and $ H_{\Omega, k}(z) \neq 0 $ for $ z \in P $.
\end{corollary}

\begin{proof}
	Suppose $ H_{\Omega,k}(z) =0 $ for $ \mathcal{H}^n $ a.e.\ $ z \in \partial \Omega $. Then we can employ Corollary \ref{Minkowski formulae} (with $ r = k $) and use \eqref{k-1 mean convexity} (for $ i = k-1 $) to infer that $ H_{\Omega, k-1}(z) =0 $ for $ \mathcal{H}^n $ a.e.\ $ z \in \partial \Omega $. Now we repeat this argument with $ r = k-1 $ and $ i = k-2 $ to infer that $ H_{\Omega, k-2}(z) =0 $ for $ \mathcal{H}^n $ a.e.\ $ z \in \partial \Omega $, and we continue until we obtain that $ H_{\Omega, 0}(z) =0 $ for $ \mathcal{H}^n $ a.e.\ $ z \in \partial \Omega $, which means $ \mathcal{H}^n(\partial \Omega) =0 $. Since the latter is clearly impossible, we have proved the assertion.
\end{proof}

\section{Sphere theorems for $W^{2,n}$-domains}

The results in the previous section in combination with the Heintze-Karcher inequality proved below can be used to generalize classical sphere theorems to $W^{2,n}$-domains.

\begin{theorem}[Heintze-Karcher inequality]\label{Heintze-Karcher}
	Suppose $ \Omega \subseteq \mathbf{R}^{n+1} $ is a bounded and connected  $ W^{2,n} $-domain such that $ H_{\Omega, 1}(z) \geq 0 $ for $ \mathcal{H}^n $ a.e.\ $ z \in \partial \Omega $. Then 
		$$ (n+1)\mathcal{L}^{n+1}(\Omega) \leq  \int_{\partial \Omega} \frac{1}{H_{\Omega, 1}(x)}\, d\mathcal{L}^n(x). $$ 
Moreover, if $ H_{\Omega, 1}(z) \geq  \frac{\mathcal{H}^n(\partial \Omega)}{(n+1)\mathcal{L}^{n+1}(\Omega)} $ for $ \mathcal{H}^n $ a.e.\ $ z \in \partial \Omega $ then $ \Omega $ is a round ball.
\end{theorem}

\begin{proof}
	 We define $ \Omega' = \mathbf{R}^{n+1} \setminus  \overline{\Omega} $ and notice that $ \Omega' $ is a $ W^{2,n} $-domain. Since $ \partial^v_+ \Omega' = \partial^v_+ \Omega $ and $ \nu_{\Omega'} = -\nu_\Omega $, it follows from Theorem \ref{W2n domains} that 
	$$ \mathcal{H}^n\big(\nor(\Omega') \setminus \{(z, -\nu_\Omega(z)): z \in \partial^v_+\Omega\}\big) =0, $$
	and \begin{equation*}
	-\rchi_{\Omega,i}(z)= \rchi_{\Omega',i}(z) = \kappa_{\Omega',i}(z, -\nu_\Omega(z)) \quad \textrm{for $ \mathcal{H}^n $ a.e.\ $ z \in \partial^v_+ \Omega $.}
	\end{equation*}   
Henceforth, 
	\begin{equation}\label{Heintze-Karcher eq1}
	 \sum_{i=1}^n \kappa_{\Omega',i}(z,u) = - n\, H_{\Omega, 1}(z)   \leq 0 \quad \textrm{for $ \mathcal{H}^n $ a.e.\ $(z,u) \in \nor(\Omega') $}
	\end{equation}  
	and we infer from Theorem \ref{HK general} and area formula \cite[3.2.20]{Fed69} that 
	\begin{flalign}\label{Heintze-Karcher eq}
	(n+1)\mathcal{L}^{n+1}(\Omega) & \leq \int_{\nor(\Omega')}J_n^{\nor(\Omega')}\pi_0(z,u)\,\frac{n}{|\sum_{i=1}^n \kappa_{\Omega',i}(z,u)|}\, d\mathcal{H}^n(z,u) \notag \\
	& = \int_{\partial \Omega}\frac{1}{H_{\Omega,1}(z)}\, d\mathcal{H}^n(z).
	\end{flalign}  
	
	We assume now that $  H_{\Omega, 1}(z) \geq  \frac{\mathcal{H}^n(\partial \Omega)}{(n+1)\mathcal{L}^{n+1}(\Omega)} $ for $ \mathcal{H}^n $ a.e.\ $ z \in \partial \Omega $. Then, we observe that 
	$$ \mathcal{H}^n\bigg( \bigg\{z \in \partial \Omega: H_{\Omega, 1}(z) \geq  (1+\epsilon)\frac{\mathcal{H}^n(\partial \Omega)}{(n+1)\mathcal{L}^{n+1}(\Omega)} \bigg\}\bigg) =0 \quad \textrm{for every $ \epsilon > 0 $,} $$
	otherwise we would obtain a contradiction with the inequality \eqref{Heintze-Karcher eq} (cf.\ proof of \cite[Corollary 5.16]{HugSantilli}). This implies that 
	$$ H_{\Omega, 1}(z) =  \frac{\mathcal{H}^n(\partial \Omega)}{(n+1)\mathcal{L}^{n+1}(\Omega)} \quad \textrm{for $ \mathcal{H}^n $ a.e.\ $ z \in \partial \Omega $}, $$
	whence we infer that \eqref{Heintze-Karcher eq} holds with equality. Recalling \eqref{Heintze-Karcher eq1} we deduce from Theorem \ref{HK general} that $ \Omega $ must be a round ball.
\end{proof}

\begin{theorem}\label{Alexandrov}
	If $ k \in \{1, \ldots , n\} $, $ \lambda \in \mathbf{R} $ and  $ \Omega \subseteq \mathbf{R}^{n+1} $ is a bounded and connected $ W^{2,n} $-domain such that
		\begin{equation}\label{Alexandrov hp1}
 H_{\Omega, i}(z) \geq 0 \quad \textrm{for $ i = 1, \ldots , k-1 $}		
		\end{equation} 
	and
\begin{equation}\label{Alexandrov hp}
H_{\Omega, k}(z) = \lambda  
\end{equation}  
for $ \mathcal{H}^n $ a.e.\ $ z \in \partial \Omega $, then $ \Omega $ is a round ball.
\end{theorem}

\begin{proof}
 Combining Theorem \ref{Minkowski formulae} and divergence theorem for sets of finite perimeter (it is clear by Lemma \ref{W2n domains basic lemma} that $ \Omega $ is a set of finite perimeter whose reduced boundary is $\mathcal{H}^n$ almost equal to the topological boundary) we obtain 
\begin{flalign}\label{Alexandrov eq1}
	\int_{\partial \Omega} H_{\Omega, k-1}\, d\mathcal{H}^n =  \lambda \int_{\partial \Omega} x \bullet \nu_\Omega(x)\, d\mathcal{H}^n(x) = \lambda  (n+1) \mathcal{L}^{n+1}(\Omega)
\end{flalign} 
and we infer that $ \lambda \geq 0 $. Hence we deduce from \cite[Lemma 2.2]{HugSantilli} and Corollary \ref{Minkowski fomrulae cor} that 
\begin{equation}\label{Alexandrov Newton inequality}
	H_{\Omega, 1}(z)\geq \ldots \geq H_{\Omega, k-1}(z)^{\frac{1}{k-1}} \geq  H_{\Omega, k}(z)^{\frac{1}{k}} =\lambda^{\frac{1}{k}} > 0
\end{equation}  
for $ \mathcal{H}^n $ a.e.\ $ z \in \partial \Omega $. By \eqref{Alexandrov Newton inequality}, 
\begin{equation*}
	\int_{\partial \Omega} H_{\Omega, k-1}(z)\, d\mathcal{H}^n(z) \geq \lambda^{\frac{k-1}{k}}\mathcal{H}^n(\partial \Omega)
	\end{equation*}
and combining with \eqref{Alexandrov eq1} we obtain 
$$ \lambda (n+1) \mathcal{L}^{n+1}(\Omega) \geq \lambda^{\frac{k-1}{k}}\mathcal{H}^n(\partial \Omega). $$
Since  $ \lambda > 0 $, we obtain from \eqref{Alexandrov Newton inequality} that
\begin{equation}\label{Alexandrov eq2}
H_{\Omega, 1}(z) \geq  \frac{\mathcal{H}^n(\partial \Omega)}{(n+1)\mathcal{L}^{n+1}(\Omega)} \quad \textrm{for $ \mathcal{H}^n $ a.e.\ $ z \in \partial \Omega $}
\end{equation}
and we conclude applying Lemma \ref{Heintze-Karcher}.
\end{proof}

\begin{remark}\label{Alexandrov rmk}
Hypothesis \eqref{Alexandrov hp1} in Theorem \ref{Alexandrov} can be equivalently replaced by the following assumption:
\begin{equation}\label{Alexandrov hp2}
		\frac{\partial \sigma_k}{\partial t_i}(\rchi_{\Omega, 1}(z), \ldots , \rchi_{\Omega, n}(z)) \geq 0 \quad \textrm{for $ i = 1, \ldots , n $ and for $ \mathcal{H}^n $ a.e.\ $ z \in \partial \Omega $.} 
\end{equation}
Assume \eqref{Alexandrov hp2} in  place of \eqref{Alexandrov eq1} in Theorem \ref{Alexandrov}. Then, first we use \cite[eq.\ (1.15)]{Salani} to infer that 
$$ H_{\Omega, k-1}(z) = \frac{1}{k} \sum_{i=1}^n \frac{\partial \sigma_k}{\partial t_i}(\rchi_{\Omega, 1}(z), \ldots , \rchi_{\Omega, n}(z)) \geq 0 $$
	for $ \mathcal{H}^n $ a.e.\ $ z \in \partial \Omega $. Then, as in \eqref{Alexandrov eq1}, we deduce that $ \lambda \geq 0 $. Finally, we employ \cite[Proposition 1.3.2]{Salani} to infer  \eqref{Alexandrov hp1}.
\end{remark}

\section{Nabelpunksatz for Sobolev graphs}

In this final section we extend the Nabelpunktsatz to graphs of twice weakly differentiable functions in terms of the approximate curvatures of their graphs. In particular, Theorem \ref{Nabelpunktsatz} provides a general version of the Nabelpunktsatz for $ W^{2,1} $-graphs. In view of well known examples of convex functions, this result is sharp; see Remark \ref{rmk primitive cantor}. In this section we use the symbols $ \der_i $ and $ \der^2_{ij} $ (respectively $ \Der_i $ and $ \Der^2_{ij} $) for the distributional partial derivatives of a Sobolev function (respectively the classical partial derivatives of a function) with respect to the standard base $ e_1, \ldots , e_n $ of $ \mathbf{R}^n $.

\begin{remark}\label{rmk umbilicality for C2}
Let $ U \subseteq \mathbf{R}^n $ be a connected open and $ f \in C^2(U) $. We define $ G = \{(x, f(x)) : x \in U\} $, and $ \nu : G \rightarrow \mathbf{S}^n \subseteq \mathbf{R}^{n+1} $ so that 
\begin{equation}\label{umbilical normal}
	\nu(\overline{f}(x)) = \frac{(-\nabla f(x), 1)}{\sqrt{1 + | \nabla f(x)|^2}} 
\end{equation} 
for every $ x \in U $.  Differentiating \eqref{umbilical normal} we get
$$ \Der \nu(\overline{f}(x))(v, \Der f(x)(v)) = \frac{(-\Der (\nabla f)(x)(v), 0)}{\sqrt{1 + | \nabla f(x)|^2}}  - \frac{\nabla f(x) \bullet \Der(\nabla f)(x)(v)}{1 + | \nabla f(x)|^2}\nu(\overline{f}(x)) $$
for every $ v \in \mathbf{R}^n $. We recall that $ G $ is umbilical if and only if there exists a function $ \lambda : G \rightarrow \mathbf{R} $ such that 
$$ \Der \nu(z) = \lambda(z) \Id_{\Tan(G,z)} \quad \forall z \in G. $$
Therefore, noting that  $ \Tan(G, \overline{f}(x)) = \{(v, \Der f(x)(v)): v \in \mathbf{R}^n\} $, we conclude that \emph{$ G $ is umbilical if and only if
	\begin{equation}\label{umbilical condition}
		\lambda(\overline{f}(x)) \big[e_i \bullet e_j + \Der_i f(x)\Der_j f(x) \big] = - \frac{\Der^2_{ij}f(x)}{\sqrt{1 + | \nabla f(x)|^2}}
	\end{equation}  
	for every $ x \in U $ and for every $ i, j = 1, \ldots , n $.}   It follows from \cite{SouamToubiana2006} that \emph{if $ U \subseteq \mathbf{R}^n $ is a connected open set, $ f \in C^2(U) $ and $ \lambda : G \rightarrow \mathbf{R} $ is a function such that \eqref{umbilical condition} holds for every $ x \in U $, then either $ \overline{f}(U) $ is contained in an $ n $-dimensional plane, or $ \overline{f}(U) $ is contained in an $ n $-dimensional sphere.}  
\end{remark} 

The first result of this section generalizes Remark \ref{rmk umbilicality for C2} to $ W^{2,1} $-functions.  Suppose $ U \subseteq \mathbf{R}^n $ is an open set, $ \nu \in \mathbf{S}^{n-1} $ and $ \pi_\nu $ is the orthogonal projection onto $ \nu^\perp $. Then we define 
$$ U_\nu = \pi_\nu[U] $$
and 
$$ U^\nu_y = \{t \in \mathbf{R} : y + t \nu \in U \} \subseteq \mathbf{R} \quad \textrm{for $ y \in U_\nu $.} $$
Notice that $ U_\nu $ is an open subset of $ \nu^\perp $ and $ U^\nu_y $ is an open subset of $ \mathbf{R} $ for every $ y \in U_\nu $. 

\begin{lemma}\label{umbilicality lemma}
	Suppose $ U \subseteq \mathbf{R}^n $ be an open set, $g \in W^{1,1}_\loc(U) $  and $ k \in \{1, \ldots , n\} $ such that 
	$$ \der_k g(x) =0 \quad \textrm{for $ \mathcal{L}^n $ a.e.\ $ x \in U $.} $$
	
	Then for $ \mathcal{L}^{n-1} $ a.e.\ $ y \in U_{e_k} $ the function mapping $ t \in U^{e_k}_y $ into $ g(y+te_k) $ is $ \mathcal{L}^1 $ almost equal to a constant function.
\end{lemma}

\begin{proof}
	It follows from \cite[Theorem 2.1.4]{Ziemerbook} that there exists a representative $ \tilde{g} $ of $ g $ such that the restriction of $ \tilde{g} $ on $ U^{e_k}_y $ is absolutely continuous and 
	$$ \der_k g(y + t e_k) = \frac{d}{dt}\tilde{g}(y + te_k) \quad \textrm{for $ \mathcal{L}^1 $ a.e.\ $ t \in U^{e_k}_y $} $$
	for $ \mathcal{L}^{n-1} $ a.e.\ $ y \in U_{e_k} $. 
	It follows from the hypothesis that 
	$$ \frac{d}{dt}\tilde{g}(y + te_k) =0   $$
	for $ \mathcal{L}^1 $ a.e.\ $ t \in U^{e_k}_y $ and for $ \mathcal{L}^{n-1} $ a.e.\ $ y \in U_{e_k} $, and we readily obtain the conclusion from the absolute continuity hypothesis of $ \tilde{g} $.
\end{proof}

We prove now the first result of this section.
\begin{theorem}\label{umbilicalty}
	Suppose $ U \subseteq \mathbf{R}^n $ is a connected open set, $ f \in W^{2,1}_\loc(U) $ and $ \mu : U \rightarrow \mathbf{R} $ is a function such that 
	\begin{equation}
		\mu(x)\big[e_i \bullet e_j + \der_i f(x)\der_j f(x) \big] = - \frac{\der^2_{ij}f(x)}{\sqrt{1 + | \nabla f(x)|^2}}
	\end{equation} 
	for $ \mathcal{L}^n $ a.e.\ $ x \in U $ and for every $ i,j = 1, \ldots , n $.
	
	Then, either $ f $ is $ \mathcal{L}^n $ almost equal to a linear function on $ U $,  or there exists a $ n $-dimensional sphere $ S $ in $ \mathbf{R}^{n+1} $ such that $ \overline{f}(x) \in S $ for $ \mathcal{L}^n $ a.e.\ $ x \in U $.
\end{theorem}

\begin{proof}
Recall the diffeomorphism $\psi $ from Remark  \ref{rmk map eta est} and define $\eta = \psi \circ \bm{\nabla} f $.
	By the classical chain rule for Sobolev maps (cf.\ \cite{GilbargTrudinger}), $ \eta \in W^{1,1}_\loc(U, \mathbf{R}^{n+1}) $ and 
	\begin{flalign*}
		\der\eta(x)(v) & = \big[\Der \psi(\bm{\nabla} f(x)) \circ \der (\bm{\nabla}f)(x)\big](v)\\
		& =  \frac{(-\der (\bm{\nabla} f)(x)(v), 0)}{\sqrt{1 + | \bm{\nabla}f (x)|^2}}  - \frac{\bm{\nabla} f(x) \bullet \der(\bm{\nabla}f)(x)(v)}{1 + | \bm{\nabla} f(x)|^2}\eta(x)
	\end{flalign*} 
	for $ \mathcal{L}^n $ a.e.\ $ x \in U $. In particular, noting that $ \eta(x) \bullet (e_j, \der_jf(x)) =0 $ for every $ j = 1, \ldots , n $ and for $ \mathcal{L}^n $ a.e.\ $ x \in U $, we employ the umbilicality condition to obtain
	\begin{flalign*}
		\der_i \eta(x) \bullet (e_j, \der_j f(x)) & =   - \frac{\der^2_{ij}f(x)}{\sqrt{1 + | \bm{\nabla} f(x)|^2}} \\
		& = \mu(x) (e_i, \der_i f(x)) \bullet (e_j, \der_j f(x))
	\end{flalign*}
	for $ \mathcal{L}^n $ a.e.\ $ x \in U $ and for every $ i,j = 1, \ldots , n $. Consequently, for every $ i = 1, \ldots , n $ and for $ \mathcal{L}^n $ a.e.\ $ x \in U $ there exists $ \lambda_i(x) \in \mathbf{R} $ such that 
	\begin{equation}\label{umbilicality eq 1}
		\der_i \eta(x) - \mu(x) (e_i, \der_i f(x)) = \lambda_i(x) \eta(x).
	\end{equation}  
	On the other hand, since $ \eta $ is a unit-length vector, we see (again from the chain rule for Sobolev maps) that $ \eta(x) \bullet  \der_i \eta(x) =0 $ for $ \mathcal{L}^n $ a.e.\ $ x \in U $ and for $ i = 1, \ldots, n $. Therefore, we infer from \eqref{umbilicality eq 1} that $ \lambda_i(x) =0 $ and 
	\begin{equation}\label{umbilicality eq 2}
		\der_i \eta(x) = \mu(x) (e_i, \der_i f(x)) = \mu(x) \der_i \overline{f}(x)
	\end{equation} 
	for $ \mathcal{L}^n $ a.e.\ $ x \in U $. For $ k = 1, \ldots , n $ let $ g_k \in W^{1,1}_\loc(U) $ be given by
	$$ g_k = - \frac{\der_k f}{\sqrt{1 + | \bm{\nabla}f|^2}}, $$
	and we notice from \eqref{umbilicality eq 2} that 
	\begin{equation}\label{umbilicality eq 3}
		\der_i g_j =0,  \qquad \textrm{whenever $ i , j \in \{1, \ldots , n\} $ and $ i \neq j $,}
	\end{equation}
	\begin{equation}\label{umbilicality eq 4}
		\der_i g_i = \mu,  \qquad \textrm{whenever $ i  \in \{1, \ldots , n\} $.}
	\end{equation}
	
	We fix now an open cube $ Q \subseteq U $ with sides parallel to the coordinate axes, $ \phi \in C^\infty_c(Q) $ and $ k = 1, \ldots , n $, and we prove that 
	\begin{equation}\label{umbilicality eq6}
		\int_Q \mu \, \Der_k \phi\, d\mathcal{L}^n =0. 
	\end{equation}
	Choose $ j \in \{1, \ldots , n\} $ with $ k \neq j $. Since by \eqref{umbilicality eq 3} we have that $ \der_k g_j =0 $, it follows from Lemma \ref{umbilicality lemma} that for $\mathcal{L}^{n-1}$ a.e.\ $ y \in U_{e_k} $ there exists $ v_j(y) \in \mathbf{R} $ such that 
	$$ g_j(y + t e_k) = v_j(y) \quad \textrm{for $ \mathcal{L}^1 $ a.e.\ $ t \in U^{e_k}_y $.} $$ 
	Now we use \eqref{umbilicality eq 4} to obtain
	\begin{flalign*}
		\int_Q \mu \, \Der_k \phi\, d\mathcal{L}^n & = \int_Q \der_j g_j\, \Der_k \phi \, d\mathcal{L}^n  \\
		& = - \int_Q g_j\, \Der_j \big(\Der_k \phi\big)\, d\mathcal{L}^n \\
		& = - \int_{Q_{e_k}} v_j(y)\int_{Q^y_{e_k}} \Der_k \big(\Der_j \phi\big)(y + te_k)\, d\mathcal{L}^1(t)\, d\mathcal{L}^{n-1}(y) =0,
	\end{flalign*}
	where the last equality follows from the fact that the function mapping $ t \in Q^{e_k}_y $ into $ \Der_j \phi(y + t e_k) $ has compact support in $ Q^{e_k}_y $.
	
	Since \eqref{umbilicality eq6} holds for every open cube $ Q $ with sides parallel to the coordinate axes and for every $ \phi \in C^\infty_c(Q) $, and since $ U $ is connected, we infer from \cite[4.1.4]{Fed69} that
	\begin{equation}\label{umbilicality eq 7}
		\textrm{$ \mu $ is $ \mathcal{L}^n $ almost equal to a constant function on $ U $.} 
	\end{equation}
	Since $ U $ is connected, we combine \eqref{umbilicality eq 7} and \eqref{umbilicality eq 2} to infer that there exists $ c \in \mathbf{R} $ and $ w \in \mathbf{R}^{n+1} $ such that 
	$$ \eta(x) - c \overline{f}(x) = w \quad \textrm{for $ \mathcal{L}^n $ a.e.\ $ x \in U $.} $$
	If $ c \neq 0 $ the last equation evidently implies that $ \overline{f}(x) \in \partial B^{n+1}(-w/c, 1/|c|) $ for $ \mathcal{L}^n $ a.e.\ $ x \in U $. If $ c =0 $, we have that $ w \bullet e_{n+1} = (1 + |\bm{\nabla} f|^2)^{-1/2} $ and
	$$ \der_i f(x) = - \frac{w \bullet e_i}{w \bullet e_{n+1}} \quad \textrm{for $ \mathcal{L}^n $ a.e.\ $ x \in U $ and $ i = 1, \ldots , n $.} $$
	This implies that $ f $ is $ \mathcal{L}^n $ almost equal to linear function on $ U $, since $ U $ is connected.
\end{proof}

\begin{definition}
	Suppose  $ X \subseteq \mathbf{R}^{n+1} $ is  $ \mathcal{H}^n $-measurable and  $ \mathcal{H}^n $-rectifiable of class $ 2 $. We say that $ X $ is \emph{approximate totally umbilical} if there exists a $ \mathcal{H}^n \restrict X $-measurable map $ \nu $ such that 
	$ \nu(x) \in \Nor^n(\mathcal{H}^n \restrict X,x) \cap \mathbf{S}^n $ and there exists a function $ \mu : X \rightarrow \mathbf{R} $ such that
	\begin{equation}\label{umbilicality}
		\ap \Der \nu(x)(\tau) = \mu(x) \tau \quad \textrm{for every $ \tau \in \Tan^n(\mathcal{H}^n \restrict X,x) $},
	\end{equation}  
	for $ \mathcal{H}^n $ a.e.\ $ x \in X $ (keep in mind Lemma \ref{lem approx diff unit normal}).
\end{definition}

\begin{definition}[Lusin (N) condition]\label{Lusin condition}
Suppose $ U \subseteq \mathbf{R}^n $ is open and $ g : U \rightarrow \mathbf{R}^k $ ($k \geq n $). We say that $ g $ satisfies the \emph{Lusin's condition (N)} if and only if $ \mathcal{H}^n(g(Z)) =0 $ for every $ Z \subseteq U $ with $ \mathcal{L}^n(Z) =0  $.
\end{definition}

We are now ready to prove the second result of this section.

\begin{theorem}\label{Nabelpunktsatz}
	Suppose $ U \subseteq \mathbf{R}^n $ is a bounded open set, $ f \in  W^{2,1}(U) $, $ \overline{f} $ satisfies the Lusin's condition (N) and $ G = \overline{f}(U) $. 
	
	Then $ G $ is $ \mathcal{H}^n $-rectifiable of class $ 2 $. Moreover, if $ G $ is approximate totally umbilical 
	then, up to  a $ \mathcal{H}^n $-negligible set, either $ G $ is a subset of a $ n $-dimensional plane or a subset of a $ n $-dimensional sphere.
\end{theorem}

\begin{proof}
	By \cite[Theorem 13]{CalderonZygmund} and \cite[2.10.19(4), 2.10.43]{Fed69} we can find countably many functions $ g_1, g_2, \ldots \in C^2(\mathbf{R}^n) $ such that $ \mathcal{L}^n(U \setminus \bigcup_{i=1}^\infty \{g_i = f\}) =0 $ and $ \Lip(g_i) < \infty $ for every $ i \geq 1 $.
	Henceforth, thanks to the Lusin's condition (N), we readily infer that $ G $ is $ \mathcal{H}^n $-rectifiable of class $ 2 $. We define $ D_i $ as the set of $ x \in \{g_i = f\} $ such that 
	$$ \Theta^n(\mathcal{L}^n \restrict U \setminus \{g_i = f\}, x) =0, \quad \Der g_i(x)  = \der f(x) $$
	and $\Tan^n(\mathcal{H}^n \restrict G, \overline{f}(x))  $ is a $ n $-dimensional plane. Since $ \Der g_i(x) = \ap \Der f(x) $ for every $ x \in D_i $, it follows from \cite[2.10.19(4), 3.2.19]{Fed69} and Lemma \ref{lem weakly vs approx diff Sobolev maps} that 
	$$ \mathcal{L}^n(\{g_i = f\} \setminus D_i) =0 \quad \textrm{for every $ i \geq 1 $.} $$
	Since $ \Tan^n(\mathcal{L}^n \restrict \{g_i = f\}, x) = \mathbf{R}^n $ for every $ x \in D_i $, and noting that $ \overline{g_i} : \mathbf{R}^n \rightarrow \overline{g_i}(\mathbf{R}^n) $ is a bi-lipschitz homeomorphism, we use \cite[Lemma B.2]{SantilliAnnali} to conclude
	\begin{flalign*}
		\der \overline{f}(x)[\mathbf{R}^n] = \Der \overline{g_i}(x)[\Tan^n(\mathcal{L}^n \restrict \{g_i = f\}, x)] \subseteq \Tan^n(\mathcal{H}^n \restrict G, \overline{f}(x))
	\end{flalign*}
	for every $ x \in D_i $. Since $ \der \overline{f}(x) $ is injective whenever it exists, we conclude that 
	$$ 	\der \overline{f}(x)[\mathbf{R}^n] = \Tan^n(\mathcal{H}^n \restrict G, \overline{f}(x)) $$
	and $$ \psi(\bm{\nabla} f(x)) \in \Nor^n(\mathcal{H}^n \restrict G, \overline{f}(x)) $$ for every $ x \in D_i $ and $ i \geq 1 $. Let $ D = \bigcup_{i=1}^\infty D_i $ and notice that $ \mathcal{H}^n(G \setminus \overline{f}(D)) =0 $ (again by Lusin's condition (N)). Let $ \nu $ be the $ \mathcal{H}^n \restrict G $-measurable map defined by 
	$$ \nu = \psi \circ \bm{\nabla} f \circ (\pi|G), $$
	where $ \pi : \mathbf{R}^n \times \mathbf{R} \rightarrow \mathbf{R}^n $ is the  orthogonal projection onto $ \mathbf{R}^n $. We observe that if $ z \in \overline{f}(D_i) $ and $ \Theta^n(\mathcal{H}^n \restrict G \setminus \overline{f}(D_i), z) = 0 $, then $ \nu  $ is $ \mathcal{H}^n \restrict G $-approximately differentiable at $ z $ (since $ \nu| \overline{f}(D_i) = (\psi \circ \nabla g_i \circ \pi) | \overline{f}(D_i) $) and
	\begin{flalign*}
		\ap \Der \nu(z)& = \Der(\psi \circ \nabla g_i \circ \pi)(z) \\
		& = \Der(\psi \circ \nabla g_i)(\pi(z)) \circ \big(\pi| \Tan^n(\mathcal{H}^n \restrict G, z)\big) \\
		& = \ap \Der(\psi \circ \bm{\nabla} f)(\pi(z)) \circ \big(\pi| \Tan^n(\mathcal{H}^n \restrict G, z)\big) \\
		&= \der(\psi \circ \bm{\nabla} f)(\pi(z)) \circ \big(\pi| \Tan^n(\mathcal{H}^n \restrict G, z)\big),
	\end{flalign*}
	whence we infer
	\begin{equation}\label{Nabelpunktsatz eq 1}
		\ap \Der \nu(z) \circ \der \overline{f}(\pi(z))  = \der (\psi \circ \bm{\nabla} f)(\pi(z)).
	\end{equation} 
	By \cite[2.10.19(4)]{Fed69} we conclude that \eqref{Nabelpunktsatz eq 1} is true for $ \mathcal{H}^n $ a.e.\ $ z \in G $. 
	
	If $ G $ is approximate totally umbilical then it is easy to see that the unit normal vector field $ \nu $ defined above fulfils the umbilicalilty condition in \eqref{umbilicality} with some function $ \mu $. Henceforth, 
	\begin{flalign*}
		\mu(\overline{f}(x))(e_i \bullet e_j + \der_if(x)\der_jf( x)) & = \big(\ap \Der \nu(\overline{f}(x)) \circ \Der \overline{f}(x)\big)(e_i) \bullet (e_j, \der_j f(x)) \\
		& = \der_i(\psi \circ \nabla f)(x) \bullet (e_j, \der_j f(x))\\
		& = - \frac{\der^2_{ij}f(x)}{\sqrt{1 + | \bm{\nabla} f( x)|^2}}
	\end{flalign*}
	for every $ i, j = 1, \ldots , n $ and for $ \mathcal{L}^n $ a.e.\ $ x \in U $. By Theorem \ref{umbilicalty} and the Lusin's condition (N) we deduce that, up to a $ \mathcal{H}^n $-negligible set, $ G $ is either a subset of a $ n $-dimensional plane, or a subset of a $ n $-dimensional sphere of $ \mathbf{R}^{n+1} $. 
\end{proof}

\begin{remark}\label{rmk Lusin condition (N) for f}
	If $ f \in W^{2,p}(U) $ with $ p > \frac{n}{2} $, then the Sobolev embedding theorem \cite[Theorem 7.26]{GilbargTrudinger} ensures that $ f \in W^{1,p^\ast}(U) $ with $ p^\ast > n $. Henceforth, $ \overline{f} $ satisfies the Lusin's condition (N) by \cite[Theorem 1.1]{MarcusMizel}.
\end{remark}

\begin{remark}\label{rmk primitive cantor}
	It is easy to find convex functions $ f \in C^{1,\alpha}(\mathbf{R}^n) $ such that the approximate principal curvatures of the graph are zero $ \mathcal{H}^n $ almost everywhere and the conclusion of Theorem \ref{Nabelpunktsatz} fails (Notice that the graph is $ \mathcal{H}^n $-rectifiable of class $ 2 $ and $ \overline{f} $ satisfies the Lusin's condition (N)). Indeed the gradient of these functions are continuous maps of bounded variation whose distributional derivative is not a function. An example of such functions is given by the primitive of the ternary Cantor function.
\end{remark}

\appendix
\section{Area of the proximal unit normal bundle}

\begin{lemma}\label{lem approx curvatures}
	If $ C \subseteq \mathbf{R}^{n+1} $ is a closed set and $ M \subseteq \mathbf{R}^{n+1} $ is a $ k $-dimensional submanifold of class $ 2 $ then there exists $ R \subseteq M \cap C $ such that:
	\begin{enumerate}
		\item \label{lem approx curvatures1} $ \nor(C) \cap (R \times \mathbf{S}^n) \subseteq \nor(M) $; 
		\item \label{lem approx curvatures2} $\mathcal{H}^k \big((M \cap C) \setminus R \big) =0.$
	\end{enumerate}
\end{lemma}

\begin{proof}
	See the proof of \cite[Lemma 6.1]{SantilliAnnali}. 
\end{proof}

Suppose $ C \subseteq \mathbf{R}^{n+1} $ is closed with $\mathcal{H}^n(C) < \infty $ and $ \Sigma = \pi_0(\nor(C)) $. It follows from \cite{MenneSantilli} that $ \Sigma $ is $ \mathcal{H}^n $-rectifiable of class $ 2 $. We fix a $ \mathcal{H}^n \restrict \Sigma $-measurable map $ \nu : \Sigma \rightarrow \mathbf{S}^n $ such that $ \nu(a) \in \Nor^n(\mathcal{H}^n \restrict \Sigma, a) $ for $ \mathcal{H}^n $ a.e.\ $ a \in \Sigma $ and we notice that it is $ \mathcal{H}^n \restrict \Sigma $-approximately differentiable at $ \mathcal{H}^n $ a.e.\ $ a \in \Sigma $ with a symmetric approximate differential $ \ap \Der \nu(a)  $ by Lemma \ref{lem approx diff unit normal}. We denote by 
$$ \rchi_{\Sigma,1}(a) \leq \ldots \leq \rchi_{\Sigma, n}(a) $$
the eigenvalues of $ \ap \Der \nu(a) $ and we define  (cf. Definition \ref{def principal curvatures})
$$ E = \{(x,u) \in \nor(C) : \kappa_{C,n}(x,u) < \infty\}. $$

\begin{lemma}\label{lem area lower bound gauss graphs}
	If $ A \subseteq \mathbf{R}^{n+1} $ is a Borel set, then
	$$ \mathcal{H}^n(E \cap (A \times \mathbf{S}^n)) \geq \int_{\Sigma \cap A} \prod_{\ell=1}^n\sqrt{1 + \rchi_{\Sigma, \ell}(x)^2}\, d\mathcal{H}^n(x). $$
\end{lemma}

\begin{proof}
Let $ \{\Sigma_i\}_{i \geq 1} $ be a sequence of $ C^2 $-hypersurfaces such that $ \mathcal{H}^{n}(\Sigma \setminus \bigcup_{i=1}^\infty \Sigma_i) =0 $. Employing Lemma \ref{lem approx curvatures} we can find a \emph{disjointed} sequence of Borel subsets $ \{R_i\}_{i \geq 1} $ of $ \Sigma_i \cap \Sigma  $ such that
$$ \mathcal{H}^n\big( \Sigma \setminus {\textstyle \bigcup_{i=1}^\infty R_i}\big) =0 \quad \textrm{and} \quad \nor(C) \cap (R_i \times \mathbf{S}^n) \subseteq \nor(\Sigma_i)  $$
	for every $ i \geq 1 $. It follows from \cite[2.10.19(4)]{Fed69} that 
	\begin{equation}\label{lem area lower bound gauss graphs eq}
	\Tan^n\big(\mathcal{H}^n \restrict (\nor(C) \cap (R_i \times \mathbf{S}^n)), (x,u)\big) = \Tan(\nor(\Sigma_i), (x,u)) 
	\end{equation}
	for $ \mathcal{H}^n $ a.e.\ $ (x,u) \in \nor(C) \cap (R_i \times \mathbf{S}^n) $ (recall that $ \nor(\Sigma_i) \cap (\Sigma_i \times \mathbf{S}^n) $ is a $ n $-dimensional submanifold of $ \mathbf{R}^{n+1} \times \mathbf{S}^n $ of class $ 1 $). Since $ \pi_0 \big(\Tan(\nor(\Sigma_i), (x,u))\big) = \Tan(\Sigma_i, x) $ is a $ n $-dimensional linear space for every $ (x,u) \in \nor(\Sigma_i) \cap (\Sigma_i \times \mathbf{S}^n) $, we  deduce from \eqref{lem area lower bound gauss graphs eq} and \cite[Lemma 3.9]{HugSantilli} that 
	$$ \kappa_{C,n}(a,u) < \infty \quad \textrm{for $ \mathcal{H}^n $ a.e.\ $ (a,u) \in \bigcup_{i =1}^n \big[\nor(C) \cap (R_i \times \mathbf{S}^n)\big]. $} $$
For each $ i \geq 1 $ let 
	$ \rchi_{\Sigma_i, 1} \leq \ldots \leq \rchi_{\Sigma_i, n} $
	be the principal curvatures of $ \Sigma_i $, and we notice that 
	$$  J_n^{\nor(\Sigma_i)} \pi_0(x,u) = \prod_{\ell=1}^n \frac{1}{\sqrt{1 + \rchi_{\Sigma_i, \ell}(x)^2}} $$
	for $  (x,u) \in \nor(\Sigma_i) \cap (\Sigma_i \times \mathbf{S}^n) $.
	Applying area formula we can estimate
	\begin{flalign*}
			\mathcal{H}^n(E \cap (A \times & \mathbf{S}^n)) \\    &\geq \sum_{i=1}^\infty \mathcal{H}^n(\nor(C) \cap ((A \cap R_i) \times \mathbf{S}^n))\\
			& = \sum_{i=1}^\infty \int_{\nor(C) \cap ((A \cap R_i) \times \mathbf{S}^n)} J_n^{\nor(\Sigma_i)} \pi_0(x,u)\, \prod_{\ell=1}^n \sqrt{1 + \rchi_{\Sigma_i, \ell}(x)^2}\, d\mathcal{H}^n(x,u)\\
			& \geq \sum_{i=1}^\infty \int_{A \cap R_i}\prod_{\ell=1}^n \sqrt{1 + \rchi_{\Sigma_i, \ell}(x)^2}\, d\mathcal{H}^n(x).
	\end{flalign*}
From the proof of Lemma \ref{lem approx diff unit normal} we obtain that if $ i \geq 1 $  then
$$ \prod_{\ell=1}^n \sqrt{1 + \rchi_{\Sigma_i, \ell}(x)^2} = \prod_{\ell=1}^n \sqrt{1 + \rchi_{\Sigma, \ell}(x)^2} \quad \textrm{for $ \mathcal{H}^n $ a.e.\ $ x \in \Sigma \cap \Sigma_i $.}$$
 Henceforth, we conclude 
$$ 	\mathcal{H}^n(E \cap (A \times \mathbf{S}^n)) \geq  \int_{A \cap \Sigma}\prod_{\ell=1}^n \sqrt{1 + \rchi_{\Sigma, \ell}(x)^2}\, d\mathcal{H}^n(x). $$
\end{proof}

\begin{lemma}\label{lem Brakke example}
	There exist a smooth $ 2 $-dimensional submanifold $ M \subseteq \mathbf{R}^3 $ with bounded mean curvature such that $ \mathcal{H}^2 \restrict \overline{M} $ is a Radon measure over $ \mathbf{R}^3 $, and a set $ P \subseteq \mathbf{R}^3 \times \mathbf{S}^2 $  such that $ \mathcal{H}^2(P) > 0 $ and
	$$  \mathcal{H}^{2}\big(\nor(\overline{M}) \cap \{(x,u) \in \mathbf{R}^{n+1} \times \mathbf{S}^n : | x-b| < r,\; | u - \upsilon| < r\}\big) = \infty  $$
	for every $(b, \upsilon) \in P $ and for every $ r > 0 $.
\end{lemma}

\begin{proof}
By \cite[Example 10.8 and Remark 10.10]{KolasinskiMenne} (see also \cite[6.1]{Brakke}) there exists a smooth $ 2 $-dimensional submanifold $ M \subseteq \mathbf{R}^3 $ such that $ \mathcal{H}^2 \restrict \overline{M} $ is a Radon measure, and if $ \rchi_1 \leq \rchi_2 $ are the principal curvatures of $ M $ with respect to a unit-normal vector field $ \eta : M \rightarrow \mathbf{S}^n $ then 
$$ | \rchi_1(a) + \rchi_2(a) | \leq 1 \quad \textrm{for every $ a \in M $} $$ 
and there exists a Borel set $ B \subseteq \overline{M} \setminus M $ such that 
$$ \int_{M \cap B(b,r)} \big(\rchi_1^2 + \rchi_2^2\big)^{q/2}\, d\mathcal{H}^2 = \infty \quad \textrm{for every $ 1 < q < \infty $,  $ b \in B $ and $ r > 0 $.} $$
Henceforth, for every $ b \in B $ and $ r > 0 $, 
$$ +\infty = \int_{M \cap B(b,r)} \big(\rchi_1^2 + \rchi_2^2\big)\, d\mathcal{H}^2 \leq \mathcal{H}^2(M \cap B(b,r)) - 2 \int_{M \cap B(b,r)}\rchi_1\, \rchi_2\, d\mathcal{H}^2 $$
and 
$$ \int_{M \cap B(b,r)}\rchi_1\, \rchi_2\, d\mathcal{H}^2 = - \infty. $$
By Lemma \ref{lem area lower bound gauss graphs} we infer that 
$$ \mathcal{H}^2 \big(\nor(\overline{M}) \cap (B(b,r) \times \mathbf{S}^n)\big) \geq \int_{M\cap B(b,r)} \prod_{\ell=1}^2 (1 + \rchi_\ell(x)^2)^{1/2}\, d\mathcal{H}^n(x) = \infty $$
for every $ r > 0 $ and $ b \in B $.

Now the compactness of $ \mathbf{S}^n $ guarantees that for each $ b \in B $ there exists $ \upsilon(b) \in \mathbf{S}^n $ such that 
$$ \mathcal{H}^{2}\big(\nor(\overline{M}) \cap \{(x,u) \in \mathbf{R}^{n+1} \times \mathbf{S}^n : | x-b| < r,\; | u - \upsilon(b)| < r\}\big) = \infty $$
for each $ r > 0 $. Henceforth, setting $ P = \{(b, \upsilon(b)) : b \in B\} $ we have that $ \mathcal{H}^2(P) > 0 $ and the proof is complete.
\end{proof}

\begin{remark}\label{rmk Brakke example}
	In particular, $ \nor(\overline{M}) $ is a Legendrian rectifiable set by Lemma \ref{lem: Santilli20}, but it cannot be the carrier a integer-multiplicity rectifiable $ n $-current of $ \mathbf{R}^3 \times \mathbf{S}^2 $.
\end{remark}

\end{document}